\documentclass{article}

\usepackage{amsmath}
\usepackage{amssymb}
\usepackage{mathtools}
\usepackage{bbm}

\usepackage{a4wide}
\usepackage{setspace}
\usepackage{enumitem}
\usepackage{framed}
\usepackage{verbatim}
\usepackage{csquotes}
\usepackage[T1]{fontenc}
\usepackage[bottom]{footmisc}

\usepackage{amsthm}
\usepackage{thmtools}

\usepackage{floatrow}
\usepackage[hidelinks]{hyperref}
\usepackage[all]{hypcap}

\setlist[itemize]{leftmargin=*} 
\setlist[enumerate]{leftmargin=*}

\usepackage{mathdots}
\usepackage{xcolor}
\usepackage{diagbox}
\usepackage{colortbl}

\usepackage{graphicx}
\usepackage{subcaption}

\theoremstyle{plain}

\newtheorem{theorem}{Theorem}[section]
\newtheorem{claim}[theorem]{Claim}
\newtheorem{proposition}[theorem]{Proposition}
\newtheorem{lemma}[theorem]{Lemma}
\newtheorem{corollary}[theorem]{Corollary}
\newtheorem{conjecture}[theorem]{Conjecture}
\newtheorem{observation}[theorem]{Observation}

\newtheorem{problem}[theorem]{Problem}
\theoremstyle{definition}

\newtheorem{defn}[theorem]{Definition}
\newtheorem*{defn*}{Definition}
\newtheorem*{claim*}{Claim}
\newtheorem*{lem*}{Lemma}

\newtheorem*{rem}{Remark}

\expandafter\def\expandafter\normalsize\expandafter{%
    \normalsize
    \setlength\abovedisplayskip{4pt}
    \setlength\belowdisplayskip{4pt}
    \setlength\abovedisplayshortskip{4pt}
    \setlength\belowdisplayshortskip{4pt}
}

\usepackage[square,sort,comma,numbers]{natbib}
\def\oliver#1{}
\def\domagoj#1{}

\let\oliver=\oliverOpt 
\let\domagoj=\domagojOpt 

\newcommand{\calF}{\mathcal{F}}

\newcommand{\calG}{\mathcal{G}}

\def\eps {\varepsilon}
\DeclareMathOperator{\E}{\mathbb{E}}

\newcommand{\Bin}{\mathrm{Bin}}

\newcommand{\cF}{\mathcal{F}}

\newcommand{\cP}{\mathcal{P}}

\newcommand{\cW}{\mathcal{W}}

\newcommand{\Vint}{V_{\mathrm{int}}}

\newcommand{\norm}[1] {
    \left \| #1 \right \|
}
\newcommand{\TV}{\mathrm{TV}}
\newcommand{\Cay}{\mathrm{Cay}}

\newcommand{\Abs}{\mathrm{Abs}}

\renewcommand{\Pr}{\mathbb{P}}
\renewcommand{\E}{\mathbb{E}}

\newcommand{\ceil}[1]{
    \left \lceil #1 \right \rceil
}
\newcommand{\floor}[1]{
    \left \lfloor #1 \right \rfloor
}

\newcommand{\plabel}[1]{\phantomsection\label{#1}}

\newcommand{\explet}{\gamma} 

\title{Hamiltonicity of regular sublinear expanders}
\author{Domagoj Brada\v{c}\thanks{Institute of Mathematics, EPFL, Lausanne, Switzerland. Email: \textbf{domagoj.bradac@epfl.ch}} \and Oliver Janzer\thanks{Institute of Mathematics, EPFL, Lausanne, Switzerland. Email: \textbf{oliver.janzer@epfl.ch}}}
\date{}
\begin{document}

\maketitle
\begin{abstract}
    We say that a $d$-regular graph is a $\gamma$-expander if for every not too large set of vertices $S$, there are at least $\gamma d |S|$ edges leaving $S$, and we say that a graph $G$ is $\gamma$-far from bipartite if at least $\gamma e(G)$ edges need to be removed to make it bipartite. We prove that there exists an absolute constant $K$ such that any $n$-vertex $d$-regular $\gamma$-expander with $d \ge (\gamma^{-1} \log n)^K$ is Hamiltonian, provided that it is bipartite or $\gamma$-far from bipartite. As applications, we obtain highly robust versions of recent important results on the Hamiltonicity of Cayley graphs and Kneser graphs. As part of our proof, we prove a random connecting lemma for sublinear expanders which might be of independent interest.
\end{abstract}

\section{Introduction}
A Hamilton cycle in a graph is a cycle that visits every vertex exactly once, and a graph is called \emph{Hamiltonian} if it contains a Hamilton cycle. The algorithmic task of determining whether a given graph is Hamiltonian is one of the 21 problems shown to be NP-complete in a seminal paper of Karp~\cite{karp}. Thus, a simple characterization of Hamiltonian graphs is unlikely to exist, motivating the search for sufficient conditions that imply Hamiltonicity.

The most classical such result, predating Karp's work, is the famous theorem of Dirac~\cite{dirac} stating that any $n$-vertex graph with minimum degree at least $n/2$ is Hamiltonian. Dirac's theorem has been generalized in two different ways by Ore~\cite{ore} and Chv\'{a}tal~\cite{chvatal}. Other notable sufficient conditions for Hamiltonicity were proved by Chv\'{a}tal and Erd\H{o}s~\cite{chvatal-erdos}, Jackson~\cite{jackson} and Nash-Williams~\cite{nash-williams}.

The aforementioned results only apply to fairly dense graphs. Indeed, the conditions in the above theorems require the graph to have linear average degree, except the Chv\'{a}tal-Erd\H{o}s theorem, the condition in which implies that the average degree is at least the square root of the number of vertices. Thus, it is natural to seek conditions that apply to sparser graphs as well. Although of a slightly different flavour, the first result in this direction is a famous theorem of P\'{o}sa~\cite{posa} who showed that an $n$-vertex Erd\H{o}s-R\'{e}nyi random graph $\calG(n, p)$ with edge probability $p = C \log n / n$ is Hamiltonian with high probability\footnote{We say that an event holds with high probability, if this probability tends to $1$ as $n$ tends to infinity.}, for a sufficiently large constant $C$. This was further refined by Korshunov~\cite{korshunov}, after which Bollob\'{a}s~\cite{bollobas} and, independently, Koml\'{o}s and Szemer\'{e}di~\cite{komlos-szemeredi-hamiltonicity}, showed that for $p = (\log n + \log \log n + \omega(1)) / n$, the random graph $\calG(n, p)$ is Hamiltonian with high probability. This result is tight, since for smaller values of $p$, with high probability there is a vertex of degree at most one. Due to results of Cooper, Frieze and Reed~\cite{cooper-frieze-reed} and Krivelevich, Sudakov, Vu and Wormald~\cite{krivelevich-sudakov-vu-wormald}, it is also known that for any $d \in [3, n-1]$, a random $d$-regular $n$-vertex graph is Hamiltonian with high probability.

Given the above results, it is natural to seek pseudorandomness conditions that imply Hamiltonicity. The preeminent way to quantify the pseudorandomness of a graph is via its spectrum, that is, the spectrum of its adjacency matrix. For the purposes of this introduction, we shall restrict our attention to $d$-regular graphs. For an $n$-vertex $d$-regular graph $G$, we denote its eigenvalues as $\lambda_1(G), \dots ,\lambda_n(G)$, where $d = \lambda_1(G) \ge \lambda_2(G) \ge \dots \ge \lambda_n(G)$. An $(n, d, \lambda)$-graph is a $d$-regular $n$-vertex graph such that $\max\{ \lambda_2(G), -\lambda_n(G)\} = \lambda$. It is known that the smaller $\lambda$ is compared to $d$, the better pseudorandomness properties an $(n, d, \lambda)$-graph has. We refer to the excellent survey of Krivelevich and Sudakov~\cite{krivelevich-sudakov-survey} for many results on pseudorandom graphs. Coming back to our topic, Krivelevich and Sudakov~\cite{krivelevich-sudakov} conjectured that any $(n, d, \lambda)$-graph with $d / \lambda \ge C$, for some large constant $C$, is Hamiltonian. Weaker versions of the conjecture requiring $d / \lambda \ge (\log n)^{1 + o(1)}$ and $d / \lambda \geq C(\log n)^{1/3},$ were proved by Krivelevich and Sudakov~\cite{krivelevich-sudakov} and Glock, Munh\'{a} Correia and Sudakov~\cite{glock-correia-sudakov}, respectively, before the conjecture was resolved in a recent breakthrough by Dragani\'{c}, Montgomery, Munh\'{a} Correia, Pokrovskiy and Sudakov~\cite{draganic-montgomery}. 

\begin{theorem}[Dragani\'{c}, Montgomery, Munh\'{a} Correia, Pokrovskiy and Sudakov~\cite{draganic-montgomery}] \plabel{thm:hamiltonicity-of-expanders}
    There is an absolute constant $C$ such that any $(n, d, \lambda)$-graph with $d / \lambda \ge C$ is Hamiltonian.    
\end{theorem}

In fact, they proved that certain expansion properties of $(n, d, \lambda)$-graphs are already enough to imply Hamiltonicity, settling a conjecture of Brandt, Broersma, Diestel, and Kriesell~\cite{brandt-broersma}. We also remark that simultaneously, Ferber, Han, Mao and Vershynin~\cite{ferber-han} proved a weaker version of Theorem~\ref{thm:hamiltonicity-of-expanders} with the added assumption that $d \ge \log^6 n$.

The above result concerns a strong version of so-called \emph{linear} expanders -- graphs in which any set $U$ of at most half the vertices has a neighbourhood of size at least $\alpha |U|$, where $\alpha > 0$ is a constant. If one allows $\alpha$ to be smaller than a constant, but rather some function of $|U|$ and the size of the graph, one obtains a broad class of definitions of what are now called \emph{sublinear expanders}. These were first introduced by Koml\'{o}s and Szemer\'{e}di~\cite{komlos-szemeredi-subdivisions-1, komlos-szemeredi-subdivisions-2} in their work on finding clique subdivisions. Expansion is often endowed with a ``robustness'' property, which, informally speaking, means that expansion is preserved even after removing a sparse subset of edges. This may be viewed as a combination of edge and vertex expansion. Robust sublinear expanders were independently introduced by Haslegrave, Kim and Liu~\cite{haslegrave-kim-liu} and Sudakov and Tomon~\cite{sudakov-tomon} and have seen many applications. We shall work with the following simple form of edge expansion.
\begin{defn} \plabel{def:rho-exp}
    An $n$-vertex graph $G$ is a $\explet$-expander if for every $S \subseteq V(G)$ with $1 \le |S| \le \frac{2}{3} n,$ we have $e_G(S, V(G) \setminus S) \ge \explet \bar{d}(G) |S|,$ where $\bar{d}(G)$ is the average degree of $G$.
\end{defn}
In the above, $\explet$ is not necessarily a constant and a prototypical example one should keep in mind is $\explet = 1 / (\log n)^2$. As we shall discuss below, we will work with graphs which are very close to being regular and for such graphs, Definition~\ref{def:rho-exp} is equivalent, up to small changes in the parameters, to several notions of robust sublinear expanders used in the literature.

While linear expanders arise in many natural contexts, they are not ubiquitous. On the other hand, any graph contains a sublinear expander with nearly the same average degree, and this fact has led to many applications as it often allows one to replace an arbitrary graph with a sublinear expander, see e.g.~\cite{liu-montgomery, fernandez-gil, haslegrave-kim-liu, kim-liu-sharifzadeh-staden, liu-montgomery-maders-conjecture}. As a concrete example, it is known (e.g.~\cite{chakraborti2025edge}) that any $n$-vertex graph with average degree $d \ge (\log n)^2$ contains a $\explet$-expander for $\explet = 1 / (\log n)^2$ with average degree close to $d$. More refined and stronger statements are known and needed for certain applications, e.g.~\cite{alon-bucic-rainbow-cycles}. While finding a single sublinear expander in an arbitrary graph is sufficient for many extremal problems, for covering or decomposition problems, one would like to efficiently cover or decompose a graph into sublinear expanders. This approach was first utilised by Buci\'{c} and Montgomery~\cite{bucic2022towards} who, in their work on the Erd\H{o}s-Gallai conjecture, showed that one can efficiently decompose the edges of any graph into sublinear expanders. Given their many uses, results on sublinear expanders and tools to work with them are likely to have further applications. For much more details on different types of sublinear expanders and their applications, we refer to the recent survey of Letzter~\cite{letzter-survey}.

Closer to the topic at hand, Letzter, Methuku and Sudakov~\cite{letzter-methuku-sudakov-nearly-ham-cycles}, proved that nearly regular sublinear expanders of at least polylogarithmic average degree contain nearly Hamilton cycles, i.e. cycles covering all but an $o(1)$ proportion of the vertices, and presented several applications of this result.

In this paper, we prove that any regular sublinear expander with sufficiently large degree is Hamiltonian, unless it is close to being bipartite, but not bipartite. We start with our main theorem for bipartite sublinear expanders.

\begin{theorem} \plabel{thm:main-bipartite}
    Let $G$ be a bipartite $n$-vertex $d$-regular $\explet$-expander and assume that $n$ is sufficiently large. If $d > (\explet^{-1} \log n)^{10^8}$, then $G$ is Hamiltonian.
\end{theorem}

The condition in Theorem~\ref{thm:main-bipartite} can also be phrased in terms of the gap between the first and second eigenvalue of $G$. By Cheeger's inequality for graphs, due to Alon and Milman~\cite{alon-milman}, a $d$-regular graph $G$ is a $\explet$-expander with $\explet = (d - \lambda_2(G)) / (3d).$ Thus, we have the following corollary.

\begin{corollary} \plabel{cor:bipartite-spectral}
    Let $G$ be a bipartite $n$-vertex $d$-regular graph and assume that $n$ is sufficiently large. If $\lambda_2(G) = (1 - \delta)d$ and $d > (3 \delta^{-1} \log n)^{10^8}$, then $G$ is Hamiltonian.
\end{corollary}

We move on to our result for far-from-bipartite graphs. For $\eps \ge 0,$ we say that a graph $G$ is $\eps$-far from bipartite if $e_G(S, V(G) \setminus S) \le (1 - \eps) e(G)$ for every $S \subseteq V(G)$. Otherwise, we say that $G$ is $\eps$-close to bipartite.

\begin{theorem} \plabel{thm:main-far-from}
    Let $G$ be an $n$-vertex $d$-regular $\explet$-expander and assume that $n$ is sufficiently large. If $G$ is $\eps$-far from bipartite and $d \ge (\explet^{-1} \eps^{-1} \log n)^{10^8}$, then $G$ is Hamiltonian.
\end{theorem}

Theorem~\ref{thm:main-far-from} can also be phrased in terms of the eigenvalues of $G$. It is easy to show that every $d$-regular graph $G$ is $\eps$-far from bipartite for $\eps=(d+\lambda_n(G))/(2d)$. We thus have the following corollary.

\begin{corollary} \plabel{cor:far-from-bipartite-spectral}
    Let $G$ be an $(n,d,\lambda)$-graph and assume that $n$ is sufficiently large. If $
    \lambda = (1 - \delta) d$ and $d \ge (\delta^{-1} \log n)^{10^9}$, then $G$ is Hamiltonian.
\end{corollary}

Comparing with Theorem~\ref{thm:hamiltonicity-of-expanders}, Corollary~\ref{cor:far-from-bipartite-spectral} applies for graphs with $\lambda / d$ arbitrarily close to $1$, provided that $d$ is sufficiently large compared to $1 - \lambda / d$. On the other hand, our result requires $d$ to be at least $(\log n)^K$ for some constant $K$, so if $\lambda / d$ is very small, Theorem~\ref{thm:hamiltonicity-of-expanders} is stronger as it has no additional assumptions on $d$. We also point out that Corollary \ref{cor:far-from-bipartite-spectral} is stronger than the main results of \cite{krivelevich-sudakov,glock-correia-sudakov,ferber-han} as those papers all require $d$ to be at least polylogarithmic (though with a smaller power) and $\lambda/d$ to be small, whereas we allow $\lambda/d$ to be close to $1$.

It is also worth mentioning that the condition in (the strongest version of) Theorem~\ref{thm:hamiltonicity-of-expanders} from~\cite{draganic-montgomery} implies that the given graph is far from bipartite. To our knowledge, Theorem~\ref{thm:main-bipartite} is the first Hamiltonicity result for bipartite expanders (linear or sublinear) with degree $n^{o(1)}$.

Slightly unbalanced complete bipartite graphs show that the regularity condition in Theorem~\ref{thm:main-bipartite} cannot be ommited. Moreover, it is not hard to construct a balanced bipartite or a far-from-bipartite non-Hamiltonian $\explet$-expander with all degrees between $d$ and $d(1 + 2 \explet)$, for most reasonable values of $d$ and $\explet$. Thus, some level of regularity is needed to guarantee Hamiltonicity. Robust versions of our main theorems, Theorem~\ref{thm:robust-bipartite} and Theorem~\ref{thm:robust-far-from-bipartite}, imply Hamiltonicity of nearly regular sublinear expanders, provided that the level of regularity is much stronger than the expansion parameter. In the bipartite case, we additionally require that the two parts have equal size which is clearly necessary and not implied by near-regularity and expansion.

\subsection{Applications}
We present several applications of our main results and their robust versions. For moderately dense regular graphs, we are able to omit the condition of being bipartite or far from bipartite.

\begin{restatable}{theorem}{thmregularrestate} \plabel{thm:dense-regular}
    There is an absolute constant $\eta > 0$ such that every $n$-vertex $d$-regular $\explet$-expander with $d \ge n^{1-\eta}$ and $\explet = n^{-\eta}$ is Hamiltonian.
\end{restatable}

An old conjecture of Rapaport-Strasser~\cite{rapaport-strasser} states that any connected Cayley graph contains a Hamilton cycle. This conjecture is often referred to as Lov\'{a}sz' conjecture who asked whether every vertex transitive graph contains a Hamilton cycle, apart from five known examples~\cite{lovasz-conjecture}. Chen and Quimpo~\cite{chen-quimpo} proved that any connected Cayley graph on at least three vertices over an abelian group is Hamiltonian. The conjecture of Lov\'{a}sz has been confirmed for graphs with linear degree by Christofides, Hladk\'{y} and M\'{a}th\'{e}~\cite{christofides} and very recently, Bedert, Dragani\'{c}, M\"{u}yesser and Pavez-Sign\'{e}~\cite{bedertdraganic} have shown that even polynomially sparse Cayley graphs are Hamiltonian.

\begin{theorem}[Bedert, Dragani\'{c}, M\"{u}yesser and Pavez-Sign\'{e}~\cite{bedertdraganic}] \plabel{thm:lovasz}
    There exists a positive constant $\eta > 0$ such that any connected $n$-vertex Cayley graph with degree $d \ge n^{1-\eta}$ is Hamiltonian.
\end{theorem}

Using the robust versions of our main results as well as a ``weak arithmetic regularity lemma'' due to Bedert, Buci\'{c}, Kravitz, Montgomery and M\"{u}yesser~\cite{bedert2025graham} and some ideas from~\cite{christofides}, we are able to prove a robust version of this result as follows.

For a graph $G$ and a real number $p \in [0, 1],$ let $G_p$ denote the random graph obtained from $G$ by keeping each edge independently with probability $p$. We prove the following statement.

\begin{restatable}{theorem}{thmlovaszpercolatedrestate}
    \plabel{thm:lovasz-percolated}
    There exists an absolute constant $C$ such that the following holds. Let $G$ be a connected $n$-vertex Cayley graph with degree $d$ and let $p \in [0, 1]$ be such that $pd \ge \left( \frac{n}{d} \log n \right)^C$. Then $G_p$ is Hamiltonian with high probability.
\end{restatable}

Note that the statement with $p = 1$ recovers Theorem~\ref{thm:lovasz}, though with a worse constant than in~\cite{bedertdraganic}. On the other end of the spectrum, for dense connected Cayley graphs, i.e. with degree $d \ge \eps n$ for a constant $\eps > 0$, Theorem~\ref{thm:lovasz-percolated} implies they remain Hamiltonian after edge-percolation with probability $\frac{(\log n)^{C_2}}{n}$, for some constant $C_2$. This may be viewed as a robust version of the result of Christofides, Hladk\'{y} and M\'{a}th\'{e}~\cite{christofides} for Cayley graphs.

Corollaries~\ref{cor:bipartite-spectral}~and~\ref{cor:far-from-bipartite-spectral} apply to many natural classes of graphs. For the purpose of illustration, let us focus on a particularly notable instance concerning Kneser graphs. For $n > 2k,$ the Kneser graph, denoted by $K(n, k),$ is the graph whose vertices are all $k$-subsets of an $n$-element ground set, where two vertices are adjacent if the corresponding $k$-sets are disjoint. If $n > 2k,$ the Kneser graph is a connected vertex transitive graph, thus according to the aforementioned conjecture of Lov\'{a}sz, $K(n, k)$ should be Hamiltonian unless $(n, k) = (5, 2),$ in which case it is the Petersen graph, one of the five known vertex-transitive non-Hamiltonian graphs. For large enough values of $n$, until very recently, the Hamiltonicity of $K(n, k)$ was known only for $n \ge (2.618 \dots + o(1)) k$ by a result of Chen~\cite{chen} as well as if $n = 2k + 2^{a}$ for some integer $a \ge 0$. The sparsest case $a = 0$ was proved by M\"{u}tze, Nummenpalo and Walczak~\cite{mutze-nummenpalo-walczak} settling an old conjeture of Meredith and Lloyd~\cite{meredith-lloyd} and Biggs~\cite{biggs} from the 1970s, while the result for $a \ge 1$ follows by an inductive argument of Johnson~\cite{johnson}. Very recently, Hamiltonicity of all Kneser graphs was proved in a breakthrough work by Merino, M\"{u}tze and Namrata~\cite{merino-mutze-namrata}.

\begin{theorem}[Merino, M\"{u}tze and Namrata~\cite{merino-mutze-namrata}]
    For all $k \ge 1$ and $n > 2k$, the Kneser graph $K(n, k)$ is Hamiltonian, unless $(n,k) = (5,2)$.
\end{theorem}

The proof of Merino, M\"{u}tze and Namrata heavily relies on the structure of the Kneser graph, while our main results are much more robust, and we can use them to obtain strong versions of the above theorem for nearly all values of $n$ and $k$ as follows. Denoting $K = K(2k+t, k)$, $K$ is an $(n, d, \lambda)$-graph with $n = \binom{2k+t}{k}, d = \binom{k+t}{t}$ and $\lambda = \binom{k+t-1}{t} = \left(1 - \frac{t}{k+t}\right)d$ (see e.g.~\cite[Chapter~6]{godsil-meagher}). 
A simple computation shows that as long as $t$ is at least some large constant, Corollary~\ref{cor:far-from-bipartite-spectral} applies and we may conclude that $K(2k+t, k)$ is Hamiltonian. In fact, a robust version of our theorem, Theorem~\ref{thm:robust-far-from-bipartite} implies that $K$ is Hamilton connected, answering a question of M\"{u}tze~\cite{mutze-survey} for this range of parameters. Moreover, it is easy to see that removing $s \le \gamma d / 4$ edge-disjoint Hamilton cycles from a $d$-regular $\gamma$-far from bipartite $\gamma$-expander yields a $(d-2s)$-regular $\gamma/2$-far from bipartite $\gamma/2$-expander. Thus, iteratively applying Theorem~\ref{thm:main-far-from}, we find $\Omega(\gamma d) \geq \Omega(d^{1 - 1/t})$ edge-disjoint Hamilton cycles in $K$. Finally, provided $pd \ge (\log n)^C$ for some absolute constant $C$, the edge-percolated Kneser graph $K_p$ is Hamiltonian with high probability. Indeed, it is not hard to see that $K_p$ with high probability retains the expansion properties of $K$, is also far from bipartite and is nearly regular (see Lemma~\ref{lem:edge-percolated-expander}), in which case Theorem~\ref{thm:robust-far-from-bipartite} implies that $K_p$ is Hamiltonian.

The robust versions of our results can be used to find spanning subgraphs other than Hamilton cycles. For instance, for any fixed $t$, under the assumptions of Theorem~\ref{thm:main-far-from}, the graph contains a spanning subdivision of $K_t$. Likewise, in the bipartite case, under the assumptions of Theorem~\ref{thm:main-bipartite}, the graph contains a spanning subdivision of $K_{2t}$ (note that for odd $t \ge 5$, even the complete bipartite graph $K_{d,d}$ does not contain a spanning subdivision of $K_t$). This strengthens results of Letzter, Methuku and Sudakov~\cite{letzter-methuku-sudakov-nearly-ham-cycles} who proved the existence of nearly spanning subdivisions in nearly-regular sublinear expanders and complements the results of Lee, Pavez-Sign\'{e} and Petrov~\cite{lee-pavezsigne-petrov} who found spanning clique subdivisions in regular strong linear expanders.

Finally, let us mention that a key part of our proof is a random connecting lemma for sublinear expanders, Lemma~\ref{lem:connections}, which might be of independent interest. Roughly speaking, it states that in a nearly regular sublinear expander, a large number of uniformly distributed pairs of endpoints can be connected with vertex disjoint paths such that the vertices appearing in those paths form a random-like set.

\textbf{Organisation.} The rest of this paper is organised as follows. In Section~\ref{sec:outline} we present a detailed outline of the proof of our main theorems. In Section~\ref{sec:random-walks} we introduce several key definitions used in the proof, we recall known results relating the spectrum, expansion and the mixing time of random walks and we prove several simple lemmas that will be used throughout the proof. Section~\ref{sec:connecting} is the most technical part of the paper and it contains the proof of the random connecting lemma, Lemma~\ref{lem:connections}. This lemma is then used in Section~\ref{sec:absorbers} for the construction of absorbers with uniformly distributed vertices. In Section~\ref{sec:robust} we combine these tools to prove the robust versions of our main theorems, and in Section~\ref{sec:applications} we use these robust theorems to prove Theorems~\ref{thm:dense-regular}~and~\ref{thm:lovasz-percolated}. Finally, we end with some concluding remarks in Section~\ref{sec:concluding}.

\textbf{Notation.} We use standard graph theoretic notation. For a graph $G$, we use $V(G)$ and $E(G)$ for its vertex and edge sets, respectively. We denote $e(G) = |E(G)|$. For $v \in V(G)$, we use $d_G(v)$ for its degree in $G$ and for a set $S 
\subseteq V(G),$ we write $d_G(S) = \sum_{v \in S} d_G(v)$. We shall often omit the subscript when $G$ is clear from the context. We denote by $\delta(G), \bar{d}(G)$ and $\Delta(G)$ the minimum, average and maximum degree of $G$, respectively. We say that a graph is $(d \pm d')$-nearly-regular if all of its vertices have degree between $d - d'$ and $d + d'$. For a set $S \subseteq V(G),$ we use $G[S]$ for the subgraph of $G$ induced by $S$ and we denote $G \setminus S = G[V(G) \setminus S].$ For two disjoint sets $A, B \subseteq V(G),$ $G[A,B]$ is the bipartite graph with vertex set $A \cup B$ and edge set $(A \times B) \cap E(G)$ and we denote $e_G(A, B) = e(G[A,B])$. We write $\mathrm{Aut}(G)$ for the automorphism group of $G$. A $p$-random subset of a set $S$ is a random set formed by keeping each element of $S$ independently with probability $p$.

We use $x = a \pm b$ to mean $x \in [a-b, a+b]$. Arithmetic operations in this notation should be viewed as operations on real closed intervals, e.g. for nonnegative $a, b, c, d,$ the equation $x = (a \pm b) \cdot (c \pm d) = ac \pm (bc + bd + ad)$ should be interpreted as $x \in [a-b, a+b] \cdot [c-d, c+d] \subseteq [ac - (bc + bd + ad), ac + (bc + bd + ad)]$.

Our asymptotic notation is as $n$ tends to infinity, unless specified otherwise. Whenever they are not crucial, we omit the use of floor and ceiling signs for the sake of clarity. We use $[n]$ for the set $\{1, \dots, n\}$. We use $\log x$ to denote the natural logarithm of $x$. 

\section{Proof outline}\label{sec:outline}
In this section, we describe our approach for the proof of Theorem~\ref{thm:main-far-from}. The proof of Theorem~\ref{thm:main-bipartite} is similar, and we comment on the necessary modifications at the end of this section. Let $G$ be an $n$-vertex $d$-regular $\gamma$-expander which is $\eps$-far from being bipartite. Throughout this outline, let us focus on the case where $\explet = 1 / (\log n)^2$, say, and $\eps$ is a positive constant. Assuming that $d \ge (\log n)^C$ for a large enough constant $C$, we wish to prove that $G$ is Hamiltonian.

Our starting point is an adaptation of an argument of Chakraborti, Janzer, Methuku and Montgomery~\cite{chakraborti2025edge} who showed that any $n$-vertex graph with at least $n (\log n)^C$ edges contains two edge-disjoint cycles on the same vertex set, where $C$ is an absolute constant. Let us remark that in~\cite{chakraborti2025edge} the graph need not be a nearly-regular sublinear expander, and a big part of the proof there is concerned with essentially, though not completely, reducing to such a case. Next, we present a natural translation of their approach to our setting and comment on the new obstacles that arise.

The proof is based on the idea of absorption, a powerful method introduced by R\"{o}dl, Ruci\'{n}ski and Szemer\'{e}di~\cite{rodl-rucinski-szemeredi}. Let $p = d^{-c},$ for a constant $c > 0,$ and let the ``reservoir'' set $R$ be a $p$-random subset of $V(G)$. Adapting the argument of Buci\'{c} and Montgomery~\cite{bucic2022towards}, it was shown in~\cite{chakraborti2025edge} that with high probability, the random set $R$ can be used to connect roughly $p^{10} n$ pairs of vertices in $G$ with short pairwise internally vertex disjoint paths through $R$, provided that the multiset of endpoints we wish to connect is reasonably spread out, that is, no vertex in $R$ has too many endpoints in its neighbourhood.

We want to construct an absorber for $R$, that is, a subgraph $H$ with $V(H) \supseteq R$ such that for any subset $R' \subseteq R$, there exists a path with vertex set $V(H) \setminus R'$. The ``template'' for constructing $H$ is fairly simple and we shall comment on it in more detail in Subsection~\ref{subsec:absorber-outline}.

After constructing the absorber $H$, the strategy is to decompose the remainder $G' \coloneqq G \setminus V(H)$ into few paths. The endpoints of these paths are then to be connected into a cycle using paths through $R$. By the properties of $H$, the unused vertices in $R$ can then be absorbed into the cycle (along with the other vertices in the absorber) to produce a Hamilton cycle. For this plan to work, we must be able to cover $G'$ with at most $p^{10} n$ paths whose endpoints are reasonably spread out. A widely used way (e.g.~\cite{kelmans-mubayi-sudakov, chakraborti2025edge, montgomery-approximate}) to construct such a linear forest is as follows. First, we  randomly partition $V(G')$ into $t$ parts $V_1 \cup V_2 \cup \dots \cup V_t$. Then for each $i \in [t-1]$, decompose the edges of $G'[V_i, V_{i+1}]$ into a small number of matchings (using Vizing's theorem, say) and take $M_i$ to be a random one of these matchings. We let $F \coloneqq \bigcup_{i \in [t-1]} M_i$ be the desired linear forest.

This works well if $G'$ is very close to being regular. Indeed, suppose that all degrees in $G'$ are $(1 \pm \beta) d'$, for some $d' = \Theta(d)$, and $\beta = d^{-\Omega(1)}$. Then, with high probability, for any $i \in [t-1]$, all degrees in $G'[V_i, V_{i+1}]$ are $(1 \pm O(\beta)) \frac{d'}{t}$ (assuming an appropriate choice of $t$). This implies that we may decompose $G'[V_i,V_{i+1}]$ into $(1 + O(\beta)) \frac{d'}{t}$ matchings. It is easy to see that any vertex $v \in V_i, i \in [2, t-1]$ is an endpoint of a path in $F$ with probability at most $O(\beta)$. Using standard martingale concentration inequalities, it follows that $F$ consists of at most $O(\beta n)$ paths, whose endpoints are reasonably well spread (again, if $t$ is appropriately chosen). Thus, we would like that $G 
\setminus V(H)$ is nearly regular, or equivalently, that $V(H)$ intersects the neighbourhood of every vertex approximately equally. More precisely, we shall require that $V(H)$ intersects the neighbourhood of every vertex in $d |V(H)| / n \pm \alpha d$ vertices, for $\alpha = d^{-\Omega(1)}$. If this holds, we shall say that $V(H)$ is $\alpha$-uniform (see Definition~\ref{def:uniform}).

To summarize, we will construct the desired absorber $H$ for $R$ such that $V(H)$ is $\alpha$-uniform for $\alpha = d^{-\Omega(1)}$. The formal statement is given in Lemma~\ref{lem:absorber-far-from-bipartite} for the far-from-bipartite case, and in  Lemma~\ref{lem:absorber-bipartite} for the bipartite case. These lemmas and their proofs represent the main technical contribution of our paper, thus we shall now discuss them in detail. We also point out that this is the key novelty compared to the argument of Chakraborti, Janzer, Methuku and Montgomery \cite{chakraborti2025edge}; in their setting uniformity of $V(H)$ is not needed, and their approach for constructing the absorber does not guarantee any pseudorandomness for $V(H)$.

\subsection{Constructing a random-like absorber} \label{subsec:absorber-outline}
Let us first describe the template for constructing the absorber $H$  for $R$. Denote $R = \{a_1, \dots, a_{|R|} \}$ and let $X=\{x_1, \dots, x_{|R|+1} \}$ be a set of vertices disjoint from $R$, both of which are chosen at random. For each $i \in [|R|],$ we construct an $(x_i, a_i, x_{i+1})$-absorber $F_i$ which means that in $G$ there is an $x_i - x_{i+1}$ path with vertex set $V(F_i)$ as well as an $x_i - x_{i+1}$ path with vertex set $V(F_i) \setminus \{a_i\}$. An $(x, a, y)$-absorber is illustrated in Figure~\ref{fig:non-bip-gadget}
and its construction is described after its definition, Definition~\ref{def:non-bip-gadget}. Any particular $F_i$ is constructed in three rounds as follows: in the first round, we find a path $Q_1$ from $a_i$ to $x_{i+1}$ of length $\ell$, in the second round we find a path $Q_2$ from $a_i$ to a given vertex in $Q_1$ and in the third round we find roughly $\ell$ paths with pairs of endpoints prescribed as certain vertices in $Q_1$ and $Q_2$, as well as $x_i$.

To construct the entire absorber $H$, we perform these three rounds in parallel for all $i \in [|R|]$. Our task thus reduces to the following problem, which we describe informally.

\begin{problem}
    Given a nearly-regular sublinear expander and prescribed pairs of uniformly distributed endpoints, find pairwise internally vertex disjoint paths connecting these pairs such that the vertices in these paths are also uniformly distributed.
\end{problem}

Note that the crucial requirement is that the set of vertices appearing in the constructed paths should be uniformly distributed. Indeed, without this requirement, it is known that these connections can be made since $G$ is a $\explet$-expander and, in fact, this task is significantly easier than connecting pairs through a random set.

Our solution to the above problem is given in Lemma~\ref{lem:connections}. Let us mention here that Lemma~\ref{lem:connections} is a combined statement that works both for far-from-bipartite as well as for bipartite sublinear expanders. In the former case, it shows that if the (multi)set of prescribed endpoints is $\alpha$-uniform, then for any $k$, the set of $k$-th vertices in our paths is an $\alpha^{1/2}$-uniform set. In fact, the proof of Lemma~\ref{lem:connections} shows that the set of $k$-th vertices on the found paths can, for many applications, be thought of as a random subset of vertices. It is not hard to see that removing a not too large, $(d^{-c})$-uniform subset of vertices from a $(d \pm d^{1-c})$-nearly-regular $\explet$-expander results in a $(d' \pm d^{1-c/2})$-nearly regular $\explet'$-expander, where $d'\geq d/2$ and $\explet' \ge \explet/2,$ say (see Lemma~\ref{lem:remove-set}). Recall that the sets $R$ and $X$ are random, so they are with high probability $d^{-1/2}$-uniform. Thus, we can indeed iterate Lemma~\ref{lem:connections} three times to find the absorber $H$ such that $V(H)$ is a $(d^{-\Omega(1)})$-uniform set.

\subsection{A random connecting lemma}
The known versions of connecting lemmas for sublinear expanders do not provide sufficient control on the set of vertices used by the connecting paths. Thus, we believe that both the statement and the proof of Lemma~\ref{lem:connections} are likely to have further applications. Let us now sketch the proof.

Since we want that our paths use a random-like set of vertices, a natural approach is to connect pairs using random paths. The key to analyzing a set of random paths is the fact that a random walk on $G$ is rapidly mixing. More precisely, in our setting, a random walk on $G$ mixes in time $(\log n)^{O(1)}$. This is proved using known spectral arguments relating expansion and the spectrum of $G$ (see e.g.~\cite{lovasz-survey}), where we remark that to bound the smallest eigenvalue of far-from-bipartite sublinear expanders, we use a result of Trevisan~\cite{trevisan}. In similar settings, random walks have been used by Jiang, Letzter, Methuku and Yepremyan~\cite{jiang-letzter-methuku-yepremyan} to find rainbow subdivisions of cliques and later by Dragani\'{c}, Methuku, Munh\'{a} Correia and Sudakov~\cite{cycleswithchords} to find cycles with many chords.

Suppose that we wish to connect pairs $(a_i, b_i)_{i \in [m]}$, where the multisets $A = \{a_1, \dots, a_m\}$ and $B = \{b_1, \dots, b_m\}$ are sufficiently uniform, with random looking pairwise internally vertex disjoint paths, where $m = pn$ and $p = (\log n)^{-K}$, for some large constant $K$. Let $\ell = (\log n)^{\Theta(1)}$ be much larger than the mixing time of $G$ and for each $i \in [m],$ let $P_i$ be a random path of length $\ell$ from $a_i$ to $b_i$. Since $\ell$ is much larger than the mixing time, after $\ell / 2$ steps, the path $P_i$ ``forgets'' its starting point $a_i$. Thus, we may think of $P_i$ as two random walks of length $\ell/2$, one from $a_i$ and another from $b_i$, glued at their ends. More formally, up to small errors, $P_i$ can be modeled by coupling a random walk of length $\ell/2$ from $a_i$ with a random walk of length $\ell/2$ from $b_i$ such that their endpoints coincide. This means that to understand the vertices visited by our set of paths $\{ P_i \mid i \in [m]\},$ the precise pairing of the endpoints is irrelevant, rather what matters are the multisets of endpoints $A$ and $B$. 

Let us denote by $W$ the set of internal vertices in the paths $\{ P_i \mid i \in [m]\}$. Using that $A$ and $B$ are uniform, we can show that each vertex appears in $W$ with almost the same probability. Indeed, consider sampling a random index $i \in [m]$ and let $v$ be the second vertex in the path $P_i$. By our discussion above, this is essentially the same as sampling a random index $i$ and taking the second vertex of a random walk from $a_i,$ that is, a random neighbour of $a_i.$ Since $A$ is a uniform multiset, it is not hard to see that $v$ is approximately uniformly distributed among all vertices in $G$. Since the stationary distribution of a random walk in a regular graph is the uniform distribution, the same argument shows that for any $j \in [1, \ell/2]$, the $j$-th vertex of $P_i$, where $i$ is a random index, is approximately uniformly distributed in $V(G)$. By symmetry, the same holds for $j \in [\ell/2, \ell-1]$, as well. This shows that every vertex has nearly equal probability of appearing in $W$.

Of course, the key aspect of a connecting lemma is that the obtained paths should be internally vertex disjoint, which need not hold for the set of paths $\{P_i \mid i \in [m]\}$. In fact, the expected number of intersecting pairs of paths is roughly $m^2 \cdot \ell^2 / n \approx pm$, where since $p 
\ll \ell^{-1}$, for simplicity we ignore the dependence on $\ell$. Since only about a $p$-fraction of the paths $P_i$ intersect other paths, this suggests that we may try to discard and resample them. So, let $\cP = \{ P_i \mid i \in [m]\}$ and let $\cP'$ denote the set of paths in $\cP$ which are internally vertex disjoint from all other paths in $\cP$. If we can show that the set of vertices appearing in the paths in $\cP'$ is sufficiently uniform, we can remove all the vertices in these paths from the graph and iterate to connect the remaining approximately $pm$ pairs.

Thus, the crux of the argument is to show that the internal vertices of the paths in $\cP'$ form a random-like set. Let $W'$ denote the set of vertices appearing in the paths in $\cP'$. We show that each vertex appears in $W'$ with almost the same probability, and moreover there is limited dependence between different vertices in $W'$. Formally, for any set $S \subseteq V(G)$, we prove that the random variable $|W' \cap S|$ has expectation very close to $\frac{|W'||S|}{n}$ and that it has exponentially decaying tails.

Recall that, by our previous discussion, $W$ is a random-like set. Thus, it suffices to show that for any vertex $v \in V(G)$ and any $i \in [m]$, the events $v \in P_i$ and $P_i \in \cP'$ are essentially independent. To this end, we show that, even after conditioning on a path $P \in \cP$ containing $v$, the so-called Poisson paradigm applies to the number of intersections of the other paths in $\cP$ with $P$. Thus, the probability that $P \in \cP'$ is approximately determined by the expected number of these intersections which is very close to $(m-1) \cdot \frac{(\ell-1)^2}{n}$ regardless of $P$, where note that $\ell-1$ is the number of internal vertices of each of the paths $P_i$. We conclude that each vertex appears in $W'$ with almost the same probability. Finally, to show concentration of $|W' \cap S|$ for a given set $S$, we use Talagrand's inequality. Hence, with high probability, $W'$ is sufficiently uniform. A similar argument shows that the set of remaining pairs to connect is also sufficiently uniform, thus we may indeed iterate this argument to find the remaining paths.

\textbf{The bipartite case.} Finally, let us comment on the modifications that need to be made for the proof of Theorem~\ref{thm:main-bipartite}. Firstly, the notion of mixing time is replaced by the notion of bipartite mixing time which was already used, e.g., in~\cite{jiang-letzter-methuku-yepremyan}~and~\cite{cycleswithchords}. Secondly, the construction of the absorber is more complicated than in the far-from-bipartite case. Our construction is almost identical to that in~\cite{chakraborti2025edge}, so we only say a few words about it here. Due to the bipartiteness, one cannot construct gadgets that absorb single vertices, i.e. $(x,a,y)$-absorbers, and instead we use gadgets that absorb a pair of vertices in different parts of the bipartition (see Definition~\ref{def:bip-gadget}). Then, to obtain a bipartite absorber for $R$, we construct gadgets for $\Theta(|R|)$ pairs of vertices in $R$ given by a ``robustly matchable'' auxiliary graph, first introduced by Montgomery~\cite{montgomery}. Finally, throughout the proof, we need to make sure that we use up vertices in different parts equally.

\section{Expansion, spectral properties and random walks} \label{sec:random-walks}
For a walk $W$ of length $\ell$, we denote its vertices in order by $W(0), W(1), \dots, W(\ell).$ For $a \le b,$ we use $W([a, b])$ to denote the subwalk $W(a), W(a+1), \dots, W(b)$. We often slightly abuse notation and use $W$ for the vertex set of $W$. We say that a walk $W$ is \emph{degenerate} if it visits the same vertex more than once, i.e. if $W(i) = W(j)$ for some $0 \le i < j \le \ell$. Otherwise, it is \emph{non-degenerate}, in which case it forms a path from $W(0)$ to $W(\ell)$.

Given a graph $G$, a vertex $v \in V(G)$ and an integer $\ell \ge 0$, let $\cW_G^{\ell}(v)$ denote the random walk of length $\ell$ starting from vertex $v$, i.e. a random walk $W(0), \dots, W(\ell)$, where $W(0) = v$ and for $i \in [\ell],$ $W(i)$ is a uniformly random neighbour of $W(i-1)$ in $G$. For two vertices, $a, b$ and an integer $\ell \ge 0$, we denote by $\cP_G^{\ell}(a, b)$ the conditional distribution of $\cW_G^{\ell}(a)$ conditioned on the event that the walk ends at vertex $b$. For two sets $A, B \subseteq V(G),$ we denote by $\cP_G^{\ell}(A, B)$ a random walk obtained by sampling independently and uniformly at random a vertex $a \in A$ and a vertex $b \in B$ and then sampling a random walk from $\cP_G^{\ell}(a,b)$. In our proofs, $A, B$ and $\ell$ will always be such that for any $a \in A, b \in B$ at least one such walk exists. We omit the subscript whenever $G$ is clear from the context.

The following two definitions might appear strange at first. One should think of them as definitions unifying the case of bipartite graphs and that of far from bipartite graphs.
\begin{defn}
    Let $G$ be a graph and let $V_0, V_1 \subseteq V(G)$. We say that $G$ is $(V_0, V_1)$-\emph{alternating} if either $V_0 = V_1 = V(G)$ or $G$ is bipartite with parts $V_0$ and $V_1$, where $|V_0| = |V_1|$. We call a set $S \subseteq V(G)$ \emph{balanced} if $|S \cap V_0| = |S \cap V_1|$.
\end{defn}
When working with a $(V_0, V_1)$-alternating graph, we shall treat indices modulo $2$ so that $V_t$ stands for $V_0$ if $t$ is even, and for $V_1$ if $t$ is odd.

\begin{defn}
    Let $G$ be an $n$-vertex graph and let $V_0, V_1 \subseteq V(G)$. We say that $G$ is $(V_0, V_1)$-mixing in time $t_0$ if $G$ is $(V_0, V_1)$-alternating and the following holds. Let $i \in \{0, 1\}, t \ge t_0$ be arbitrary and let $j = (i + t) \pmod{2}$. Then, for any $a \in V_i$ and $b \in V_j$,
    \[ \left| \Pr_{W \sim \cW^t(a)}[W(t) = b] - \frac{d(b)}{d(V_j)} \right| \le n^{-10}. \]
\end{defn}

Given a graph $G$, let $A(G)$ denote the adjacency matrix of $G$ and let $D(G)$ be the diagonal matrix with $D(G)_{i,i} = 1 / d_G(i)$ for $i \in V(G)$ and $D(G)_{i,j} = 0$ for $i \neq j$. Let $N(G) = D(G)^{1/2} A(G) D(G)^{1/2}$. For an $n$-vertex graph $G$, since $N(G)$ is a symmetric matrix, it has $n$ real eigenvalues, which we denote by $\lambda_1(N(G)), \dots, \lambda_n(N(G))$, where $1 = \lambda_1(N(G)) \ge \lambda_2(N(G)) \ge \dots \ge \lambda_n(N(G))$. The following lemma, which can be obtained by combining Theorem 2.8 and Lemma 2.9 from~\cite{cycleswithchords}, bounds the second eigenvalue of nearly regular expanders.

\begin{lemma}[\cite{cycleswithchords}] \plabel{lem:second-eigenvalue}
    If $G$ is a $(d \pm d')$-nearly regular $\explet$-expander with $d' \le d/4,$ then $\lambda_2(N(G)) \le 1 - \explet^2 / 32.$
\end{lemma}

We shall use the following lemma from~\cite{cycleswithchords} bounding the mixing time of nearly regular bipartite expanders, which follows from Lemma~\ref{lem:second-eigenvalue} and a lemma in~\cite{jiang-letzter-methuku-yepremyan} relating the eigenvalues of $N(G)$ and the mixing time of random walks for bipartite graphs. The definition of mixing time in~\cite{cycleswithchords} has error $n^{-2}$ as opposed to $n^{-10}$, but this changes nothing apart from the constant factor in the following lemma, so we shall not reprove this trivial modification.

\begin{lemma}[{\cite[Corollary 2.12]{cycleswithchords}}] \plabel{lem:mixing-time-bipartite}
    Let $G$ be a $(d \pm d')$-nearly regular bipartite graph with parts $V_0, V_1$ of size $n/2$. If $n$ is large enough, $d' \le d/4$ and $G$ is a $\explet$-expander, then $G$ is $(V_0,V_1)$-mixing in time $t_0 = 1000 \explet^{-2} \log n$.
\end{lemma}

To bound the mixing time of random walks in far from bipartite graphs, we need to bound their smallest eigenvalue. We shall do so using the following result of Trevisan~\cite{trevisan}.

\begin{lemma}[{\cite[Lemma 3]{trevisan}}] \plabel{lem:trevisan}
    Suppose there is a vector $x \in \mathbb{R}^{V(G)}$ such that $x^T A x \le (c-1) \cdot x^T D(G)^{-1} x.$ Then, there exists a vector $y \in \{-1, 0, 1\}^{V(G)}$ such that
    \[ \frac{\sum_{u \sim v} |y_u + y_v|}{\sum_{v \in V(G)} d_G(v) |y_v| } \le \sqrt{8 c}. \]
\end{lemma}

\begin{lemma} \plabel{lem:last-eigenvalue}
    Let $G$ be a $(d \pm d')$-nearly regular $\explet$-expander and assume that $G$ is $\eps$-far from bipartite. If $d'/d \le \eps^2 \explet^2 / 4000,$ then $\lambda_n(N(G)) \ge \beta - 1$ for $\beta = \eps^2 \explet^2 / 800.$
\end{lemma}

\begin{proof}
    Let $A = A(G), D = D(G), N = N(G)$ and for the sake of contradiction, suppose that $\lambda_n(N) < \beta-1$, where $\beta$ is defined in the statement. In other words, there is a unit vector $x \in \mathbb{R}^{V(G)}$ such that $x^T N x < \beta-1$.

    Since $G$ is $(d \pm d')$-nearly regular, 
    \[ x^T N x = \sum_{u \sim v} \frac{x_u x_v}{ \sqrt{d(u) d(v)}} = \sum_{u \sim v} \frac{x_u x_v}{d \pm d'} = \frac{1}{d} \sum_{u \sim v} x_u x_v (1 \pm 2d' / d) = \frac{1}{d} (x^T A x \pm 2d' / d \cdot (d + d')) = \frac{1}{d} x^T A x \pm 4d' / d, \]
    where to justify the second to last inequality, note that for an arbitrary matrix $M$ obtained by replacing each $1$-entry of $A$ by a real value in $[-2d'/d, 2d'/d]$, we have $|x^T M x| \le \frac{2d'}{d} \cdot (d + d'),$ which holds true because $\norm{M}_2 \le \sqrt{ \norm{M}_1 \norm{M}_\infty} \le \frac{2d'}{d} \cdot (d + d')$, where we used that each absolute row and column sum of $M$ is at most $\frac{2d'}{d} \cdot (d + d')$.
    This implies that $x^T A x \le (\beta-1)d + 4d'$.
    Since $x^T D^{-1} x \le d + d',$ for $c = \beta + 5d'  / d \le \eps^2 \explet^2 / 400,$ we have $x^T A x \le (c - 1) x^T D^{-1} x.$

    By Lemma~\ref{lem:trevisan}, there is a vector $y \in \{-1, 0, 1\}^{V(G)}$ such that
    \begin{equation} \plabel{eq:trevisan-y}
        \frac{\sum_{u \sim v} |y_u + y_v|}{\sum_{v \in V(G)} d_G(v) |y_v| } \le \sqrt{8 c}.        
    \end{equation}
    Let $L = \{ v \in V(G) \mid y_v = -1 \}, R = \{ v \in V(G) \mid y_v = 1 \}$ and $S = L \cup R$. Note that \eqref{eq:trevisan-y} implies
    \begin{equation*} \plabel{eq:trevisan-partition}
        4 e_G(L) + 4e_G(R) + 2 e_G(S, V(G) \setminus S) \le \sqrt{8c} (d + d') |S| \le 4 c^{1/2} d|S|.
    \end{equation*}
    In particular, $e_G(S, V(G) \setminus S) \le 2 c^{1/2} d |S|.$ Since $G$ is a $\explet$-expander, we have $e_G(S, V(G) \setminus S) \ge \min \{ |S|, n - |S|\} \cdot \explet d / 2.$ Suppose $|S| \le n/2,$ then $2 c^{1/2} d |S| \ge |S| \explet d / 2,$ implying $c \ge \explet^2 / 16$, which contradicts $c \le \eps^2 \explet^2 / 400$, since $\eps \le 1$. Thus, $|S| \geq n/2,$ and so
    \[ (n - |S|) \explet d / 2 \le 2 c^{1/2} d |S| \le 2 c^{1/2} d n, \]
    implying $n - |S| \le \frac{4 c^{1/2}}{\explet} \cdot n$. Let $L' = L \cup (V(G) \setminus S)$. We claim that the cut $(L', R)$ has too many edges. Indeed,
    \begin{align*}
        e_G(L') + e_G(R) &\le e_G(L) + e_G(R) + (d+d') |L' \setminus L| \le  c^{1/2} d n + \frac{5}{4}d(n-|S|)\\
        &\le (c^{1/2} + 5c^{1/2} / \explet) \cdot dn \le (6c^{1/2} / \explet) \cdot dn \le (3/10) \eps dn < \eps e(G),
    \end{align*}
    contradicting our assumption that $G$ is $\eps$-far from bipartite. We conclude that $\lambda_n(N) \ge  \beta - 1$, as claimed.
\end{proof}

Finally, the following result relating the mixing time of random walks with the second and smallest eigenvalues of $N$ is well-known and, for example, it follows directly from~\cite[Theorem~5.1]{lovasz-survey}.

\begin{lemma}[\cite{lovasz-survey}] \plabel{lem:lovasz-mixing}
    Let $G$ be a $(d \pm d')$-nearly regular graph on $n$ vertices with $d' \le d/4$ and assume that $\lambda_2(N(G)), -\lambda_n(N(G)) \le 1 - c$ for some $c > 0$. Then, $G$ is $(V(G), V(G))$-mixing in time ${100 c^{-1} \log n}$.
\end{lemma}

Combining Lemmas~\ref{lem:second-eigenvalue},~\ref{lem:last-eigenvalue}~and~\ref{lem:lovasz-mixing}, we obtain the following.
\begin{lemma} \plabel{lem:far-from-bip-mixing}
    Let $G$ be a $(d \pm d')$-nearly regular $\explet$-expander and assume $G$ is $\eps$-far from bipartite. If $d'/d \le \eps^2 \explet^2 / 4000,$ then $G$ is $(V(G), V(G))$-mixing in time $10^5 \log n \cdot \eps^{-2} \explet^{-2}.$
\end{lemma}
\begin{proof}
    By Lemma~\ref{lem:second-eigenvalue}, $\lambda_2(N(G)) \le 1 - \explet^2 / 32$ and by Lemma~\ref{lem:last-eigenvalue}, $\lambda_n(N(G)) \ge \eps^2 \explet^2 / 800 - 1.$ Hence, by Lemma~\ref{lem:lovasz-mixing}, $G$ is $(V(G), V(G))$-mixing in time $10^5 \log n \cdot \eps^{-2} \explet^{-2},$ as needed.
\end{proof}

Throughout the proof, we shall often consider the subgraph of a nearly-regular graph obtained by removing a small set of vertices which intersects the neighbourhood of every vertex approximately in the same number of vertices. In the bipartite case, we might consider a vertex subset of one part intersecting the neighbourhoods of vertices in the other part approximately equally. Hence, the following definition will be crucial for our proof.

\begin{defn} \plabel{def:uniform}
    Given a graph $G$, we say a multiset $S$ is $(V_0, V_1, \alpha)$-bi-uniform with respect to $G$ if $S \subseteq V_0$ and
    \[ \left||N(v) \cap S| - \frac{|S| \cdot \bar{d}(G)}{|V_0|} \right| \le \alpha \bar{d}(G), \forall v \in V_1. \]
    We say a multiset $S$ is $\alpha$-uniform with respect to $G$ if it is $(V(G), V(G), \alpha)$-bi-uniform.    
\end{defn}

When $G$ is clear from the context, we shall often simply say a multiset is $(V_0, V_1, \alpha)$-bi-uniform or $\alpha$-uniform. Let us make several simple observations which follow directly from the definition. These will be implicitly used throughout our proofs.

\begin{observation}
    Let $G$ be a $(V_0, V_1)$-alternating graph.
    \begin{itemize}
        \item If $S \subseteq V_0$ is $(V_0, V_1, \alpha)$-bi-uniform, then it is also $(V_0, V_1, \beta)$-bi-uniform for any $\beta \ge \alpha$.
        \item If $S_1 \subseteq V_0$ is $(V_0, V_1, \alpha_1)$-bi-uniform and $S_2 \subseteq V_0$ is $(V_0, V_1, \alpha_2)$-bi-uniform, then the multiset $S_1 \cup S_2$ is $(V_0, V_1, \alpha_1+\alpha_2)$-bi-uniform.
        \item If $G$ is $(d \pm d')$-nearly regular and $S \subseteq V_0$ is such that $|S| \le \alpha(d - d')$, then $S$ is $(V_0, V_1, \alpha)$-bi-uniform.
    \end{itemize}
\end{observation}

The next lemma shows that the removal of a small uniform set preserves near-regularity and expansion.

\begin{lemma} \plabel{lem:remove-set}
    Let $G$ be an $n$-vertex $(d \pm d')$-nearly regular graph with $d' \le d/8$, let $S \subseteq V(G)$ be an $\alpha$-uniform set of distinct vertices and let $G_2 = G \setminus S.$ Then, for any $v \in V(G)$, it holds that $|N_G(v) \setminus S| = d_2 \pm d_2',$ where $d_2 = d (1 - |S|/n)$ and $d_2' = d'(1 + |S|/n) + \alpha(d + d')$. In particular, $G_2$ is $(d_2 \pm d_2')$-nearly regular. Moreover, if $G$ is a $\explet$-expander, then $G_2$ is a $\explet_2$-expander, where $\explet_2 = \explet - 4(|S|/n+\alpha+d'/d).$
\end{lemma}
\begin{proof}
    Let $G_2 = G \setminus S$. Consider an arbitrary vertex $v \in V(G)$. Since $S$ is $\alpha$-uniform, we have
    \begin{align*}
        |N_G(v) \setminus S|  &= d_G(v)-|N_G(v)\cap S|= d_G(v) - |S|\bar{d}(G) / n \pm \alpha \bar{d}(G) = d\pm d' - (d \pm d') |S| / n \pm \alpha (d \pm d')\\
        &= d(1 - |S|/n) \pm d'(1 + |S|/n) \pm \alpha(d + d') =
        d_2 \pm d_2',
    \end{align*}
    as needed. Since $d_{G_2}(v) = |N_G(v) \setminus S|,$ this implies that $G_2$ is $(d_2 \pm d_2')$-nearly regular.

    Now, suppose $G$ is a $\explet$-expander and we show that $G_2$ is a $\explet_2$-expander. We may assume that $|S| \le n/4$, as otherwise $\explet_2 \le 0$ since $\explet \le 1$, so there would be nothing to prove.
    
    Let $X \subseteq V(G_2)$ be an arbitrary set with $1 \le |X| \le 2|V(G_2)| / 3 \le 2 n/ 3.$ Using that $G$ is a $\explet$-expander and $S$ is $\alpha$-uniform with respect to $G$, we have
    \begin{align*}
        e_{G_2}(X, V(G_2) \setminus X) &= e_G(X, V(G) \setminus X) - e_G(X, S) \ge \explet \bar{d}(G) |X| - |X| \cdot (|S| / n + \alpha) \bar{d}(G)\\
        &= (\explet - |S|/n - \alpha) \bar{d}(G) |X|
    \end{align*}
    Since $\bar{d}(G_2) \le d_2+d_2' \le d(1 - |S|/n + \alpha + 2d'/d)$, we have
    \begin{align*} 
    \frac{e_{G_2}(X, V(G_2) \setminus X)}{\bar{d}(G_2) |X|} &\ge \frac{(d-d') (\explet - |S|/n-\alpha)}{d(1 - |S|/n + \alpha + 2d'/d)} \ge \frac{d \explet}{d(1 - |S|/n + \alpha + 2d'/d)}
    - \frac{d' \explet + d(|S|/n + \alpha)}{d(1 - |S|/n + \alpha + 2d'/d)}\\
    &\ge \explet (1 + |S|/n - \alpha - 2d'/d) - \frac{d' \explet + d(|S| / n + \alpha)}{\frac{3}{4} d}\\
    &\ge \explet - (\alpha + 2d' / d) - \frac{4}{3} (d'/d + |S|/n + \alpha) \ge \explet_2,
    \end{align*}
    where we used that $\frac{1}{1+x} \ge 1-x$ for $x > -1$ and that $\explet \le 1$. This proves that $G_2$ is a $\explet_2$-expander and finishes the proof of the lemma.
\end{proof}

Combining Lemma~\ref{lem:mixing-time-bipartite} and Lemma~\ref{lem:remove-set}, we obtain the following convenient lemma.
\begin{lemma} \plabel{lem:remove-set-bip-mixing-time}
    Let $n$ be sufficiently large, let $G$ be an $n$-vertex $(d \pm d')$-nearly regular bipartite $\explet$-expander with parts $V_0$ and $V_1$, each of size $n/2$ and let $S$ be a balanced $\alpha$-uniform set with respect to $G$. If $|S| \le \explet n / 100, \alpha \le \explet/100$ and $d' \le \explet d / 100$, then $G \setminus S$ is $(V_0 \setminus S, V_1 \setminus S)$-mixing in time $t_0 = 4000 \explet^{-2} \log n$.
\end{lemma}
\begin{proof}   
    By Lemma~\ref{lem:remove-set}, the graph $G_2 = G \setminus S$ is a $(d_2 \pm d_2')$-nearly regular $\explet_2$-expander, where $d_2 = d (1 - |S|/n) \ge d/2$, $d_2' = d'(1 + |S|/n) + \alpha(d + d') \le d/8$ and $\explet_2 = \explet - 4(|S|/n+\alpha+d'/d) \ge \explet / 2.$ By Lemma~\ref{lem:mixing-time-bipartite}, $G_2$ is $(V_0 \setminus S, V_1 \setminus S)$-mixing in time $1000 \explet_2^{-2} \log |V(G_2)| \le 4000 \explet^{-2} \log n$, as needed.
\end{proof}

Similarly, in the far-from-bipartite case, we have the following.
\begin{lemma} \plabel{lem:remove-set-far-from-bip-mixing-time}
    Let $G$ be an $n$-vertex $(d \pm d')$-nearly regular $\explet$-expander, with $0 < \explet \le 1$. Assume that $G$ is $\eps$-far from bipartite for some $\eps > 0$ and $d'/ d \le \eps^2 \explet^2 / 10^5$. Let $S$ be an $\alpha$-uniform subset of $V(G)$ of size at most $\eps \explet n/ 16$, where $\alpha \le \eps^2 \explet^2 / 10^5.$ Then, the graph $G_2 = G \setminus S$ is $(V(G_2), V(G_2))$-mixing in time $t_0 = 10^7 \log n \cdot \eps^{-2} \explet^{-2}.$
\end{lemma}
\begin{proof}
    By Lemma~\ref{lem:remove-set}, the graph $G_2 = G \setminus S$ is a $(d_2 \pm d_2')$-nearly regular $\explet_2$-expander, where $d_2 = d(1 - |S|/n) \ge d/2, d_2' = d'(1 + |S|/n) + \alpha(d + d') \le 2d' + 2\alpha d \le (\eps^2 \explet^2 / 10^4) d$ and $\explet_2 = \explet - 4(|S| /n + \alpha + d'/d) \ge (3/4) \explet$. The number of edges touching $S$ in $G$ is at most $(d + d') |S| \le 1.1d |S| \le (\eps/4) e(G)$, hence $G_2$ is $(3/4) \eps$-far from bipartite. By Lemma~\ref{lem:far-from-bip-mixing}, we get that $G_2$ is $(V(G_2), V(G_2))$-mixing in time $10^5 \log |V(G_2)| \cdot ((3/4) \eps)^{-2} \cdot ((3/4) \explet)^{-2} \le 10^7 \log n \cdot \eps^{-2} \explet^{-2}$.
\end{proof}

Given two random variables $X$ and $Y$ supported on a finite set $S$, we denote their total variation distance as $\TV(X, Y) \coloneqq \sum_{s \in S} |\Pr[X = s] - \Pr[Y = s]|.$
 
The following lemma is the key place where we use the mixing time of random walks.

\begin{lemma} \plabel{lem:combine-two-walks}
    Let $G$ be an $n$-vertex graph and let $V_0, V_1 \subseteq V(G)$ be such that $G$ is $(V_0, V_1)$-mixing in time $t_0$. Let $i \in \{0, 1\}, \ell \ge 2t_0$ and let $j = (i+\ell)\pmod 2$. Let $a \in V_i, b \in V_j$ be fixed. Let $0 \le t \le \ell/2$ be arbitrary and let $W$ be a random walk of length $\ell$ from $a$ to $b$ defined as follows. Sample $W_1 \sim \cW^t(a)$ and let $w = W_1(t)$ denote its endpoint. Sample $W_2 \sim \cP^{\ell-t}(w,b)$ and let $W$ be the concatenation of $W_1$ and $W_2$. Let $W' \sim \cP^\ell(a,b)$. Then,
    \[ \TV(W, W') \le n^{-6}. \]    
\end{lemma}
\begin{proof}
    By definition, conditioning on $W'(t) = w$, the walks $W'([0, t])$ and $W'([t, \ell])$ are distributed as $\cP^t(a,w)$ and $\cP^{\ell-t}(w,b)$, respectively. Therefore, it is enough to show that $\TV(W(t), W'(t)) \le n^{-6}.$
    
    Denote $\ell' = \ceil{\ell/2}$. Let $W_3 \sim \cW^{\ell'
    }(a)$. By the triangle inequality, we have
    \[ \TV(W(t), W'(t)) \le \TV(W(t), W_3(t)) + \TV(W_3(t), W'(t)). \]

    Observe that both $W([0, t])$ and $W_3([0, t])$ are distributed as $\cW^{t}(a)$, so $\TV(W(t), W_3(t)) = 0.$
    
    We claim that $\TV(W_3, W'([0, \ell'])) \le n^{-6},$ which implies the statement. Conditioning on $W_3(\ell') = w'$, $W_3$ is distributed as $\cP^{\ell'}(a, w')$ and conditioning on $W'(\ell')=w'$, $W'[0,\ell']$ is distributed as $\cP^{\ell'}(a,w')$, so it is enough to show that $\TV(W_3(\ell'), W'(\ell')) \le n^{-6}.$
    Since $G$ is $(V_0,V_1)$-mixing in time $t_0 \le \ell',$ setting $j' = (i + \ell') \pmod 2,$ we have
    \begin{equation} \plabel{eq:w_3}
        \Pr[W_3(\ell') = w] = d(w) / d(V_{j'}) \pm n^{-10}, \forall w \in V_{j'}.
    \end{equation}

    Let $W_4 \sim \cW^{\ell}(a)$. We have
    \begin{align*}
        \Pr[W'(\ell') = w] &= \frac{\Pr[W_4(\ell') = w \land W_4(\ell) = b]}{\Pr[W_4(\ell) = b]} = \frac{\Pr[W_4(\ell') = w] \cdot \Pr[W_4(\ell) = b \, \vert \, W_4(\ell') = w]}{\Pr[W_4(\ell) = b]}\\
        &= \frac{\big(d(w) / d(V_{j'})  \pm n^{-10}\big) \cdot \big(d(b) / d(V_j) \pm n^{-10}\big)}{d(b) / d(V_j) \pm n^{-10}} = d(w) / d(V_{j'}) \pm n^{-8}.
    \end{align*}

    Combining with ~\eqref{eq:w_3} and summing over all $w \in V_{j'},$ we have $\TV(W'(\ell'), W_3(\ell')) \le n (n^{-8} + n^{-10}) < n^{-6},$ finishing the proof.
\end{proof}

For a walk $W,$ let $\overline{W}$ denote its reversal. Next, we prove that in a nearly regular graph, the distribution of a random walk from $a$ to $b$ is very close to the distribution of the reversal of a random walk from $b$ to $a$.

\begin{lemma} \plabel{lem:reverse-walk}
    Let $G$ be a $(d \pm d')$-nearly regular $n$-vertex graph with $d' \le d/4$ and let $V_0, V_1 \subseteq V(G)$ be such that $G$ is $(V_0, V_1)$-mixing in time $t_0$. Let $i \in \{0, 1\}$, let $\ell \ge t_0$ and let $j = (i+\ell)\pmod 2$. Let $a \in V_i, b \in V_{j}$. Let $P_1 \sim \cP_{G}^{\ell}(a, b)$ and $P_2 \sim \cP_{G}^{\ell}(b, a)$. Then, for any $a-b$ walk $P$ of length $\ell$, it holds that
    \[ \frac{\Pr[P_1 = P]}{\Pr[\overline{P_2} = P]} = 1 \pm 100d' / d \pm n^{-8}. \]
\end{lemma}
\begin{proof}
    Let $W_1 \sim \cW^\ell(a)$ and $W_2 \sim \cW^\ell(b)$. Then, 
    \begin{align*}
        \Pr[P_1 = P] &= \Pr[W_1 = P] / \Pr[W_1(\ell) = b] = \left(\prod_{i=0}^{\ell-1} d(P(i))^{-1}\right) \cdot (d(b) / d(V_{j}) \pm n^{-10})^{-1}\\
        &= \left(\prod_{i=0}^{\ell-1} d(P(i))^{-1}\right) \cdot (1 \pm 4d' / d \pm n^{-9}) |V_j|,
    \end{align*}
    and analogously
    \begin{align*}
        \Pr[\overline{P_2} = P] &= \Pr[\overline{W_2} = P] / \Pr[W_2(\ell) = a] = \left(\prod_{i=1}^{\ell} d(P(i))^{-1}\right) \cdot (d(a) / d(V_i) \pm n^{-10})^{-1}\\
        &= \left(\prod_{i=1}^{\ell} d(P(i))^{-1}\right) \cdot (1 \pm 4d' / d \pm n^{-9}) |V_i|,
    \end{align*}

    Since $G$ is $(d \pm d')$-nearly regular, it follows that
    \[ \frac{\Pr[P_1 = P]}{\Pr[\overline{P_2} = P]} = 1 \pm 100d' / d \pm n^{-8}. \]

    
    

\end{proof}

Next, we show that a sufficiently long random walk with prescribed endpoints is unlikely to hit any particular vertex.
\begin{lemma} \plabel{lem:path-hits-vtx}
    Let $G$ be an $n$-vertex $(d \pm d')$-nearly regular graph with $d' \le d/100$ that is $(V_0,V_1)$-mixing in time $t_0$. Let $i \in \{0, 1\}$ and $\ell$ be given integers such that $2t_0 \le \ell$ and let $a, b$ be given vertices in $V_i$ and $V_{i+\ell}$, respectively. Then, for any $t \in [1, \ell-1]$ and any $v \in V_t$,
    \[ \Pr_{P \sim \cP^\ell(a, b)}[P(t) = v] \le 3/d. \]
\end{lemma}
\begin{proof}
    Assume first that $t \le \ell / 2$. By Lemma~\ref{lem:combine-two-walks}, we have 
    \[ \Pr_{P \sim \cP^\ell(a,b)}[P(t) = v] \le n^{-6} + \Pr_{W \sim \cW^t(a)}[W(t) = v] \le n^{-6} + 1 / (d-d') \le 2 / (d-d') \le 3/d, \]
    where we used that $\delta(G) \ge d - d'$.

    If $t > \ell / 2$, applying the case $t \le \ell / 2$ to the random walk $\cP^{\ell}(b,a)$ and Lemma~\ref{lem:reverse-walk}, we get 
    \[ \Pr_{P \sim \cP^\ell(a, b)}[P(t) = v] \le (n^{-6} + 1 / (d-d'))  \cdot (1 \pm 100d' / d \pm n^{-8}) \le 2.5 / (d-d') \le 3/d. \]
\end{proof}

Using the previous lemma, we show that a random walk $\cP^{\ell}(a,b)$ is very likely non-degenerate if $\ell$ is larger than the mixing time but small compared to the degree.

\begin{lemma} \plabel{lem:path-degenerate}
    Let $G$ be $(d \pm d')$-nearly regular, where $d' < d / 100$, and suppose that $G$ is $(V_0, V_1)$-mixing in time $t_0$. Given $\ell \ge 4t_0, i \in \{0, 1\}$ and vertices $a \in V_i, b \in V_{i+\ell},$ 
    \[ \Pr_{P \sim \cP^{\ell}(a,b)}[P \text{ is degenerate}] \le 3\ell^2 / d. \]
\end{lemma}
\begin{proof}
    Let $P \sim \cP^{\ell}(a,b)$. We will show that for any $0 \le i < j \le \ell$, $\Pr[P(i) = P(j)] \le 3 / d$, implying that the expected number of self-collisions is less than $3\ell^2 / d$ after which the statement follows by Markov's inequality.

    So, fix $i, j$ with $0 \le i < j < \ell$. If $i \le \ell/2,$ then given $P([0, i]) = \tilde{P},$ denote $c = \tilde{P}(i)$. The remainder of the path, $P([i, \ell])$ is distributed as $\cP^{\ell-i}(c, b)$ so by Lemma~\ref{lem:path-hits-vtx}, we have
    \[ \Pr[P(j) = c \mid P([0, i]) = \tilde{P}] \le \frac{3}{d}. \]
    Since $\tilde{P}$ was arbitrary, we have $\Pr[P(j) = P(i)] \le 3/d,$ as needed.

    If $i > \ell/2,$ then $j > \ell/2$ and the argument is analogous, except we condition on $P([j, \ell]) = \tilde{P}$ and apply Lemma~\ref{lem:path-hits-vtx} to $P([0,j] \mid P[j, \ell] = \tilde{P})$ which is distributed as $\cP^j(a, \tilde{P}(0))$.
\end{proof}

\section{The random connecting lemma} \label{sec:connecting}
In this section we prove our main technical result, Lemma~\ref{lem:connections}, which allows us to connect pairs of vertices with short vertex disjoint paths such that the set of vertices appearing in those paths is a random-like set.

\begin{lemma} \plabel{lem:connections}
    Let $G$ be a $(d \pm d')$-nearly regular $n$-vertex graph, where $n$ is sufficiently large, and let $V_0, V_1 \subseteq V(G)$ be such that $G$ is $(V_0, V_1)$-alternating. Let $m = np$ be a given integer and let $(a_i, b_i)_{i \in [m]}$ be given pairs of vertices with $a_i \in V_0, b_i \in V_1, \forall i \in [m]$ and $a_i \neq b_i$ such that the multiset $\{ a_i \mid i \in [m]\}$ is $(V_0, V_1, \alpha)$-bi-uniform and the multiset $\{ b_i \mid i \in [m]\}$ is $(V_1, V_0, \alpha)$-bi-uniform. Let $U \subseteq V(G)$ be a given $\alpha$-uniform set with respect to $G$ containing the set $\bigcup_{i \in [m]} \{a_i, b_i\}$ such that $|U \cap V_0| = |U \cap V_1|$ and $|U| \le n/100$. Assume that for any $\alpha^{7/8}$-uniform set $T \subseteq V(G)$ with $|T \cap V_0| = |T \cap V_1|$ of size $|T| \le |U| + m\ell$, the graph $G \setminus T$ is $(V_0 \setminus T, V_1 \setminus T)$-mixing in time $t_0$. Let $\ell$ be an odd integer such that $\ell \ge 10t_0$. Assume that the following inequalities hold: $d' \le \alpha d, p \le \ell^{-8}, \alpha \le \ell^{-400}, \alpha \ge d^{-1/4}$ and $\ell \ge \log n$. Then, there is a collection of paths $(P_i)_{i \in [m]},$ where for $i \in [m],$ $P_i$ is an $a_i-b_i$ path of length $\ell$, satisfying the following:    
    \begin{itemize}
        \item $\Vint(P_i) \cap \Vint(P_j) = \emptyset$ for all $1 \le i < j \le m$;
        \item $\Vint(P_i) \cap U = \emptyset$ for all $i \in [m]$;
        \item For all $t \in [\ell-1]$, the set $W_t = \{ P_i(t) \mid i \in [m]\}$ is $(V_t, V_{t+1}, \alpha^{1/2})$-bi-uniform.
    \end{itemize}    
\end{lemma}

\begin{rem}
    Though the statement of Lemma~\ref{lem:connections} is deterministic, it can be modified to obtain a distribution on collections of internally vertex-disjoint paths $P_i$ as in the statement, such that for any $t \in [\ell-1]$ and any set of vertices $S \subseteq V_t$, the random variable $|S \cap \{ P_i(t) \mid i \in [m]\}|$ has expectation very close to $|S| m / |V_t|$ and exponentially decaying tails. Additionally, let us note that $\alpha^{1/2}$ in the conclusion of the lemma can be replaced by $\alpha^{1-c}$ for any positive constant $c$ at the cost of requiring stronger inequalities in the assumption.
\end{rem}
As described in the proof outline, to prove Lemma~\ref{lem:connections} we sample the desired paths at random and keep the set of paths which are disjoint from all the other ones. We perform this iteratively to create the remaining paths. Thus, the main driving force behind the proof of Lemma~\ref{lem:connections} is the following crucial lemma which proves that one step of this iteration behaves sufficiently randomly. In the following, Part~\ref{item:pathdistr} shows that the set of vertices appearing in the kept paths is a random-like set, while Part~\ref{item:indexdistr} shows that the set of pairs connected in one step is random-like so that we may indeed iterate this argument.

\begin{lemma} \plabel{lem:one-bite-probabilistic}
    Let $G$ be a $(d \pm d')$-nearly regular $n$-vertex graph and let $V_0, V_1 \subseteq V(G)$ be such that $G$ is $(V_0, V_1)$-mixing in time $t_0$. Let $m \le n / (8 (\ell+1)^2)$ and suppose we are given sets $A_i, B_i, i \in [m]$ such that $A_i \subseteq V_0, B_i \subseteq V_1$ with $|A_i| = d \pm d'$ and $|B_i| = d \pm d'$ for all $i \in [m]$. Furthermore, assume that for all $v \in V_0, |\{ i \in [m] \mid v \in A_i\}| = d m / |V_0| \pm d'$ and similarly, for all $v \in V_1, |\{ i \in [m] \mid v \in B_i\}| = d m / |V_1| \pm d'$.

    Let $\ell \ge 8t_0$ be an odd integer. For each $i \in [m],$ let $a_i$ and $b_i$ be independently uniformly random vertices from $A_i$ and $B_i$, respectively. Let $P_i \sim \cP^{\ell}_G(a_i, b_i)$. Let
    \[ J_0 = \{ i \in [m] \mid P_i \text{ is non-degenerate}\} \text{ and } J = \{ i \in J_0 \mid P_i \cap P_j = \emptyset, \forall j \in [m], j \neq i \}. \]

    Denote $q \coloneqq e^{-(\ell+1)^2 m / n}$ and $\eta = \max\{  d' |V_0| / (md), \ell^2 / d\}$. There is an absolute constant $C > 0$ such that if $\ell \ge C$ and $\eta < 1/C$, then the following hold:
    \begin{enumerate}[label=\alph*)]
        \item \plabel{item:pathdistr} For any $t \in [0, \ell]$, any $S \subseteq V_t$ and any $\tau > 0,$ writing $\mu \coloneqq q m |S| / |V_t|$, we have
        \[ \Pr\left[ \big| |\{ P_i(t) \mid i \in J\} \cap S | - \mu \big|  > \tau + C(\eta \mu + \ell \sqrt{\mu} + \ell^2)\right] \le 8 \exp\left(\frac{-\tau^2}{C \ell^2 (\mu + \tau)}\right). \]
        \item \plabel{item:indexdistr} Additionally, for any set $I \subseteq [m]$ and any $\tau > 0,$
        \[ \Pr[| |J \cap I| - q|I|| > \tau + C (\eta|I| + \ell \sqrt{|I|} + \ell^2)] \le 8\exp\left( \frac{-\tau^2}{C \ell^2 (|I| + \tau)} \right). \]
    \end{enumerate}
\end{lemma}

In the above lemma, one should think of $q$ as the proportion of paths that we keep -- indeed, property \ref{item:indexdistr} applied with $I=[m]$ shows that we keep roughly $qm$ paths. Then part \ref{item:pathdistr} shows that each set $S$ is hit roughly as many times by the surviving paths as one would expect if the paths consisted of random vertices.

To show concentration, we shall use the following version of Talagrand's inequality.

\begin{theorem}[Talagrand's inequality~{\cite[Lemma~17]{harutyunyan2025delta}}] \plabel{thm:talagrand}
    Let $X$ be a random variable supported on the nonnegative integers, determined by $n$ independent trials and satisfying the following for some integers $c, r \ge 1$:
    \begin{enumerate}[label=(\roman*)]
        \item changing the outcome of any one trial can affect $X$ by at most $c$, and
        \item for every integer $s \ge 0$, if $X \ge s$, then there is a set of at most $rs$ trials whose outcomes certify that $X \ge s$.
    \end{enumerate}
    Then, for any real $\tau > 126 c \sqrt{r \E[X]} + 344c^2 r,$ we have
    \[ \Pr\left[ |X - \E[X]| > \tau\right] \le 4\exp\left( \frac{-\tau^2}{32c^2 r(\E[X] + \tau)} \right). \]
\end{theorem}

\begin{proof}[Proof of Lemma~\ref{lem:one-bite-probabilistic}]
    In the proof of this lemma, asymptotic notation is as $\eta$ tends to $0$.
    \begin{claim} \plabel{cl:expectations}
        For any $v \in V_t$, the following hold.
        \begin{equation} \plabel{eq:v-is-hit}
            \sum_{i=1}^m \Pr[P_i(t) = v] = \frac{m}{|V_t|} \cdot (1 \pm O(\eta)).
        \end{equation}

        \begin{equation} \plabel{eq:v-is-hit-by-non-deg}
            \sum_{i=1}^m \Pr[P_i(t) = v \land P_i \text{ is non-degenerate}] = \frac{m}{|V_t|} \cdot (1 \pm O(\eta)).
        \end{equation}
        \end{claim}
    \begin{proof}
        We will prove the statement in the case $t \le \ell / 2.$ The case $t > \ell / 2$ then follows by Lemma~\ref{lem:reverse-walk} since the conditions are symmetric with respect to $A_i, B_i$ and $V_0, V_1$ and since $(1 \pm 100d' / d \pm n^{-8})(1 \pm O(\eta)) = 1 \pm O(\eta). $ (Note that Lemma \ref{lem:reverse-walk} applies since $d'\leq \eta d\leq d/C$ and $C$ is large.)
    
        Let $v \in V_t$ be arbitrary. Let $P$ be a random walk defined as follows. Sample $j \in [m]$ uniformly at random, let $a$ and $b$ be uniformly random vertices in $A_j$ and $B_j$, respectively and let $P \sim \cP^{\ell}(a, b).$ Thus, showing (\ref{eq:v-is-hit}) is equivalent to showing that
        \[ \Pr[P(t) = v] = |V_t|^{-1} (1 \pm O(\eta)). \]
    
        Let $P'$ be a random walk generated as follows. Let $a'$ be a random vertex in $V_0$, where $\Pr[a' = u] = d(u) / d(V_0)$, for all $u \in V_0$. Then, let $j'$ be a uniformly random index such that $a' \in A_{j'}$ and let $b'$ be a uniformly random vertex in $B_{j'}$. Finally, let $P' \sim \cP^\ell(a', b')$.

        Having chosen $a'$, let $W \sim \cW^t(a')$. By Lemma~\ref{lem:combine-two-walks}, conditioning on $a'$ and $b'$, $\TV(P'(t), W(t)) \le n^{-6}.$ Since $a'$ is distributed proportional to $d(a'),$ we have that $\Pr[W(t) = v] = d(v) / d(V_t)$ for any $v \in V_t$. Therefore, for any $v \in V_t$,
        \begin{equation} \plabel{eq:prob-P'}
            \Pr[P'(t) = v] = d(v) / d(V_t) \pm n^{-6}.
        \end{equation}

        Next, we show that $P'$ and $P$ have a similar distribution. Consider arbitrary $i \in [m]$, $x \in A_i, y \in B_i.$ We have
        \begin{align*}
            \Pr[j' = i \land a' = x \land b' = y] &= \frac{d(x)}{d(V_0)} \cdot |\{i' \mid x \in A_{i'}\}|^{-1} \cdot |B_i|^{-1} = (1 \pm 4d'  / d) \frac{1}{|V_0|} (dm / |V_0| \pm d')^{-1} (d \pm d')^{-1}\\
            &= \frac{1}{d^2 m} \cdot (1 \pm O(\eta)).
        \end{align*}

        On the other hand,
        \begin{align*}
            \Pr[j = i \land a = x \land b = y] = \frac{1}{m} |A_i|^{-1} |B_i|^{-1} = \frac{1}{d^2 m} \cdot (1 \pm O(\eta)).
        \end{align*}

        Note that given $a = a', b = b',$ the walks $P'$ and $P$ are both distributed as $\cP^\ell(a,b)$. Hence, 
        \begin{align*}
            \Pr[P(t) = v] &= \sum_{i=1}^m \sum_{x \in A_i, b \in B_i} \Pr[j = i \land a = x \land b = y] \cdot \Pr[P(t) = v \mid a = x \land b = y]\\
            &= \sum_{i=1}^m \sum_{x \in A_i, b \in B_i} (1 \pm O(\eta)) \Pr[j' = i \land a' = x \land b' = y] \cdot \Pr[P'(t) = v \mid a' = x \land b' = y]\\
            &= (1 \pm O(\eta)) \Pr[P'(t) = v] = (1 \pm O(\eta)) \frac{d(v)}{d(V_t)}=(1 \pm O(\eta)) |V_t|^{-1}.
        \end{align*}
        
        Next, we show that for any $v \in V_t,$ it holds that $\Pr[P \text{ is degenerate} \mid P(t) = v] = O(\eta)$, which will imply (\ref{eq:v-is-hit-by-non-deg}). Let us fix $v \in V_t$.
        
        We produce a random walk $Q$ such that $Q(t) = v$ as follows. Sample $Q_0 \sim \cW^t(v)$ and let $Q_1 = \overline{Q_0}$. Denote $a_Q = Q_1(0)$ and let $j_Q$ be a uniformly random index $i$ such that $a_Q \in A_i$. Let $b_Q$ be a uniformly random vertex in $B_i$. Finally, sample $Q_2 \sim \cP^{\ell-t}(v, b_Q)$ and let $Q$ be the concatenation of $Q_1$ and $Q_2$.

        First, let us show that the distribution of $Q$ is close to the conditional distribution $\{ P \mid P(t) = v \}.$ To this end, consider an arbitrary index $i \in [m]$ and an arbitrary $\ell$-path $\tilde{P}$ such that $\tilde{P}(0) \in A_i, \tilde{P}(\ell) \in B_i$ and $\tilde{P}(t) = v$. We have
        \begin{align*}
            \Pr[j = i \land P = \tilde{P}] &= \frac{1}{m} |A_i|^{-1} |B_i|^{-1} \cdot \Pr[P = \tilde{P} \mid P(0) = \tilde{P}(0) \land P(\ell) = \tilde{P}(\ell)]\\
            &= \frac{1}{m}(d \pm d')^{-2} \cdot \frac{\Pr_{W \sim \cW^\ell_{\tilde{P}(0)}}[W = \tilde{P}]}{\Pr_{W \sim \cW^\ell_{\tilde{P}(0)}}[W(\ell) = \tilde{P}(\ell)]} = \frac{1}{m}(d \pm d')^{-2}\frac{(d \pm d')^{-\ell}}{(1 \pm O(\eta)) / |V_1|}\\
            &= \frac{|V_1|}{m (d \pm d')^{\ell+2}}\cdot (1 \pm O(\eta)) = \frac{|V_1|}{m d^{\ell+2}} \cdot (1 \pm O(\eta)),
        \end{align*}
        where we used that $1/C>\eta \geq \frac{d' |V_0|}{md} \ge \frac{d' \ell^2}{d}$. Recalling that $\Pr[P(t) = v] = (1 \pm O(\eta)) |V_t|^{-1},$ we have
        \[ \Pr[j = i \land P = \tilde{P} \mid P(t) = v] = \frac{|V_1||V_t|}{m d^{\ell+2}} \cdot (1 \pm O(\eta)). \]
        On the other hand,
        \begin{align*}
            \Pr[j_Q = i \land Q = \tilde{P}] &= \Pr[\overline{Q_0} = \tilde{P}([0,t])] \cdot \Pr[j_Q = i \mid Q(0) = \tilde{P}(0)] \cdot \Pr[b_Q = \tilde{P}(\ell) \mid j_Q = i] \cdot \Pr[Q_2 = \tilde{P}([t, \ell]) \mid b_Q = \tilde{P}(\ell)]\\
            &= (d \pm d')^{-t} \cdot (d m / |V_0| \pm d')^{-1} \cdot (d \pm d')^{-1} \cdot \frac{\Pr_{W \sim \cW^{\ell-t}(v)}[W = \tilde{P}([t,\ell])]}{\Pr_{W \sim \cW^{\ell-t}(v)}[W(\ell-t) = \tilde{P}(\ell)]}\\
            &= (d \pm d')^{-t} \cdot (d m / |V_0| \pm d')^{-1} \cdot (d \pm d')^{-1} \cdot \frac{(d \pm d')^{-(\ell-t)}}{(1+O(\eta)) / |V_\ell|}\\
            &= \frac{|V_0||V_\ell|}{md^{\ell+2}} (1 \pm O(\eta)).
        \end{align*}
        Since $|V_0| = |V_1|,$ we have
        \[ \frac{\Pr[j = i \land P = \tilde{P} \mid P(t) = v]} {\Pr[j_Q = i \land Q = \tilde{P}]} = 1 \pm O(\eta). \]

        Hence, it is enough to show that $\Pr[Q \text{ is degenerate}] = O(\eta).$
        Indeed, in that case, we have
        \begin{align*}
            \Pr[P \text{ is degenerate} \mid P(t) = v]
            &=\sum_{i=1}^m \sum_{\tilde{P} \text{ is degenerate}} \Pr[j = i \land P = \tilde{P} \mid P(t) = v] \\
            &=\sum_{i=1}^m \sum_{\tilde{P} \text{ is degenerate}} (1 \pm O(\eta))\Pr[j_Q = i \land Q = \tilde{P}] \\
            &=(1 \pm O(\eta))\Pr[Q \text{ is degenerate}] = O(\eta).
        \end{align*}
        
        Recall that $Q([0, t]) = Q_1 = \overline{Q_0}$ and $Q([t, \ell]) = Q_2$. Since $Q_0$ is a random walk from $v$, we have
        \[ \Pr[Q_1 \text{ is degenerate}] = \Pr[Q_0 \text{ is degenerate}] \le \binom{t+1}{2} \frac{1}{d-d'} \le \frac{(t+1)^2}{d}. \]
        Additionally, for any outcome $\tilde{Q}$ of $Q_1$, we have 
        \[ \Pr[b_Q \in V(Q_1) \mid Q_1 = \tilde{Q}] \le (t+1) / |B_{j_Q}| \le (t+1) / (d - d') \le 2(t+1) / d. \] Using Lemma~\ref{lem:path-hits-vtx} for $Q_2$ and Markov's inequality we get,
        \[ \Pr[V(Q_1) \cap V(Q_2[1,\ell-t]) \neq \emptyset] \le \Pr[b_Q \in V(Q_1)] + \sum_{i=0}^t \sum_{j=1}^{\ell-t-1} \Pr[Q_1(i) = Q_2(j)] \le 2(t+1) / d + (t+1) (\ell-t-1) \cdot 3 / d. \]
        Finally, by Lemma~\ref{lem:path-degenerate},
        \[ \Pr[Q_2 \text{ is degenerate}] \le 3(\ell-t)^2 / d. \]
        Combining, we have
        \begin{align*} 
            \Pr[Q \text{ is degenerate}] &\le \Pr[Q_1 \text{ is degenerate}] + \Pr[Q_2 \text{ is degenerate}] + \Pr[V(Q_1) \cap V(Q_2([1, \ell-t])) \neq \emptyset]\\ &\le (t+1)^2 / d + 3(\ell-t)^2 / d + 2(t+1)/d + (t+1)(\ell-t-1) \cdot 3 / d \le 7\ell^2 / d + 2(t+1) / d = O(\eta).
        \end{align*}
    \end{proof}

    \begin{claim} \plabel{cl:prob-collision}
        Let $I \subseteq [m]$ be an arbitrary subset of size $m-1$. Let $U \subseteq V(G)$ be a given set of vertices of size at most $\ell + 1$. Then, 
        \[ \Pr[P_j \cap U = \emptyset, \forall j \in I] = e^{-(\ell+1) |U| m / n} \cdot (1\pm O(\eta)). \]
    \end{claim}
    \begin{proof}
        For $j \in [m],$ let $X_j = |\{ t \in [0, \ell] \mid P_j(t) \in U\}|$ and let $p_j = \Pr[X_j > 0].$ Then, $ p_j \ge \E[X_j] - \E\left[ \binom{X_j}{2} \right].$
        We claim that $\E\left[ \binom{X_j}{2} \right] \le \E[X_j] \cdot 4|U| \ell / d$. Indeed, we can write $\E\left[ \binom{X_j}{2} \right] = \sum_{0 \le t < t' \le \ell} \Pr[P_j(t) \in U \land P_j(t') \in U].$ To bound this expectation, we consider two cases. 
        Assume first that $t \le \ell - t'$ and note that, in particular, this implies that $t \le \ell/2$. If $t' = \ell$, then $t = 0$, so $\Pr[P_j(t) \in U \land P_j(t') \in U] \le \Pr[P_j(0) \in U] \cdot \frac{|U|}{d - d'} \le \Pr[P_j(0) \in U] \cdot \frac{3|U|}{d}$. If $t' \neq \ell,$ then conditioning on $P_j[0, t]$ and $b_j$, the rest of the walk $P_j[t, \ell]$ is distributed as $\cP^{\ell-t}(P_j(t), b_j)$. Since $\ell - t \ge 2t_0,$ by Lemma~\ref{lem:path-hits-vtx}, for any $b \in B_j$ and any valid walk $\tilde{P}$, we have 
        \[ \Pr[P_j(t') \in U \mid b_j = b \land P_j([0,t]) = \tilde{P}] \le 3|U| / d. \]
        Since $\tilde{P}$ and $b$ were arbitrary, we have 
        \[ \Pr\left[ P_j(t) \in U \land P_j(t') \in U  \right] \le \Pr[P_j(t) \in U] \cdot 3|U| / d. \]

        Now assume that $t > \ell - t'$. Applying Lemma~\ref{lem:reverse-walk} and the case $t \le \ell - t',$ we obtain
        \[  \Pr\left[  P_j(t) \in U \land P_j(t') \in U  \right] \le (1 \pm 100d' / d \pm n^{-8}) \cdot \Pr[P_j(t') \in U] \cdot 3|U| / d \le \Pr[P_j(t') \in U] \cdot 4 |U| / d. \]

        Summing over all $0 \le t < t' \le \ell,$ we have
        \[ \E\left[\binom{X_j}{2}\right] \le \sum_{t=0}^{\ell} \Pr[P_j(t) \in U] \cdot \ell \cdot 4|U| / d = \E[X_j] \cdot 4|U|\ell / d. \]

        By Markov's inequality, $p_j \le \E[X_j]$, so $p_j = \E[X_j] (1 \pm O(\eta)).$


        If $V_0 = V_1 = V(G),$ then by~\eqref{eq:v-is-hit}, for any $v \in U$ and any $t \in [0, \ell],$ it holds that $\sum_{j \in [m]} \Pr[P_j(t) = v] = \frac{m}{n}(1 \pm O(\eta)).$ On the other hand, if $V_0 \cap V_1 = \emptyset$ and $|V_0| = |V_1| = n/2,$ then for any $v \in U \cap V_k$, where $k \in \{0,1\},$ and $t \in [0,\ell],$ the sum $\sum_{j \in [m]} \Pr[P_j(t) = v]$ equals $0$ if $k$ and $t$ are of different parity and equals $\frac{m}{|V_t|}(1 \pm O(\eta)) = \frac{2m}{n} (1 \pm O(\eta))$ otherwise, using~\eqref{eq:v-is-hit}. Since $\ell$ is odd, in either case we obtain
        \begin{equation} \plabel{eq:sum_EXj}
            \sum_{j \in [m]} \E[X_j] = \sum_{v \in U} \sum_{t=0}^{\ell} \sum_{j \in [m]} \Pr[P_j(t) = v] = \frac{|U| (\ell+1) m}{n} \cdot (1 \pm O(\eta)).
        \end{equation}
        
        Note that for $t=0$ or $t=\ell$, since $|A_j|, |B_j| \ge d/2,$ for any $j \in [m],$ we have that $\Pr[P_j(t) \in U] \le 2|U| / d.$ Using Lemma~\ref{lem:path-hits-vtx}, it follows that 
        \[ \E[X_j] = \Pr[P_j(0) \in U] + \Pr[P_j(\ell) \in U] + \sum_{t=1}^{\ell-1} \Pr[P_j(t) \in U] \le 4|U| / d + 3 |U| \ell / d = O(\eta), \]
        for any $j \in [m]$.
        Recalling that $p_j \le \E[X_j]$ and combining with~\eqref{eq:sum_EXj}, we have
        \begin{align*}
            \sum_{j \in I} p_j &= \sum_{j \in [m]} p_j \pm O(\eta) = \frac{|U| (\ell+1) m}{n} \cdot (1 \pm O(\eta)) \pm O(\eta).
        \end{align*}        
        Finally, we can estimate the desired probability as follows. Using that $e^{-x} \ge 1 - x \ge e^{-x-x^2}$ for $0 \le x \le 1/2,$ we obtain
        \begin{align*}
             \Pr[P_j \cap U = \emptyset, \forall j \in I] &= \prod_{j \in I} (1 - p_j) = \prod_{j \in I} e^{-p_j (1 \pm p_j)}\\
            &= \exp\left(\left(-\frac{|U| (\ell+1) m}{n} \cdot (1 \pm O(\eta)) \pm O(\eta) \right) \cdot (1 \pm O(\eta)) \right)\\
            &= \exp\big(-(\ell+1) |U| m / n  \pm O(\eta) \big) =\exp\big(-(\ell+1) |U| m / n \big)\cdot (1 \pm O(\eta)),            
        \end{align*}
        where in the second last equality we used that $|U| \le \ell+1$ and $(\ell+1)^2 m / n \le 1/8$.
    \end{proof}
    With these two claims, we are ready to prove the lemma. We start with part~\ref{item:pathdistr}. We define several random variables. Let 
    \begin{align*}
        X &= |\{ v \in S \mid \exists i \in J, P_i(t) = v\}|, \\
        X_1 &= |\{ v \in S \mid \exists i \in [m], P_i(t) = v \land P_i \text{ is non-degenerate} \}|, \\
        X_2 &= |\{ v \in S \mid \exists i, i' \in [m], i \neq i', P_i(t) = v \land P_i \text{ is non-degenerate} \land P_i \cap P_{i'} \neq \emptyset \}|.        
    \end{align*}

    Note that $X$ is the relevant variable in the statement of part~\ref{item:pathdistr} and observe that $X = X_1 - X_2.$ First, we calculate $\E[X]$. Using Claims~\ref{cl:prob-collision}~and~\ref{cl:expectations}, we have    
    \begin{align*}
        \E[X] &= \sum_{i \in [m]} \Pr\big[P_i(t) \in S \land P_i \text{ is non-degenerate} \land P_i \cap P_j = \emptyset, \forall j \in [m] \setminus \{i\} \big]\\
        &= \sum_{i \in [m]} \Pr\big[P_i(t) \in S \land P_i \text{ is non-degenerate}] \cdot e^{-(\ell+1)^2 m / n} \cdot (1 \pm O(\eta)) \\
        &= \frac{|S| m}{|V_t|}(1 \pm O(\eta)) \cdot q \cdot (1 \pm O(\eta)) = \mu(1 \pm O(\eta)).
    \end{align*}
       
    Using Talagrand's inequality, we show that both $X_1$ and $X_2$ are concentrated around their respective expectations. To this end, first we loosely upper bound the expectations of $X_1$ and $X_2$ as follows. Using Claim~\ref{cl:expectations},
    \[ \E[X_2] \le  \E[X_1] \le \frac{|S|m}{|V_t|} (1 \pm O(\eta)) \le 4 \mu, \]
    where we used that $q \ge 1/2.$
    
    Next, we claim that changing the outcome of one path $P_j$ changes $X_1$ by at most $1$ and $X_2$ by at most $\ell+2$. The first claim is immediate and the second follows from the first and the fact that $P_j$ intersects at most $\ell+1$ paths which are not intersected by any other path $P_{j'}, j \neq j'$. Clearly, for any $s \ge 1,$ there is a set of at most $s$ trials certifying that $X_1 \ge s$ and a set of at most $2s$ trials certifying that $X_2 \ge s$. Let $\tau > 0$ be arbitrary and let $\tau_2 = \frac{1}{2}\tau + \frac{1}{2} C(\ell \sqrt{\mu} + \ell^2)$, where $C$ is a large absolute constant to be chosen implicitly later. Applying Theorem~\ref{thm:talagrand} with $c=\ell+2$ and $r=2$ and recalling that $\E[X_1], \E[X_2] \le 4\mu$, for any $i \in \{1, 2\}$, we have
    \[ \Pr\left[ |X_i - \E[X_i]| > \tau_2\right] \le 4\exp\left( \frac{-\tau_2^2}{256 (\ell+2)^2 (\mu + \tau_2)} \right). \]
    Since $X = X_1 - X_2$, we have
    \[ \Pr[|X - \E[X]| > 2\tau_2] \le \Pr[|X_1 - \E[X_1]| > \tau_2] + \Pr[|X_2 - \E[X_2]| > \tau_2] \le 8 \exp\left( \frac{-\tau_2^2}{256(\ell+2)^2 (\mu + \tau_2)}\right). \]
    Finally, using that $\E[X] = \mu (1 \pm O(\eta))$ and recalling that $2\tau_2 = \tau + C(\ell \sqrt{\mu} + \ell^2)$, we have
    \[ \Pr\left[\left| X - \mu \right|  > \tau + C(\eta \mu + \ell \sqrt{\mu} + \ell^2)\right] \le 8 \exp\left(\frac{-\tau_2^2}{\frac{1}{4}C \ell^2 (\mu + \tau_2)}\right) \le 8\exp\left(\frac{-\tau^2}{C \ell^2 (\mu + \tau)}\right)\]
    as claimed.

    Part~\ref{item:indexdistr} is proved similarly. Let 
    \begin{align*} 
        Y &= |J \cap I|,\\
        Y_1 &= |J_0 \cap I | = |\{ i \in I \mid P_i \text{ is non-degenerate} \}|,\\
        Y_2 &= |(J_0 \setminus J) \cap I| = |\{i \in I \mid P_i \text{ is non-degenerate } \land \exists j \in [m] \setminus \{i\}, P_j \cap P_i \neq \emptyset\}|.
    \end{align*}
    By Lemma~\ref{lem:path-degenerate} and Claim~\ref{cl:prob-collision}, for any $i \in I,$ we have
    \begin{align*}
        \Pr[i \in J] &= \Pr[P_i \text{ is non-degenerate}] \cdot \Pr[P_j \cap P_i = \emptyset, \forall j \in [m] \setminus \{i\} \mid P_i \text{ is non-degenerate}]\\
        &= (1 \pm 3\ell^2 / d) \cdot q (1 \pm O(\eta)) = q ( 1 \pm O(\eta)).
    \end{align*} 
    This implies $\E[Y] = q|I| (1 \pm O(\eta)).$ Observe that $Y = Y_1 - Y_2$. Note that changing the outcome of a single path $P_j$ can change $Y_1$ by at most $1$ and $Y_2$ by at most $\ell+2$. Furthermore, for any $s \ge 1$, there is a set of at most $s$ trials certifying $Y_1 \ge s$ and a set of at most $2s$ trials certifying $Y_2 \ge s$. We loosely bound $\E[Y_2] \le \E[Y_1] \le |I|$. Let $\tau > 0$ be arbitrary and let $\tau_2 = \frac{1}{2}\tau + \frac{1}{2}C(\ell \sqrt{|I|} + \ell^2)$. Applying Theorem~\ref{thm:talagrand} with $c=\ell+2$ and $r=2$, for any $i \in \{1, 2\}$, it holds that
    \[ \Pr[ |Y_i - \E[Y_i]| > \tau_2] \le 4 \exp\left( \frac{-\tau_2^2}{64 (\ell+2)^2 (|I| + \tau_2)} \right). \]
    Hence,
    \[ \Pr[|Y - \E[Y]| > 2\tau_2] \le \Pr[|Y_1 - \E[Y_1]| > \tau_2] + \Pr[|Y_2 - \E[Y_2]| > \tau_2] \le 8 \exp\left( \frac{-\tau_2^2}{64 (\ell+2)^2 (|I| + \tau_2)} \right). \]
    Finally, recalling that $\E[Y] = q|I| (1 \pm O(\eta))$ and $2\tau_2 = \tau + C(\ell \sqrt{|I|} + \ell^2)$, we get
    \[ \Pr[| Y - q|I|| > \tau + C (\eta|I| + \ell \sqrt{|I|} + \ell^2)] < 8\exp\left( \frac{-\tau_2^2}{\frac{1}{4}C \ell^2 (|I| + \tau_2)} \right) \le 8\exp\left( \frac{-\tau^2}{C \ell^2 (|I| + \tau)} \right), \] 
    finishing the proof of the lemma.
\end{proof}

Note that in Lemma~\ref{lem:one-bite-probabilistic}, we require $\eta < 1/C$, for some constant $C$, which implies $m \ge C |V_0| d'/d$. As long as we still have many pairs to connect, we may apply Lemma~\ref{lem:one-bite-probabilistic} directly. The following lemma provides a deterministic statement to be used for one step of our iteration provided we still have many pairs remaining.

\begin{lemma} \plabel{lem:one-bite}
    Let $G$ be a $(d \pm d')$-nearly regular $n$-vertex graph and let $V_0, V_1 \subseteq V(G)$ be such that $G$ is $(V_0, V_1)$-alternating.  Let $m = np$ be a given integer and let $(a_i, b_i)_{i \in [m]}$ be given pairs of vertices with $a_i \in V_0, b_i \in V_1$ and $a_i \neq b_i$, for all $i \in [m]$, such that the multiset $\{ a_i \mid i \in [m]\}$ is $(V_0, V_1, \alpha)$-bi-uniform and the multiset $\{ b_i \mid i \in [m]\}$ is $(V_1, V_0, \alpha)$-bi-uniform. Furthermore, let $U \subseteq V(G)$ be an $\alpha$-uniform set with respect to $G$ such that $|U \cap V_0| = |U \cap V_1| \le n/20,$ $U$ contains the set $\bigcup_{i \in [m]} \{a_i, b_i\}$ and $G \setminus U$ is $(V_0 \setminus U, V_1 \setminus U)$-mixing in time $t_0$. Let $\ell \ge 10t_0$ be a given odd integer. Assume the following inequalities hold: $p \ge \alpha^{3/4}, d' \le \alpha d, p \le \ell^{-8}, \alpha > d^{-1/4}, d \ge \ell^8$ and $\ell \ge \log n$. There is an absolute constant $C$ such that there is a set $J \subseteq [m]$ and a collection of paths $(P_i)_{i \in J},$ where for $i \in J,$ $P_i$ is an $a_i-b_i$ path of length $\ell$, satisfying the following with $\alpha' = C\alpha$:
    \begin{itemize}
        \item $\Vint(P_i) \cap \Vint(P_j) = \emptyset, \forall i, j \in J, i \neq j.$
        \item $\Vint(P_i) \cap U = \emptyset, \forall i \in J.$
        \item $|[m] \setminus J| \le 4 \ell^2 m^2/n + \alpha' n.$ 
        \item For all $t \in [\ell-1]$, the set $W_t = \{ P_i(t) \mid i \in J\}$ is $(V_t, V_{t+1}, \alpha')$-bi-uniform.
        \item The multiset $\{a_i \mid i \in [m] \setminus J\}$ is $(V_0, V_1, \alpha')$-bi-uniform and the multiset $\{b_i \mid i \in [m] \setminus J\}$ is $(V_1, V_0, \alpha')$-bi-uniform.
    \end{itemize}
\end{lemma}
\begin{proof}
    When $n$ is bounded, the statement is trivial with $C$ sufficiently large and $J=\emptyset$, so let us assume that $n$ is sufficiently large.

    Denote $G' = G \setminus U, V_0' = V_0 \setminus U, V_1' = V_1 \setminus U$ and $n' = |V(G')| = n - |U| \ge n/2$. Note that $|V_0'| = |V_1'| = |V_0| - |U| \frac{|V_0|}{n}$. By Lemma~\ref{lem:remove-set}, $G'$ is $(d_1 \pm d_1')$-nearly regular with $d_1 = d(1 - |U| / n) \ge 0.9 d$ and $d_1' = 4\alpha d,$ since $4 \alpha d \ge d'(1 + |U|/n) + \alpha(d + d')$. Furthermore, by assumption, $G'$ is $(V_0', V_1')$-mixing in time $t_0$.

    For $i \in [m],$ let $A_i = N_G(a_i) \setminus U \subseteq V_1'$ and $B_i = N_G(b_i) \setminus U \subseteq V_0'$. By Lemma~\ref{lem:remove-set}, we have $|A_i| = |N_G(a_i) \setminus U| = d_1 \pm d_1'$ and, analogously, $|B_i| = d_1 \pm d_1'.$ Since the multiset $\{a_i \mid i \in [m]\}$ is $(V_0, V_1, \alpha)$-bi-uniform, each vertex $u \in V_1'$ appears in $m \bar{d}(G) / |V_0| \pm \alpha \bar{d}(G) = m d / |V_0| \pm d_1'=md_1/|V_0'|\pm d_1'=md_1/|V_1'|\pm d'_1$ sets $A_i, i \in [m]$. Analogously, each vertex $u \in V_0'$ appears in $md_1 / |V_0'| \pm d_1'$ sets $B_i, i \in [m].$ With the aim of applying Lemma~\ref{lem:one-bite-probabilistic}, for each $i \in [m],$ let $P_i' \sim \cP^{\ell-2}_{G'}(A_i, B_i)$ be independently sampled. Let $J = \{ i \in [m] \mid P_i' \text{ is non-degenerate} \land P_i' \cap P_j' = \emptyset, \forall j \in [m], j \neq i \}$. For each $i \in J,$ let $P_i$ be the path obtained from $P_i'$ by appending $a_i$ to the beginning and $b_i$ to the end. Observe that for each $i \in J,$ the path $P_i$ is an $a_i-b_i$ path of length $\ell$ and the paths $(P_i)_{i \in J}$ are internally vertex disjoint and their internal vertex sets do not intersect $U$, which shows that the first two points are satisfied. It remains to show that with positive probability, the last three points  hold.

    Let $C_{\ref{lem:one-bite-probabilistic}}$ be the constant given by Lemma~\ref{lem:one-bite-probabilistic}. Let $q = e^{-(\ell-1)^2 m / n'}$ and observe that $q \ge 1/2$ since $m \le \ell^{-8} n, n' \ge n/2$ and $\ell$ is sufficiently large.  Denote
    \[ \eta \coloneqq \max \{ d_1' |V_0'| / (m d_1), (\ell-2)^2 / d_1 \} = \frac{d_1' |V_0'|}{m d_1} \in [2 \alpha n / m, 8 \alpha n / m], \]
    where we used that $d_1'=4\alpha d$, $\alpha > d^{-1/4}$ and $m \le n / \ell^8$, as well as that $\frac{|V_0'|}{d_1}=\frac{|V_0|}{d}$. Furthermore, observe that $\eta \le 8 \alpha^{1/4} < 1/C_{\ref{lem:one-bite-probabilistic}}$ since $m \ge \alpha^{3/4} n$ and $\alpha = o(1)$.

    Note that $\eta m \le 8 \alpha n$ and $\max\{\ell \sqrt{m}, \ell^2\} \le \sqrt{n} \le \alpha n,$ where we used that $\alpha \ge d^{-1/4} \ge n^{-1/4}.$ Hence, applying Lemma~\ref{lem:one-bite-probabilistic}~\ref{item:indexdistr} with $I = [m]$ and $\tau = C/8 \cdot \alpha n$ for a sufficiently large absolute constant $C$, we have
    \[ \Pr[ ||J| - q m| > C/4 \cdot \alpha n] \le 8\exp\left( \frac{-  C^2 \alpha^2 n^2}{64 C_{\ref{lem:one-bite-probabilistic}} \ell^2(m + C \alpha n)}\right) < n^{-10}, \]
    where we used that $m \le  n / \ell^8$ and $\alpha > d^{-1/4} \ge n^{-1/4}$.
    Let us assume that
    \begin{equation} \label{eq:J concentrates}
        ||J| - qm| \le C/4 \cdot \alpha n.
    \end{equation}
    Then, in particular, using that $q = e^{-(\ell-1)^2 m / n'} \ge 1 - 4\ell^2 m / n,$ we have
    \[ |[m] \setminus J| \le (1-q)m + C/4 \cdot \alpha n \le 4 \ell^2 m^2/n + C \alpha n, \]
    proving the third point.

    Let $\tau' = (C/4) \alpha d$. Let $t \in [1, \ell-1]$ be arbitrary and note that for any $i \in J,$ $P_i(t) = P_i'(t-1),$ so $P_i(t) \in V_t$. Let $v \in V_{t+1}$ be arbitrary. Let $S = N_G(v) \setminus U \subseteq V_t'$ and recall that $|S| = d_1 \pm d_1',$ so in particular $|S| \ge d/2$. Let $\mu = q m |S| / |V_t'| \in [md / (4n), 4md / n]$. Note that $C_{\ref{lem:one-bite-probabilistic}}(\eta \mu + \ell \sqrt{\mu} + \ell^2) \le 3C_{\ref{lem:one-bite-probabilistic}} \eta \mu \le 96 C_{\ref{lem:one-bite-probabilistic}} \alpha d$, since $\alpha > d^{-1/4}, m \le n $ and $d \ge \ell^8$. Denote $W_t = \{ P_i(t) \mid i \in J\}$. Recalling that $|S| = d_1 \pm d_1'$ and $d_1' = 4 \alpha d,$ using Lemma~\ref{lem:one-bite-probabilistic}~\ref{item:pathdistr} with $\tau=\tau'/2$, we get
    \[ \Pr\left[ \big| | W_t \cap S | - q d_1  \cdot \frac{m}{|V_t'|} \big|  > \tau'\right] \le 8 \exp\left(\frac{-(\tau')^2}{4 C_{\ref{lem:one-bite-probabilistic}} \ell^2 (\mu + \tau')}\right) < n^{-10}, \]
    where we used that $m \le n,  d \ge \ell^8 \ge \log^8 n$ and $\alpha > d^{-1/4}$. Observe that $q d_1 \cdot m / |V_t'| = q d m / |V_t|$.

    Union bounding over all $t \in [1, \ell-1]$ and all $v\in V_{t+1}$, it follows that with probability $1 - o(1)$, for all such $t$ and~$v$ we have $|W_t\cap N_G(v)|=qdm/|V_t|\pm C\alpha d/4$. Since by (\ref{eq:J concentrates}), $|W_t|=|J|=qm \pm C\alpha n/4$, we have $qdm/|V_t| \pm C\alpha d/4=|W_t|d/|V_t| \pm 3C\alpha d /4=|W_t|\bar{d}(G)/|V_t| \pm d' \pm 3C\alpha \bar{d}(G)/4 \pm 3C\alpha d'/4=|W_t|\bar{d}(G)/|V_t| \pm C\alpha \bar{d}(G)$.
    Hence, the set $W_t$ is $(V_t, V_{t+1}, C \alpha)$-bi-uniform with respect to $G$. 

    Similarly, we argue that the last point holds with high probability. To this end, consider an arbitrary vertex $u \in V_1$ and let $I_u = \{ i \in [m] \mid u \in N_G(a_i)\}.$ Since the multiset $\{a_i \mid i \in [m]\}$ is $(V_0, V_1, \alpha)$-bi-uniform, we have $|I_u| = m \bar{d}(G) / |V_0| \pm \alpha \bar{d}(G) = md / |V_0| \pm 4\alpha d$. Note that $\eta |I_u| \le 8\alpha n /m \cdot 4 m d / n = 32 \alpha d$, while $\ell \sqrt{|I_u|} < \ell \sqrt{2d} \le \alpha d$ and $\ell^2 \le \alpha d.$ Hence, letting $\tau'=(C/4)\alpha d$ and using  Lemma~\ref{lem:one-bite-probabilistic}~\ref{item:indexdistr} with $\tau=\tau'/2$, we have
    \[ \Pr[||J \cap I_u| - q \cdot m d / |V_0|| > \tau'] \le 8\exp\left( \frac{-(\tau')^2}{4 C_{\ref{lem:one-bite-probabilistic}} \ell^2 (m d / |V_0| + \tau')}\right) < n^{-10}, \]
    where we used that $\alpha > d^{-1/4}, d \ge \ell^8 \ge \log^8 n$ and $m \le n$. We obtain the analogous statement for all $u \in V_0$.
    
    Note that if $|J \cap I_u| = q md / |V_0| \pm \tau',$ then $|I_u \setminus J| = md / |V_0| \pm 4 \alpha d - (qmd / |V_0| \pm \tau') = (1 - q)md / |V_0| \pm (C/3) \alpha \bar{d}(G) .$ On the other hand, by (\ref{eq:J concentrates}),
    \[ \frac{|[m] \setminus J| \cdot \bar{d}(G)}{|V_0|} = \frac{((1-q)m \pm C/4 \cdot \alpha n) \cdot \bar{d}(G)}{|V_0|} = \frac{(1-q)md}{|V_0|} \pm (2C/3) \alpha \bar{d}(G),\]
    where we used that $|V_0| \ge n/2, m \le n$ and $d' \le \alpha d$. It follows that as long as $|J\cap I_u|=qmd/|V_0|\pm \tau'$, we have $\left||I_u\setminus J|-\frac{|[m] \setminus J| \cdot \bar{d}(G)}{|V_0|}\right|\leq C\alpha \bar{d}(G)=\alpha' \bar{d}(G)$. Note that if $S=\{a_i \mid i\in [m]\setminus J\}$, then $|I_u\setminus J|=|N(u)\cap S|$ and $|[m]\setminus J|=|S|$.
    
    Hence, union bounding over all vertices $u$, it follows that with probability $1 - o(1)$, $S$ is $(V_0,V_1,\alpha')$-bi-uniform. A symmetric argument shows that almost surely the multiset $\{b_i \mid i\in [m]\setminus J\}$ is $(V_1,V_0,\alpha')$-bi-uniform.
    
    We conclude that with high probability, all the requirements are satisfied, implying the existence of the desired paths and thus finishing the proof of the lemma.    
\end{proof}

Once we have few paths left to connect, we cannot apply Lemma~\ref{lem:one-bite-probabilistic} directly. Instead, we attempt to create many paths connecting the same pair so that one of them is disjoint from the others with high probability and we keep an arbitrary such path. Since the number of pairs to connect is small, to obtain that the set of vertices appearing in the paths is sufficiently uniform, we only require an upper bound on the number of times these paths hit the neighbourhood of any particular vertex. The last step of the iteration is given by the following lemma.

\begin{lemma} \plabel{lem:last-step}
    Let $G$ be a $(d \pm d')$-nearly regular $n$-vertex graph and let $V_0, V_1 \subseteq V(G)$ such that $G$ is $(V_0, V_1)$-alternating and assume that $n$ is sufficiently large. Let $m = np$ be a given integer and let $(a_i, b_i)_{i \in [m]}$ be given pairs of vertices with $a_i \in V_0, b_i \in V_1$ and $a_i \neq b_i$, for all $i \in [m]$, such that the multiset $\{ a_i \mid i \in [m]\}$ is $(V_0, V_1, \alpha)$-bi-uniform and the multiset $\{ b_i \mid i \in [m]\}$ is $(V_1, V_0, \alpha)$-bi-uniform. Furthermore, let $U \subseteq V(G)$ be an $\alpha$-uniform set with respect to $G$ such that $|U \cap V_0| = |U \cap V_1| \le n/20$, $U$ contains the set $\bigcup_{i \in [m]} \{a_i, b_i\}$ and $G \setminus U$ is $(V_0 \setminus U, V_1 \setminus U)$-mixing in time $t_0$. Let $\ell \ge 10t_0$ be a given odd integer.
    
    Assume the following inequalities hold: $p \le \alpha^{3/4}, d' \le \alpha d, \alpha \le \ell^{-40}, \alpha \ge d^{-1/4}$ and $\ell \ge \log n$. Then, there is a collection of paths $(P_i)_{i \in [m]},$ where for $i \in [m],$ $P_i$ is an $a_i-b_i$ path of length $\ell$, satisfying the following:    
    \begin{itemize}
        \item $\Vint(P_i) \cap \Vint(P_j) = \emptyset, \forall 1 \le i < j \le m$,
        \item $\Vint(P_i) \cap U = \emptyset, \forall i \in [m]$,
        \item For all $t \in [\ell-1]$, the set $W_t = \{ P_i(t) \mid i \in [m]\}$ is $(V_t, V_{t+1}, \alpha^{2/3})$-bi-uniform.
    \end{itemize}    
\end{lemma}
\begin{proof}
    As in the proof of Lemma~\ref{lem:one-bite}, denote $G' = G \setminus U, V_0' = V_0 \setminus U, V_1' = V_1 \setminus U$ and $n' = |V(G')| = n - |U| \ge n/2$. Note that $|V_0'| = |V_1'| = |V_0| - |U| \frac{|V_0|}{n}$. By Lemma~\ref{lem:remove-set}, $G'$ is $(d_1 \pm d_1')$-nearly regular with $d_1 = d(1 - |U| / n) \ge 0.9 d$ and $d_1' = 4\alpha d,$ since $4 \alpha d \ge d'(1 + |U|/n) + \alpha(d + d').$ Furthermore, by assumption, $G'$ is $(V_0', V_1')$-mixing in time $t_0$.

    For $i \in [m],$ let $A_i = N(a_i) \setminus U, B_i = N(b_i) \setminus U$. By Lemma~\ref{lem:remove-set}, we have $|A_i| = |N_G(a_i) \setminus U| = d_1 \pm d_1'$ and, analogously, $|B_i| = d_1 \pm d_1'.$ Because the multisets $\{a_i \mid i \in [m]\}$ and $\{b_i \mid i \in [m]\}$ are $(V_0, V_1, \alpha)$-bi-uniform and $(V_1, V_0, \alpha)$-bi-uniform, respectively, we have that each vertex $u \in V_1'$ appears in $md / |V_0| \pm d_1'=md_1/|V_0'|\pm d_1'=md_1/|V_1'|\pm d'_1$ sets $A_i, i \in [m]$ and each vertex $u \in V_0'$ appears in $md_1 / |V_0'| \pm d_1'$ sets $B_i, i \in [m]$. Let $r \coloneqq 8C \ell^3$, where $C$ is a large constant to be chosen implicitly later.
    
    Recall that in order to apply Lemma~\ref{lem:one-bite-probabilistic}, we require a lower bound on the number of random walks to be sampled (as we require $\frac{md}{d'|V_0|}$ in that lemma to be large). To this end, denote $m' = r \cdot \alpha^{3/4} n$ and note that $rm \le m'$. We add $z \coloneqq m' - r m$ dummy pairs $X_i \subseteq V_0', Y_i \subseteq V_1'$ of size $|X_i| = |Y_i| = d_1$ such that for every vertex $v \in V_0'$,
    \[ r |\{ i\in [m] \mid v \in B_i\}| + |\{ i \in [z] \mid v \in X_i \}| = m' d_1 / |V_0'| \pm 2r d_1', \]
    and analogously, for every $v \in V_1',$ 
    \[ r |\{ i\in [m] \mid v \in A_i\}| + |\{ i \in [z] \mid v \in Y_i \}| = m' d_1 / |V_1'| \pm 2r d_1'. \]
    It is not hard to see this is always possible by iteratively defining the sets $X_i, Y_i$ to consist of least popular vertices.

    With the aim of applying Lemma~\ref{lem:one-bite-probabilistic}, for each $i \in [m]$ and $j \in [r]$, we independently sample a walk $P_{i,j}' \sim \cP_{G'}^{\ell-2}(A_i, B_i)$. Furthermore, for each $i \in [z],$ independently sample a walk $Q_i \sim \cP_{G'}^{\ell-2}(X_i, Y_i)$. Let $\calF$ be the set of these walks which are non-degenerate and vertex disjoint from all the other sampled walks. Let
    \[ \eta = \max\{ 2rd_1' |V_0'| / (m'd_1), (\ell-2)^2 / d_1 \} = \max\{ 8 \alpha^{1/4} |V_0| / n, (\ell-2)^2 / d_1 \} = 8\alpha^{1/4} |V_0| / n  \in [4\alpha^{1/4}, 8\alpha^{1/4}], \]
    where we used that $d_1 \ge 0.9d, \frac{|V_0'|}{d_1} = \frac{|V_0|}{d},$ $|V_0| \in [n/2, n], \alpha > d^{-1/4}$ and $d \ge \alpha^{-4} \ge \ell^{160}$. Note that $\eta < 1 / (8C)$ since $\alpha = o(1)$.

    Consider an arbitrary vertex $v \in V(G)$ and index $t \in [0, \ell-2]$. Denote $S_v = N_G(v) \setminus U, q = e^{-(\ell-1)^2 m' / n'}$ and $\mu = q m' |S_v| / |V_0'|$. Observe that $q \ge 1/2$ since $m'= 8C\ell^3\alpha^{3/4} n, \alpha \le \ell^{-40}$ and $n' \ge n/2$. Furthermore, note that $\mu \in [(1/4) r \alpha^{3/4} d, 4 r\alpha^{3/4} d]$. Using $\alpha \leq \ell^{-40},$ we have $\eta \mu \le 32 \alpha r d \leq \alpha^{3/4} d$. Furthermore, $\ell \sqrt{\mu} \le \ell \cdot 2 \sqrt{r} \alpha^{3/8} \sqrt{d} \le \alpha^{3/4} d$ and $\ell^2 \le \alpha^{3/4} d.$ If $G$ is bipartite with parts $V_0, V_1$ and $v \in V_{t+1},$ then for any $P \in \cF,$ $P(t) \not\in S_v,$ since $P(t) \in V_{t+1}$ and $S_v \subseteq V_t$. Otherwise, setting $\tau = 4r \alpha^{3/4} d$ and using Lemma~\ref{lem:one-bite-probabilistic}~\ref{item:pathdistr}, we get
    \begin{equation*} \plabel{eq:prob-S-last-step}
        \Pr[|\{ P \in \calF \mid P(t) \in S_v \}| \ge 16r \alpha^{3/4}d ] \le 8\exp\left( \frac{- (4r \alpha^{3/4} d)^2}{C_{\ref{lem:one-bite-probabilistic}}\ell^2 (4r \alpha^{3/4} d + 4r \alpha^{3/4}d)} \right) \le \exp(- \alpha^{3/4} d) < n^{-10},
    \end{equation*}
    where we used that $r = 8C \ell^3$, $\alpha > d^{-1/4}, d \ge \ell^{100}$ and $\ell \ge \log n$. Let us assume that for every $v \in V(G)$ and every $t \in [0, \ell-2],$ we have
    \begin{equation} \plabel{eq:last-step-uniform}
        |\{ P \in \calF \mid P(t) \in S_v \}| < 16r \alpha^{3/4} d,
    \end{equation}
    which by a union bound holds with high probability.

    Next, we show that with high probability, for any $i \in [m],$ at least one path $P'_{i,j}, j \in [r]$ is in $\cF$. To this end, observe that $\eta r \le 8 \alpha^{1/4} r \le r / (16C_{\ref{lem:one-bite-probabilistic}}), \ell \sqrt{r} \le r / (16C_{\ref{lem:one-bite-probabilistic}})$ and $\ell^2 \le r / (16C_{\ref{lem:one-bite-probabilistic}})$. Now, fix $i \in [m]$. Applying Lemma~\ref{lem:one-bite-probabilistic} \ref{item:indexdistr} with the index set $I$ corresponding to the set of paths $\{P_{i, j}' \mid j \in [r]\}$ and $\tau = r / 4,$ we see that the probability that none of these paths is in $\calF$ is at most $8\exp\left( -\frac{\tau^2}{C_{\ref{lem:one-bite-probabilistic}} \ell^2(r + \tau)} \right) < n^{-10},$ where we used that $q \ge 1/2, r = 8C \ell^3$, $C$ is sufficiently large and $\ell \ge \log n$. Hence, by a union bound, with high probability, for every $i \in [m]$ there exists $j \in [r]$ such that $P'_{i,j} \in \cF$ and we shall assume that this holds.
    
    For each $i \in [m],$ pick an arbitrary path $P'_{i, j_i}$ in $\calF$ and let $P_i$ be the path obtained by appending $a_i$ to the beginning and $b_i$ to the end of $P'_{i,j_i}.$ Note that for $i \in [m],$ since $P_{i,j_i}' \in \cF$, the paths $P_i$ are internally vertex disjoint and their internal vertices are disjoint from $U$, which combined with the fact that $a_i, b_i \in U$ and $a_i\neq b_i$ further implies that each $P_i$ is non-degenerate. Moreover, clearly $P_i$ has length $\ell$. It remains to prove the last point.

    For $t \in [1, \ell-1]$, let $W_t = \{ P_i(t) \mid i \in [m] \}.$ For any $v \in V(G),$ we have $0 \le |W_t \cap N(v)| \le 16r \alpha^{3/4}d$ by (\ref{eq:last-step-uniform}), implying that $W_t$ is $(V_t, V_{t+1}, \alpha^{2/3})$-bi-uniform. Indeed, note that $\alpha^{2/3}\bar{d}(G) \ge |W_t| \bar{d}(G) / |V_0|,$ where we used that $|W_t| = m \le \alpha^{3/4}n, |V_0| \ge n/2$ and $\alpha = o(1)$. On other hand, $16r \alpha^{3/4} d \le \alpha^{2/3} \bar{d}(G)$ because $r = 8 C \ell^3, \alpha \le \ell^{-40}$, $\ell \ge \log n = \omega(1)$ and $\bar{d}(G) \geq d/2$. We conclude that the paths $P_i$ satisfy all the requirements, thus finishing the proof.
\end{proof}

Finally, using Lemmas~\ref{lem:one-bite}~and~\ref{lem:last-step} we prove Lemma~\ref{lem:connections}.
\begin{proof}[Proof of Lemma~\ref{lem:connections}]
    We shall iteratively apply Lemma~\ref{lem:one-bite} in steps $k=1, \dots$ to find the desired paths in batches and once the number of remaining pairs to connect becomes small enough, we will find them in one step using Lemma~\ref{lem:last-step}. Let $C$ be a sufficiently large constant to be chosen implicitly later and for $k \ge 1,$ define $\alpha_k = (C\ell)^{k-1} \alpha$. Denote $K = 3 \log (\log(\alpha^{-1}) / \log(p^{-1}))$ and observe that for $k \le K+1,$ we have $\alpha_k \le (C \ell)^K \alpha \le \ell^{2K} \alpha \le \alpha^{15/16}.$ Indeed, the first inequality holds since $\ell \ge C$ so it remains to justify that $\ell^{2K} \le \alpha^{-1/16},$ or equivalently $\log (\alpha^{-1}) - 96 \log (\log(\alpha^{-1}) / \log(p^{-1})) \log \ell \ge 0$. The derivative in $\alpha$ of the left hand side is negative for $\alpha \in (0, \ell^{-400}]$, so it is enough to prove the inequality for $\alpha = \ell^{-400}$ which is our upper bound on $\alpha.$ In this case, we have $\log (\ell^{400}) - 96 \log (\log(\ell^{400}) / \log(p^{-1})) \log \ell \ge 400 \log \ell - 96 \log (400 \log \ell / (8 \log \ell)) \log \ell > 0.$

    At step $k$, we are given a set $I_k \subseteq [m]$ with $|I_k| \le \max\{ p^{(3/2)^{k-1}} n, \alpha^{45/64} n\}$ and already defined paths $P_i$ for each $i \in [m] \setminus I_k$ such that $P_i$ is an $a_i-b_i$ path of length $\ell$ disjoint from $U$ and the paths $P_i, i \in [m] \setminus I_k$ are pairwise internally vertex disjoint. Moreover, for each $t \in [\ell-1]$, the set $\{ P_i(t) \mid i \in [m] \setminus I_k\}$ is $(V_t, V_{t+1}, \alpha_k / \ell)$-bi-uniform. Finally, the multisets $\{ a_i \mid i \in I_k \}$ and $\{ b_i \mid i \in I_k\}$ are $(V_0, V_1, \alpha_k)$-bi-uniform and $(V_1, V_0, \alpha_k)$-bi-uniform, respectively. We set $I_1 = [m]$ and note that by assumption all the requirements are satisfied for $k=1$ since $\alpha_1 = \alpha,$ none of the paths have been defined and the empty set of vertices is $\beta$-uniform for any $\beta \ge 0.$

    Now, consider an arbitrary step $k \ge 1$. Let $U_k = U \cup \bigcup_{i \in [m] \setminus I_k} \Vint(P_i)$. Note that $|U_k \cap V_0| = |U_k \cap V_1| \le |U \cap V_1| + m(\ell-1) \le |U| + m \ell \le n/50$, where the equality is justified by the fact that $G$ is $(V_0, V_1)$-alternating, $|U \cap V_0| = |U \cap V_1|$ and $\ell$ is odd, while the inequality holds since $|U| \le n/100, m \le n / \ell^{8}$ and $\ell = \omega(1)$.
    
    We claim that $U_k$ is $\alpha_k$-uniform with respect to $G$. If $k=1,$ then $U_k = U$ which is $\alpha$-uniform by assumption. If $k \ge 2,$ by assumption, for each $t \in [\ell-1],$ the set $\{ P_i(t) \mid i \in [m] \setminus I_k\}$ is $(V_t, V_{t+1}, \alpha_k / \ell)$-bi-uniform. If $V_0 = V_1,$ then $U_k$ is $\alpha_k$-uniform since $\alpha + (\ell-1) \alpha_k / \ell \le \alpha_k$. On the other hand, if $G$ is bipartite with parts $V_0$ and $V_1$ of equal size, then the set $(U \cap V_0) \cup \bigcup_{t \in [\ell-1], t \text{ even}} \{P_i(t) \mid i\in [m]\setminus I_k\}$ is $(V_0, V_1, \alpha_k)$-bi-uniform and analogously the set $(U \cap V_1) \cup \bigcup_{t \in [\ell-1], t \text{ odd}} \{P_i(t) \mid i\in [m]\setminus I_k\}$ is $(V_1, V_0, \alpha_k)$-bi-uniform, and since these sets are of equal size and their union is $U_k,$ it follows that $U_k$ is $\alpha_k$-uniform. Observe that if $k \le K+1,$ then recalling that $\alpha_k \le \alpha^{15/16}$, $U_k$ is balanced and $|U_k| \le |U| + m \ell$, by assumption, it follows that $G \setminus U_k$ is $(V_0 \setminus U_k, V_1 \setminus U_k)$-mixing in time $t_0$.

    Denote $m_k = |I_k|.$ First, suppose that $m_k > \alpha^{45/64} n$. Then, $|I_k| \le p^{(3/2)^{k-1}} n$ implies $k \le K$, where we used that $\log(3/2) > 1/3$. Denote by $C_{\ref{lem:one-bite}}$ the constant provided by Lemma~\ref{lem:one-bite} and let $\alpha' = C_{\ref{lem:one-bite}} \alpha_k$.
    
    We apply Lemma~\ref{lem:one-bite} to the graph $G$ with the set of pairs $\{(a_i, b_i) \mid i \in I_k\}$, $U = U_k,$ $\alpha = \alpha_k$ and the other parameters directly translated. Let us verify all the inequalities needed for this application. We have $m_k/n \ge \alpha^{45/64} \ge \alpha_k^{3/4}$ since $\alpha_k \le \alpha^{15/16},$ and furthermore $d' \le \alpha d \le \alpha_k d, m_k/n \le m/n \le \ell^{-8}, \alpha_k \ge \alpha > d^{-1/4}, d \ge \alpha^{-4} \ge \ell^8$ and $\ell \ge \log n$.
    
    We obtain a set $J \subseteq I_k$ such that $|I_k \setminus J| \le 4 \ell^2 m_k^2 / n + \alpha' n$ and a collection of paths $(P_i^k)_{i \in J},$ such that $P_i^k$ is an $a_i-b_i$ path of length $\ell,$ these paths are internally vertex disjoint and their internal vertices are disjoint from $U_k$, that is, their internal vertices are disjoint from $U$ and all the internal vertices in the paths found in previous steps. Moreover for each $t \in [\ell-1],$ the set $\{ P_i^k(t) \mid i \in J\}$ is $(V_t, V_{t+1}, \alpha')$-bi-uniform. Finally, the  multisets $\{ a_i \mid i \in I_k \setminus J\}$ and $\{ b_i \mid i \in I_k \setminus J\}$ are $(V_0, V_1, \alpha')$-bi-uniform and $(V_1, V_0, \alpha')$-bi-uniform, respectively.

    For $i \in J$, let $P_i = P_i^k$ and let $I_{k+1} = I_k \setminus J.$ Note that $4\ell^2 m_k^2 / n \le 4\ell^2 p^{2\cdot (3/2)^{k-1}} n \le (1/2) p^{(3/2)^k} n,$ where we used that $k \ge 1, p \le \ell^{-8}$ and $\ell = \omega(1)$. Furthermore, $\alpha' n \le C_{\ref{lem:one-bite}} \alpha^{15/16} n \le (1/2) \alpha^{45/64} n,$ since $\alpha = o(1)$. Combining, we have $|I_{k+1}| \le 4\ell^2 m_k^2 / n + \alpha' n \le \max\{ p^{(3/2)^k} n, \alpha^{45/64} n\}$. Consider an arbitrary index $t \in [\ell-1]$. Since the set $\{ P_i(t) \mid i \in [m] \setminus I_k\}$ is $(V_t, V_{t+1}, \alpha_k)$-bi-uniform and the set $\{ P_i(t) \mid i \in J\}$ is $(V_t, V_{t+1}, \alpha')$-bi-uniform and $\alpha_k + \alpha' \le \alpha_{k+1}/\ell,$ the set $\{ P_i(t) \mid i \in [m] \setminus I_{k+1}\}$ is $(V_t, V_{t+1}, \alpha_{k+1}/\ell)$-bi-uniform. Finally, note that the multisets $\{ a_i \mid i \in I_{k+1}\}$ and $\{ b_i \mid i \in I_{k+1}\}$ are $(V_0, V_1, \alpha_{k+1})$-bi-uniform and $(V_1, V_0, \alpha_{k+1})$-bi-uniform respectively since $I_{k+1} = I_k \setminus J$ and $\alpha_{k+1} \ge \alpha'$. Thus, we may proceed to step $k+1$.

    Now, suppose that $m_k \le \alpha^{45/64} n$ and either $k=1$ or $m_{k-1} > \alpha^{45/64} n$. If $k = 1$, then $\alpha_k = \alpha \le \alpha^{15/16}$ and, otherwise, we have $k \le K+1$, so $\alpha_k \le \alpha^{15/16}.$ We apply Lemma~\ref{lem:last-step} to the graph $G$ with the remaining pairs $(a_i, b_i)_{i \in I_k}, U = U_k, \alpha = \alpha^{15/16}$ and the other parameters directly translated. Let us verify the necessary inequalities. We have $m_k/n \le \alpha^{45/64} = (\alpha^{15/16})^{3/4}, d' \le \alpha d \le \alpha^{15/16}d, \alpha^{15/16} \le \ell^{-40}, \alpha^{15/16} \ge \alpha \ge d^{-1/4}$ and $\ell \ge \log n$.
    
    Thus, for each $i \in I_k,$ we get an $a_i-b_i$ path $P_i^k$, these paths are internally vertex disjoint and disjoint from $U_k$. Moreover, for each $t \in [\ell-1]$, the set $\{ P_i^k(t) \mid i \in I_k \}$ is $(V_t, V_{t+1}, (\alpha^{15/16})^{2/3})$-bi-uniform. For $i \in I_k,$ let $P_i = P_i^k$. For $t \in [\ell-1]$, since the set $\{ P_i(t) \mid i \in [m] \setminus I_k\}$ is $(V_t, V_{t+1}, \alpha_k / \ell)$-bi-uniform and $\{ P_i(t) \mid i \in I_k \}$ is $(V_t, V_{t+1}, (\alpha^{15/16})^{2/3})$-bi-uniform, it follows that $\{ P_i(t) \mid i \in [m]\}$ is $(V_t, V_{t+1}, \alpha^{1/2})$-bi-uniform as $\alpha_k / \ell + \alpha^{15/16 \cdot 2/3} \le \alpha^{15/16} + \alpha^{5/8} \le \alpha^{1/2}.$ This concludes the proof of the lemma.
\end{proof}

\section{Finding the absorbers} \label{sec:absorbers}
The main purpose of this section is to find a reservoir set through which we can connect many pairs of vertices and an absorbing structure for this reservoir. Crucially, we require that the vertex set of the absorbing structure including the reservoir is a uniform set. We start with a formal definition of the connecting property we shall use.

\begin{defn}[Connecting set] \plabel{def:connecting-set}
    Let $G$ be a graph. We say that a set $V \subseteq V(G)$ is $(D, \ell)$-connecting in $G$ if for every sequence $x_1, \dots, x_r, y_1, \dots, y_r$ of (not necessarily distinct) vertices outside of $V$ such that each vertex in $V$ has at most $D$ neighbours in the multiset $\{x_1, \dots, x_r, y_1, \dots, y_r\}$, there is a collection of internally vertex-disjoint paths $(P_i)_{i \in [r]}$ such that, for each $i \in [r],$ $P_i$ is an $x_i-y_i$ path in $G$ through $V$ of length at most $\ell$.
\end{defn}

We need the following lemma stating that a large random set is connecting with high probability. This statement was proved by Chakraborti, Janzer, Methuku and Montgomery~\cite{chakraborti2025edge} in the case $\explet = \Theta(1/ (\log n)^2)$ adapting a lemma of Buci\'{c} and Montgomery~\cite{bucic2022towards} where edge-disjoint paths are constructed.

\begin{restatable}{lemma}{lemconnectingingrestate} \plabel{lem:randomsetconnecting}
    Let $G$ be an $n$-vertex $(d\pm d')$-nearly-regular $\explet$-expander, where $d'\leq d/4$ and $\explet<\frac{1}{(\log n)^2}$. Assume that $d\ge p^{-4}\explet^{-6}(\log n)^{22}$ and $p\leq 1$. Let $V$ be a $p$-random subset of $V(G)$. Then, with probability $1-o(1)$, $V$ is $(D,\ell)$-connecting for $D=\frac{p^9 \explet^{13} d }{(\log n)^{51}}$ and $\ell=\frac{5}{2}\explet^{-1}(\log n)^4$.
\end{restatable}

The proof of Lemma~\ref{lem:randomsetconnecting} is a trivial modification of the corresponding proof in~\cite{chakraborti2025edge} so we defer it to the appendix.

We will need the following simple consequence of Lemma \ref{lem:randomsetconnecting} which finds a connecting set avoiding a small uniform subset of the vertices.

\begin{lemma} \plabel{lem:randomset_avoiding_connecting}
    Let $G$ be an $n$-vertex $(d\pm d')$-nearly-regular $\explet$-expander, where $d'\leq \explet d/100$ and $\explet<\frac{1}{(\log n)^2}$. Let $U$ be an $\alpha$-uniform subset of $V(G)$ with $|U|\leq \explet n/100$ and $\alpha\leq \explet/100$. Assume that $d\ge p^{-10}\explet^{-100}$ and $p\leq 1$. Let $R$ be a $p$-random subset of $V(G)\setminus U$. Then, with probability $1-o(1)$, $R$ is $(D,\ell)$-connecting in $G$ for $D=\frac{p^{9} \explet^{13} d }{(\log n)^{52}}$ and $\ell=5\explet^{-1}(\log n)^4$.
\end{lemma}

\begin{proof}
    Note that we may assume that $n$ is sufficiently large.

    By Lemma \ref{lem:remove-set}, $G\setminus U$ is a $(d_2\pm d'_2)$-nearly regular $\explet_2$-expander, where $d_2=d(1-|U|/n)\geq d/2$, $d'_2=d'(1+|U|/n)+\alpha(d+d')\leq d/8\leq d_2/4$ and $\explet_2=\explet-4(|U|/n+\alpha+d'/d)\geq 5\explet/8$. Let $n_2$ be the number of vertices in $G\setminus U$, and note that $n_2\geq n/2$.

    Let $R_1\cup R_2$ be a random partition of $R$, where each $v\in R$ belongs to each $R_i$ with probability $1/2$, independently over all vertices $v$. Clearly, $R_2$ is a $(p/2)$-random subset of $V(G)\setminus U$. Hence, by Lemma \ref{lem:randomsetconnecting}, with probability $1-o(1)$, $R_2$ is $(D',\ell')$-connecting in $G\setminus U$ for $D'=\frac{(p/2)^9 \explet_2^{13} d_2 }{(\log n_2)^{51}}\geq 40D$ and $\ell'=\frac{5}{2}\explet_2^{-1}(\log n_2)^4\leq 4\explet^{-1}(\log n)^4$. Indeed, the conditions of Lemma \ref{lem:randomsetconnecting} are satisfied as $d'_2\leq d_2/4$, $\explet_2\leq \explet<\frac{1}{(\log n)^2}\leq \frac{1}{(\log n_2)^2}$, $d_2\geq d/2\geq \frac{1}{2}p^{-10}\explet^{-100}\geq (p/2)^{-4}\explet_2^{-6}(\log n_2)^{22}$.

    Since $U$ is $\alpha$-uniform for $\alpha\leq 1/10$ and $|U|\leq n/10$, each vertex $v\in V(G)$ has at most $d/3$ neighbours in $U$, and so each $v$ has at least $d/2$ neighbours in $V(G)\setminus U$. Hence, as $R_1$ is a $(p/2)$-random subset of $V(G)\setminus U$, it follows by the Chernoff bound and the union bound that with probability $1-o(1)$, every $v\in V(G)$ has at least $dp/10$ and at most $2dp$ neighbours in $R_1$. Therefore, it suffices to prove that if these degree conditions hold (and $R_2$ is $(D',\ell')$-connecting in $G\setminus U$), then $R$ is $(D,\ell)$-connecting in $G$.

    Let $x_1,\dots,x_r,y_1,\dots,y_r$ be a sequence of (not necessarily distinct) vertices in $V(G)\setminus R$ such that every vertex in $R$ has at most $D$ neighbours in the multiset $\{x_1,\dots,x_r,y_1,\dots,y_r\}$. Since each vertex in $\{x_1,\dots,x_r,y_1,\dots,y_r\}$ has at least $dp/10$ neighbours in $R_1$, by a simple application of Hall's theorem, there exist pairwise disjoint sets $X_1,\dots,X_r,Y_1\dots,Y_r\subset R_1$ of size $\frac{dp/10}{D}$ each such that $X_i\subset N(x_i)$ and $Y_i\subset N(y_i)$. For each $i\in [r]$, let $x'_i$ be a uniformly random element of $X_i$, and let $y'_i$ be a uniformly random element of $Y_i$. Now fix a vertex $w\in R_2$ and let $L_w$ be the random variable counting the number of neighbours of $w$ in the set $\{x'_1,\dots,x'_r,y'_1,\dots,y'_r\}$. Since $w$ has at most $2dp$ neighbours in $R_1$, and the size of each $X_i$ and $Y_i$ is at least $\frac{dp}{10D}$, the expected value of $L_w$ is at most $2dp\cdot \frac{10D}{dp}=20D$. Since $L_w$ is $1$-Lipschitz with respect to the choices $x'_i$ and $y'_i$, it follows that the probability that $L_w\geq 40D$ is at most $1/n^2$. Indeed, we can apply Talagrand's inequality (Theorem \ref{thm:talagrand}) with $c=r=1$ and $\tau=D\geq (\log n)^5$. Hence, with probability $1-o(1)$, every $w\in R_2$ has at most $40D\leq D'$ neighbours in $\{x'_1,\dots,x'_r,y'_1,\dots,y'_r\}$. Since $R_2$ is $(D',\ell')$-connecting in $G\setminus U$, there is a collection of vertex-disjoint paths $(P'_i)_{i\in [r]}$ such that, for each $i\in [r]$, $P'_i$ is an $x'_i-y'_i$ path in $G$ through $R_2$ of length at most $\ell'$. Appending the edge $x_i x'_i$ at the start and the edge $y'_i y_i$ at the end of each $P'_i$, we obtain a collection of internally vertex-disjoint $x_i-y_i$ paths of length at most $\ell'+2\leq \ell$ through $R=R_1\cup R_2$, so $R$ is $(D,\ell)$-connecting in $G$.
\end{proof}

We also need the following simple lemma stating that a random set of vertices of a given size is likely to be uniform.

\begin{lemma} \plabel{lem:random-s-set-uniform}
    Let $G$ be a $(V_0, V_1)$-alternating $(d \pm d')$-nearly regular graph on $n$ vertices with $d' \le d/64$ and let $U$ be an $\alpha_U$-uniform subset of $V(G)$ with $|U \cap V_0| = |U \cap V_1| \le n / 64$, where $\alpha_U \le 1/64$. Given a nonnegative integer $s \le |V_0 \setminus U|$, let $S$ be a random subset of $V_0 \setminus U$ of size $s$. Set $p = s / |V_0 \setminus U|$ and assume that $pd \ge 10 \log n$. Then with probability $1 - o(1)$, $S$ is $(V_0, V_1, \alpha)$-bi-uniform with respect to $G$, where $\alpha = 20(\alpha_U + d'/d + \sqrt{p \log n / d})$.
\end{lemma}
\begin{proof}
    Consider a fixed vertex $v \in V_1$. The random variable $|S \cap N_G(v)|$ is distributed as a hypergeometric random variable with parameters $|V_0 \setminus U|, s$ and $k \coloneqq |N_G(v) \setminus U|$, where $k = d(1 - |U| / n) \pm 4d' \pm 4 \alpha_U d$, where we used $d' \le d/64$ and that $U$ is $\alpha_U$-uniform. Note that, in particular, we have $k \ge 3d/4$ since $d' \le d/64, \alpha_U \le 1/64$ and $|U| \le n/32$. By \cite[Theorem~2.10]{JLRbook}, standard Chernoff bounds apply in this situation and we shall use \cite[Corollary~2.3]{JLRbook}. Recall that $p = s / |V_0 \setminus U|,$ set $\eps = \sqrt{20 \log n /(pd)}$ and note that $\eps \le \sqrt{2} < 3/2$. Thus, we have
    \[ \Pr\left[ ||S \cap N_G(v)| - pk \ge \eps pk \right] \le 2\exp(- \eps^2 / 3 \cdot p k) \le 2n^{-2}, \]
    where we used that $k \ge 3d/4$. By a union bound, with probability at least $1 - 2n^{-1}$, for all $v \in V_1,$ we have 
    \begin{align*}
        |S \cap N_G(v)| &= (1 \pm \eps) p |N_G(v) \setminus U| = (1 \pm \eps) p (d(1-|U|/n) \pm 4d' \pm 4 \alpha_U d)\\
        &= pd(1-|U|/n) \pm (\alpha/2) d = (s/|V_0|) d \pm (\alpha/2) d. 
    \end{align*}
    Since $\bar{d}(G) =d\pm d'$ and $d'\leq d/64$, it follows that with probability $1 - o(1)$, the set $S$ is $(V_0, V_1, \alpha)$-bi-uniform.
\end{proof}

The rest of this section is split into two subsections. In the first one, we deal with bipartite graphs and in the second, we work with far from bipartite graphs.

\subsection{Finding a bipartite absorber}
First we formally define the bipartite absorber for a set $R$.

\begin{defn}[Bipartite absorber]
    Given a bipartite graph $G$ with parts $A, B$, a balanced set $R \subseteq V(G)$ and vertices $x, y \not\in R,$ a bipartite absorber for $R$ with endpoints $x$ and $y$ is a subgraph $H \subseteq G$ such that $V(H) \supseteq R \cup \{x, y\}$ with the following property. For any balanced set $R' \subseteq R$ of size at most $|R| / 2,$ there is an $x-y$ path in $H$ with vertex set $V(H) \setminus R'$.
\end{defn}

A bipartite absorber is composed of many small gadgets which can absorb pairs of vertices which we define next.
\begin{defn}[$(x,a,b,y)$-absorber] \plabel{def:bip-gadget}
    Given a graph $G$ and distinct vertices $x, a, b, y$, we say a subgraph $H \subseteq G$ such that $\{x, a, b, y\} \subseteq V(H)$ is an $(x, a, b, y)$-absorber if $H$ contains an $x-y$ path with vertex set $V(H)$ as well as an $x-y$ path with vertex set $V(H) \setminus \{ a, b\}$. We shall refer to $V(H) \setminus \{x, a, b, y\}$ as the set of internal vertices of the $(x, a, b, y)$-absorber $H$.
\end{defn}

Next, we describe how an $(x, a, b, y)$-absorber will be constructed. An illustration is given in Figure~\ref{fig:bip-gadget} which is essentially borrowed from~\cite{chakraborti2025edge}.

\begin{figure}
    \centering
    \includegraphics[width=0.9\textwidth]{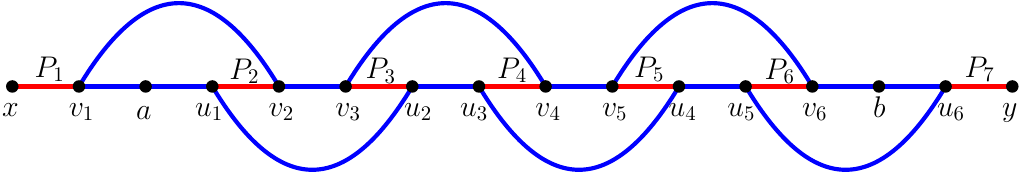}
    \caption{The construction of an $(x, a, b, y)$-absorber with $\ell = 7$. The blue segments and curves represent single edges, where the $a-b$ path with arcs below corresponds to $Q_1$ and the $a-b$ path with arcs above corresponds to $Q_2$. The red segments represent paths $P_1, \dots, P_\ell$.}
    \plabel{fig:bip-gadget}
\end{figure}

\textbf{Construction of an $(x,a,b,y)$-absorber.}
Suppose we are given distinct vertices $x, a, b, y$ in a graph $G$ and an odd integer $\ell$. Consider two $a-b$ paths $Q_1 = u_0, \dots, u_\ell$ and $Q_2 = v_0, \dots, v_\ell$ such that $u_0 = v_0 = a, u_\ell = v_\ell = b$ and $u_i \neq v_j$ for all $i, j \in [\ell-1]$. Let $P_1$ be a path from $x$ to $v_1$, let $P_{\ell}$ be a path from $u_{\ell-1}$ to $y$ and for $i \in [2, \ell-1],$ let $P_i$ be a path from $v_i$ to $u_{i-1}$. Suppose the paths $P_1, \dots, P_{\ell}$ are pairwise vertex-disjoint and do not contain $a$ or $b$. Then the subgraph $H$ with vertex set $V(H) = \{a, b\} \cup \bigcup_{i=1}^{\ell} V(P_i)$ and edge set $E(H) = E(Q_1) \cup E(Q_2) \cup \bigcup_{i=1}^{\ell} E(P_i)$ is an $(x, a, b, y)$-absorber. Indeed, consider the following two paths.

\begin{itemize}
    \item $x, P_1, v_1, a, u_1, P_2, v_2, v_3, P_3, \dots, u_{\ell-3}, u_{\ell-2}, P_{\ell-1}, v_{\ell-1}, b, u_{\ell-1}, P_\ell, y$.
    \item $x, P_1, v_1, v_2, P_2, u_1, u_2, P_3, v_3, v_4, P_4, u_3, \dots, u_{\ell-3}, P_{\ell-2}, v_{\ell-2}, v_{\ell-1}, P_{\ell-1}, u_{\ell-2}, u_{\ell-1}, P_{\ell}, y$.
\end{itemize}
The first path is an $x-y$ path with vertex set $V(H)$ while the second is an $x-y$ path with vertex set $V(H) \setminus \{a, b\}$.

Finally, we use the following lemma which provides a template for the construction of a bipartite absorber for $R$.

\begin{lemma}[Lemma~3.9~in~\cite{chakraborti2025edge}] \plabel{lem:robustly-matchable-initial}
    For every sufficiently large $m$, there exists a bipartite graph $K$ with an odd number of edges, maximum degree at most $102$, and with vertex classes $A_1 \cup A_2$ and $B_1 \cup B_2$, with $|A_1| = |B_1| = 2m$ and $|A_2| = |B_2| = 5m$ such that the following holds. For every $A_1' \subseteq A_1$ and $B_1' \subseteq B_1$ with $|A_1'| = |B_1'| \ge m,$ there is a perfect matching in $H$ between $A_1' \cup A_2$ and $B_1' \cup B_2$.
\end{lemma}

In our proof, it will be convenient to assume that $K$ is regular, so we use the following simple corollary of Lemma~\ref{lem:robustly-matchable-initial}.

\begin{lemma} \plabel{lem:robustly-matchable-bipartite}
     For every sufficiently large $m$, there exists a bipartite $103$-regular multigraph $K$ with vertex classes $A_1 \cup A_2$ and $B_1 \cup B_2$, with $|A_1| = |B_1| = 2m$ and $|A_2| = |B_2| = 5m$ such that the following holds. For every $A_1' \subseteq A_1$ and $B_1' \subseteq B_1$ with $|A_1'| = |B_1'| \ge m,$ there is a perfect matching in $K$ between $A_1' \cup A_2$ and $B_1' \cup B_2$.
\end{lemma}
\begin{proof}
    Let $K_0$ be the graph given by Lemma~\ref{lem:robustly-matchable-initial}. As long as there are vertices of degree less than $103$, add an edge between two such vertices in different parts until the multigraph is $103$-regular. The resulting multigraph $K$ clearly satisfies all the requirements.
\end{proof}

Having done all the preparation, we prove the following lemma which finds a connecting set of vertices and a bipartite absorber for it.

\begin{lemma} \plabel{lem:absorber-bipartite}
    Let $G$ be a $(d \pm d')$-nearly-regular bipartite $\explet$-expander with parts $V_0$ and $V_1$ of size $n/2$, where $0 < \explet < 1 / (\log n)^2.$ Let $U_0 \subseteq V(G)$ be a balanced $\alpha_U$-uniform set. Suppose we are given $c$ with $0 < c < 1/4$ such that $d' \le d^{1-c}, |U_0| \le d^{-c/1000} n, \alpha_U \le d^{-c}$ and $d \ge \explet^{-10^6 / c}$ and suppose $n$ is sufficiently large with respect to $c$. Then, there exists a balanced set $R$ which is $(D, \ell_0)$-connecting in $G$ such that $|R| \ge d^{-c/10^4} n$ and a balanced bipartite absorber $H$ for $R$ such that $V(H) \cap U_0 = \emptyset,$ $|V(H)| \le d^{-c / (2\cdot 10^4)} n$ and $V(H)$ is $\alpha$-uniform with respect to $G$, where $D = d^{1-c/1000}, \ell_0 = 5 \explet^{-1} (\log n)^4$ and $\alpha = d^{-c/20}$.
\end{lemma}
\begin{proof}
    We start by defining some parameters. Let $p = d^{-c/10^4}$. Let $\alpha_0 = 60d^{-c}, \alpha_1 = 10^4 d^{-c}, \alpha_2 = d^{-c/3}$ and $\alpha_3 = d^{-c/9}$. Let $\ell = \explet^{-4}$ rounded to the nearest odd integer. Let $m_0 = pn / 2$ rounded to the nearest odd integer. Let $m_1 = m_2 = 721 m_0$ and $m_3 = \ell m_1$. Finally, let $m_X = (m_1 + 1) / 2$ and note that $m_X$ is an integer since $m_1 = 721m_0$ is odd. To help the reader, we remark that the set $R$ will be a balanced random subset of $V(G) \setminus U_0$ of size roughly $pn$. The absorber is then constructed using Lemma~\ref{lem:connections} in three rounds, where in round $i$ the parameters will be $\alpha_i$ and $m_i$ in place of $\alpha$ and $m$, respectively and the length of the paths will be $\ell$.

    Let $V_0' = V_0 \setminus U_0, V_1' = V_1 \setminus U_0$. Let $n_2 = |V(G)\setminus U_0| = n - |U_0| \ge 0.9n$ and note that $|V_0'| = |V_1'| = n_2 / 2$ since $U_0$ is balanced. 
    
    Let $(A_1, B_1, A_2, B_2, X, Y)$ be a randomly chosen $6$-tuple of pairwise disjoint sets such that $A_1, A_2, X \subseteq V'_0$ and $B_1, B_2, Y \subseteq V'_1$ with sizes $|A_1| = |B_1| = 2m_0, |A_2| = |B_2| = 5m_0$ and $|X| = |Y| = m_X$. Since $A_1$ is distributed as a random subset of $V_0 \setminus U_0$ of size $2m_0$, applying Lemma~\ref{lem:random-s-set-uniform} with the graph $G$ and $U = U_0$, we get that with probability at least $1 - o(1)$, $A_1$ is $(V_0, V_1, \alpha_0)$-bi-uniform with respect to $G,$ where we used that $2m_0 / |V_0 \setminus U_0| \cdot d \ge pd \ge 10 \log n, |U_0| \le d^{-c/1000} n \le n / 64, \alpha_U \le d^{-c} \le 1/64$
    and $20(\alpha_U + d'/d + \sqrt{2m_0 / |V_0 \setminus U_0| \cdot \log n / d}) \le \alpha_0$. Noting that we had plenty of slack in the above inequalities, the analogous statement holds true for the other five sets. Thus, by a union bound, we get that with probability $1 - o(1)$, the sets $A_1, A_2$ and $X$ are $(V_0, V_1,\alpha_0)$-bi-uniform with respect to $G$, while $B_1, B_2$ and $Y$ are $(V_1, V_0, \alpha_0)$-bi-uniform with respect to $G$.
    
    Denote $R = A_1 \cup B_1$. Since $\Pr[\Bin(n_2/2, p) \ge 2m_0] = o(1)$, we may couple a $p$-random subset of $V(G)\setminus U_0$ and $R$ such that the $p$-random subset is a subset of $R$ with probability $1-o(1)$. Noting that a superset of a $(D,\ell_0)$-connecting set is also $(D, \ell_0)$-connecting, by Lemma~\ref{lem:randomset_avoiding_connecting} we have that with probability $1-o(1)$, $R$ is $(D, \ell_0)$-connecting in $G$ for $D = \frac{p^{9} \explet^{13} d}{(\log n)^{52}} \ge d^{1-c/1000}$ and $\ell_0 = 5 \explet^{-1} (\log n)^4$. We used that $d'\leq d^{1-c}\leq \explet d/100$, $|U_0|\leq d^{-c/1000}n\leq \explet n/100$, $\alpha_U\leq d^{-c}\leq \explet/100$ and $d \ge p^{-10}\explet^{-100}$, and to justify $D \ge d^{1-c/1000},$ we used that $p = d^{-c/10^4}, d \ge \explet^{-10^6/c}$ and $\explet < 1/(\log n)^2$.
    
    Hence, let us assume that $R$ is $(D, \ell_0)$-connecting in $G$, the sets $A_1, A_2$ and $X$ are $(V_0, V_1,\alpha_0)$-bi-uniform and the sets $B_1, B_2$ and $Y$ are $(V_1, V_0, \alpha_0)$-bi-uniform with respect to $G$, which by a union bound happens with probability $1 - o(1)$.

    Let $K$ be the $103$-regular multigraph given by Lemma~\ref{lem:robustly-matchable-bipartite} such that for any $A_1' \subseteq A_1$ and $B_1' \subseteq B_1$ with $|A_1'| = |B_1'| \ge m_0$, $K$ contains a perfect matching between $A_1' \cup A_2$ and $B_1' \cup B_2$. Note that $e(K) = 103 |V(K)| / 2 = 721 m_0 = m_1$. We denote $X = \{ x_1, \dots, x_{m_X}\}$ and $Y = \{ y_1, \dots, y_{m_X}\}$. Furthermore, denote the edges of $K$ as $\{ (a_1, b_1), \dots, (a_{m_1}, b_{m_1})\},$ where $a_i \in V_0, b_i \in V_1$ for all $i \in [m_1]$.

    For $i \in [m_1],$ we construct an $(x_{\floor{i/2} + 1}, a_i, b_i, y_{\ceil{i/2}})$-absorber $\Abs_i$. Suppose the internal vertices of all of these absorbers are pairwise disjoint. Then, the union of these absorbers is a bipartite absorber for $R$ with endpoints $x_1$ and $y_{m_X}$. Indeed, denote $S = R \cup A_2 \cup B_2 \cup X \cup Y \cup \bigcup_{i \in [m_1]} V(\Abs_i)$ and consider an arbitrary balanced set $R' \subseteq R$ with $|R'| \le |R| / 2$. Denoting $A_1' = A_1 \setminus R'$ and $B_1' = B_1 \setminus R',$ we have $|A_1'| = |B_1'| \ge m_0$. Let $M$ be a perfect matching in $K$ between $A_1' \cup A_2$ and $B_1' \cup B_2$, which exists by the properties of $K$, and let $I \subseteq [m_1]$ denote the indices corresponding to the edges in $M$, so that $M = \{ (a_i, b_i) \mid i \in I\}$. For each $i \in [m_1],$ let $P_i$ be the path from $x_{\floor{i/2} + 1}$ to $y_{\ceil{i/2}}$ with vertex set $V(\Abs_i)$ if $i \in I$ and with vertex set $V(\Abs_i) \setminus \{ a_i, b_i\}$ if $i \not\in I$. Then the concatenation of the paths $P_1, \overline{P_2}, P_3, \overline{P_4}, \dots, \overline{P_{m_1-1}}, P_{m_1}$ is an $x_1-y_{m_X}$ path with vertex set $S \setminus R'$, where we remind the reader that $\overline{P}$ denotes the reversal of $P$ and we used that $m_1$ is odd.

    Finally, it remains to show that we can construct the absorbers $\Abs_1, \dots, \Abs_{m_1}$ such that the resulting set $S$ is $\alpha$-uniform. We construct these absorbers in three rounds. In the first round we construct an $a_i-b_i$ path $Q^1_i$ for each $i \in [m_1],$ in the second round we construct another $a_i - b_i$ path $Q^2_i$ for each $i \in [m_1]$ and in the third round, for each $i \in [m_1],$ we construct the remaining $\ell$ paths as described in the construction of $(x, a, b, y)$-absorbers.

    First we establish several properties and inequalities needed to apply Lemma~\ref{lem:connections} in each of the rounds. For $i \in [3],$ in round $i$, we shall apply Lemma~\ref{lem:connections} with $\alpha = \alpha_i$, $m = m_i$ and a balanced $\alpha_i$-uniform forbidden set $U_i$ of size at most $4 \ell m_i$. Provided this holds, since $\alpha_i^{7/8} \le \explet / 100$ for any $i \in [3]$ and $d' \le \explet d / 100,$ by Lemma~\ref{lem:remove-set-bip-mixing-time}, for any balanced $\alpha_i^{7/8}$-uniform set $S \subseteq V(G)$ of size at most $|U_i| + \ell m_i \le 5\ell m_i \le \explet n / 100$, the graph $G \setminus S$ is $(V_0 \setminus S, V_1 \setminus S)$-mixing in time $t_0 = 4000 \explet^{-2} \log n \le \ell/10$. Furthermore, note that for any $i \in [3],$ the following inequalities hold: $d' \le d^{1-c} \le \alpha_i d, m_i/n \le 1000 \ell p \le 1000 \ell \explet^{100} \le \ell^{-8}, \alpha_i \le d^{-c/9} \le \explet^{10^5} \le \ell^{-400}, \alpha_i \ge d^{-c} \ge d^{-1/4}$ and $\ell = \explet^{-4} \ge \log n$.

    \textbf{Round 1.} Let $U_1 = U_0 \cup R \cup A_2 \cup B_2 \cup X \cup Y$ and observe that it is $\alpha_U + 3\alpha_0$-uniform, so it is also $\alpha_1$-uniform. Additionally, note that $U_1$ is balanced since $U_0$ and $R$ are balanced, $|A_2| = |B_2|, |X| = |Y|,$ $A_2, X \subseteq V_0, B_2, Y \subseteq V_1$ and the sets $U_0, R, A_2, B_2, X, Y$ are pairwise disjoint. Furthermore, $|U_1| = |U_0| + 14m_0 + 2m_X \le 3\ell m_1$. Since $K$ is $103$-regular, the multiset $\{ a_i \mid i \in [m_1]\}$ is $(V_0, V_1, \alpha_1)$-bi-uniform and the multiset $\{ b_i \mid i \in [m_1]\}$ is $(V_1, V_0, \alpha_1)$-bi-uniform, where we used that $\alpha_1 \ge 103 \alpha_0$. We apply Lemma~\ref{lem:connections} with the graph $G,$ the forbidden set $U_1$ and the multiset of pairs to connect $\{ (a_i, b_i) \mid i \in [m_1] \}.$ Thus, for each $i \in [m_1],$ we obtain an $a_i - b_i$ path $Q^1_i$ of length $\ell$ such that these paths are pairwise internally vertex disjoint, their internal vertices are disjoint from $U_1,$ and for each $t \in [\ell-1],$ the set $\{ Q^1_i(t) \mid i \in [m_1]\}$ is $(V_t, V_{t+1}, \alpha_1^{1/2})$-bi-uniform.

    \textbf{Round 2.} Let $U_2 = U_1 \cup \bigcup_{i \in [m_1]} \Vint(Q^1_i)$. Note that $U_2$ is balanced since $U_1$ is balanced and $\Vint(Q^1_i)$ is balanced for each $i \in [m_1].$ Furthermore, $U_2$ is $\alpha_2$-uniform since $U_1$ is $\alpha_1$-uniform, for each $t \in [\ell-1],$ the set $\{Q^1_i(t) \mid i \in [m_1]\}$ is $(V_t, V_{t+1}, \alpha_1^{1/2})$-bi-uniform, $\ell$ is odd and $\alpha_2 \ge \ell \alpha_1^{1/2}$. Finally, $|U_2|\leq |U_1|+\ell m_1\leq 4\ell m_2$ since $m_1=m_2$.
    We apply Lemma~\ref{lem:connections} to the graph $G$ with the forbidden set $U_2$, the multiset of pairs $\{ (a_i, b_i) \mid i \in [m_2] \}$ and length $\ell$. Thus, for each $i \in [m_2]$, we get an $a_i - b_i$ path $Q^2_i$ of length $\ell$ such that these paths are pairwise internally vertex disjoint, their internal vertices are disjoint from $U_2$ and for each $t \in [\ell-1]$, the set $\{ Q^2_i(t) \mid i \in [m_2] \}$ is $(V_t, V_{t+1}, \alpha_2^{1/2})$-bi-uniform.

    \textbf{Round 3.} Recalling our construction of an $(x, a, b, y)$-absorber, we define the set of pairs to be connected in the third round as follows. For each $i \in [m_1],$ we need to create $\ell$ paths. For $i \in [m_1]$, let 
    \begin{itemize}
        \item $(a^3_{(i-1)\ell + 1}, b^3_{(i-1)\ell + 1}) = (x_{\floor{i/2}+1}, Q^2_i(1)),$
        \item for even $j \in [2, \ell-1],$ let $(a^3_{(i-1)\ell + j}, b^3_{(i-1)\ell + j}) = (Q^2_i(j), Q^1_i(j-1))$,
        \item for odd $j \in [2, \ell-1],$ let $(a^3_{(i-1)\ell + j}, b^3_{(i-1)\ell+ j}) = (Q^1_i(j-1), Q^2_i(j))$,
        \item and $(a^3_{(i-1)\ell + \ell}, b^3_{(i-1)\ell + \ell}) = (Q^1_i(\ell-1), y_{\ceil{i/2}})$.
    \end{itemize}

    Note that the multiset $\{ x_{\floor{i/2}+1} \mid i \in [m_1]\}$ is $(V_0, V_1, 3 \alpha_0)$-bi-uniform since the set $\{ x_i \mid i \in [m_X]\}$ is $(V_0,V_1,\alpha_0)$-bi-uniform and $\alpha_0 \geq 1/\bar{d}(G)$. Since for each even $j \in [\ell-1]$, the set $\{ Q_i^2(j) \mid i \in [m_1]\}$ is $(V_0, V_1, \alpha_2^{1/2})$-bi-uniform and for each odd $j \in [2,\ell],$ the set $\{ Q_i^1(j-1) \mid i \in [m_1]\}$ is $(V_0, V_1, \alpha_2^{1/2})$-bi-uniform, it follows that the multiset $\{ a^3_i \mid i \in [m_3]\}$ is $(V_0, V_1, \alpha_3)$-bi-uniform since $\alpha_3 \ge 3 \alpha_0 + 2 \ell \alpha_2^{1/2}$. Completely analogously, the multiset $\{ b^3_i \mid i \in [m_3]\}$ is $(V_1, V_0, \alpha_3)$-bi-uniform. 
    
    Let $U_3 = U_2 \cup \bigcup_{i \in [m_2]} \Vint(Q^2_i)$ and note that it is $\alpha_3$-uniform as $U_2$ is $\alpha_2$-uniform, for each $t \in [\ell-1]$, the set $\{ Q^2_i(t) \mid i \in [m_2]\}$ is $(V_t, V_{t+1}, \alpha_2^{1/2})$-bi-uniform, $\ell$ is odd and $\alpha_3 \ge 2\ell \alpha_2^{1/2}$. Furthermore, $U_3$ is balanced since $U_2$ is balanced and for each $i \in [m_2]$, $\Vint(Q^2_i)$ is balanced. Finally, $|U_3|\leq |U_2|+\ell m_2\leq 5\ell m_2 < 4\ell m_3$.

    We apply Lemma~\ref{lem:connections} to the graph $G$, with the forbidden set $U_3$, the multiset of pairs $\{(a^3_i, b^3_i) \mid i \in [m_3]\}$ and $\alpha = \alpha_3$. Thus, for each $i \in [m_3],$ we obtain an $a^3_i - b^3_i$ path $P_i$ of length $\ell$ with all these paths pairwise internally vertex disjoint and disjoint from $U_3$. Moreover, by the last point in Lemma~\ref{lem:connections}, the union of the internal vertices of these paths is $\ell \alpha_3^{1/2}$-uniform. Note that for any $i \in [m_1],$ the union $\Abs_i = Q^1_i \cup Q^2_i \cup \bigcup_{j \in [\ell]} P_{(i-1)\ell + j}$ is an $(x_{\floor{i/2}+1}, a_i, b_i, y_{\ceil{i/2}})$-absorber. Thus, taking $H = \bigcup_{i \in [m_1]} \Abs_i$, by the discussion above, $H$ is a bipartite absorber for $R$ and, by construction, $V(H) \cap U_0 = \emptyset$.  Noting that $|V(H)| = A_1 \cup B_1 \cup A_2 \cup B_2 \cup X \cup Y \cup \bigcup_{i \in [m_1], j \in [2]} \Vint(Q^j_i) \cup \bigcup_{i \in [m_3]} \Vint(P_i)$, it follows that $|V(H)| \le m_1 \cdot 2 \ell^2 \le d^{-c / (2 \cdot 10^4)}n$ and that $V(H)$ is $(2\ell \alpha_3^{1/2})$-uniform, so it is also $\alpha$-uniform as $2\ell d^{-c/18} \le d^{-c/20}$.
\end{proof}

\subsection{Finding a non-bipartite absorber}
We move on to defining the analogous structures for the far from bipartite case. Here, the absorbing structure will be simpler since we can construct absorbers for individual vertices.

\begin{defn}[Absorber]
    Given a graph $G$, a set $R \subseteq V(G)$ and vertices $x,y\not \in R$, an absorber for $R$ with endpoints $x$ and $y$ is a subgraph $H \subseteq G$ such that $V(H) \supseteq R \cup \{x, y\}$ with the following property. For any set $R' \subseteq R$, there is an $x-y$ path in $H$ with vertex set $V(H) \setminus R'$.
\end{defn}

We next define an absorber for a single vertex.
\begin{defn}[$(x,a,y)$-absorber] \plabel{def:non-bip-gadget}
    Given a graph $G$ and distinct vertices $x, a, y$, we say a subgraph $H \subseteq G$ such that $\{x, a, y\} \subseteq V(H)$ is an $(x, a, y)$-absorber if $H$ contains an $x-y$ path with vertex set $V(H)$ as well as an $x-y$ path with vertex set $V(H) \setminus \{ a\}$. We shall refer to $V(H) \setminus \{x, a, y\}$ as the set of internal vertices of the $(x, a, y)$-absorber $H$.
\end{defn}

Next, we describe how to construct an $(x, a, y)$-absorber. This is illustrated in Figure~\ref{fig:non-bip-gadget}.

\begin{figure}[H]
    \centering
    \includegraphics[width=0.9\textwidth]{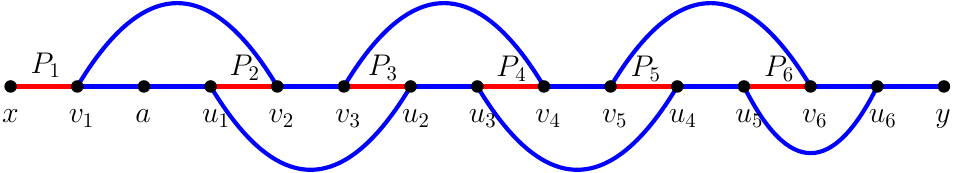}
    \caption{The construction of an $(x, a, y)$-absorber with $\ell = 7$. The blue segments and arcs represent single edges, where the $a-y$ path with arcs below corresponds to $Q_1$ and the $a-u_6$ path with arcs above to $Q_2$. The red segments represent paths $P_1, \dots, P_{\ell-1}$.}
    \plabel{fig:non-bip-gadget}
\end{figure}

\textbf{Construction of an $(x,a,y)$-absorber.}
Suppose we are given distinct vertices $x, a, y$ in a graph $G$ and an odd integer $\ell$. Let $Q_1 = u_0, \dots, u_\ell$ be an $a-y$ path of length $\ell,$ where $u_0 = a, u_\ell = y$. Next, let $Q_2 = v_0, \dots, v_\ell$ be an $a - u_{\ell-1}$ path of length $\ell$, where $v_0 = a$ and $v_{\ell} = u_{\ell-1}$ such that $v_i \neq u_j$ for all $i \in [\ell-1], j \in [\ell]$. Let $P_1$ be a path from $x$ to $v_1$ and for $i \in [2, \ell-1],$ let $P_i$ be a path from $u_{i-1}$ to $v_i$. Suppose the paths $P_1, \dots, P_{\ell-1}$ are pairwise vertex-disjoint and are disjoint from the set $\{a, u_{\ell-1}, y\}$. Then the subgraph $H$ with vertex set $V(H) = \{a, u_{\ell-1}, y\} \cup \bigcup_{i=1}^{\ell-1} V(P_i)$ and edge set $E(H) = E(Q_1) \cup E(Q_2) \cup \bigcup_{i=1}^{\ell-1} E(P_i)$ is an $(x, a, y)$-absorber. Indeed, consider the following two paths.

\begin{itemize}
    \item $x, P_1, v_1, a, u_1, P_2, v_2, v_3, P_3, u_2, u_3, \dots, u_{\ell-3}, u_{\ell-2}, P_{\ell-1}, v_{\ell-1}, u_{\ell-1}, y$.
    \item $x, P_1, v_1, v_2, P_2, u_1, u_2, P_3, v_3, v_4, P_4, u_3, u_4, \dots, u_{\ell-4}, u_{\ell-3}, P_{\ell-2}, v_{\ell-2}, v_{\ell-1}, P_{\ell-1}, u_{\ell-2}, u_{\ell-1}, y$.
\end{itemize}
The first path is an $x-y$ path with vertex set $V(H)$, while the second path is an $x-y$ path with vertex set $V(H) \setminus \{a\}$.

We are ready to prove the nonbipartite analogue of Lemma~\ref{lem:absorber-bipartite} which has a very similar proof.

\begin{lemma} \plabel{lem:absorber-far-from-bipartite}
    Let $G$ be an $n$-vertex $(d \pm d')$-nearly-regular $\explet$-expander and assume $G$ is $\eps$-far from bipartite, where $0 < \explet < 1 / (\log n)^2$ and $\eps>0$. Let $U_0 \subseteq V(G)$ be an $\alpha_U$-uniform set. Suppose we are given $c$ with $0 < c < 1/4$ such that $d' \le d^{1-c}, |U_0| \le d^{-c/1000} n, \alpha_U \le d^{-c}, d \ge (\explet \eps)^{-10^6 / c}$ and suppose $n$ is sufficiently large with respect to $c$. Then there exists a set $R \subseteq V(G) \setminus U_0$ which is a $(D, \ell_0)$-connecting set in $G$ and an absorber $H$ for $R$ such that $V(H) \cap U_0 = \emptyset$, $|V(H)| \le d^{-c / (2 \cdot 10^4)}n$ and $V(H)$ is $\alpha$-uniform with respect to $G$, where $D = d^{1-c/1000}$, $\ell_0 = 5 \explet^{-1} (\log n)^4$ and $\alpha = d^{-c/20}$.
\end{lemma}
\begin{proof}
    We start by defining some parameters. Let $p = d^{-c/10^4}$. Let $\alpha_0 = 60d^{-c}, \alpha_1 = 10^4 d^{-c}, \alpha_2 = d^{-c/3}$ and $\alpha_3 = d^{-c/9}$. Let $\ell = (\eps \explet)^{-4}$ rounded to the nearest odd integer. Let $m = m_1 = m_2 = 2pn$ rounded to the nearest integer and $m_3 = (\ell-1) m_1$. Similarly as in the proof of the previous lemma, the set $R$ will be a random subset of $V(G) \setminus U_0$ of size $m$. The absorber is then constructed using Lemma~\ref{lem:connections} in three rounds, where in round $i$ the parameters will be $\alpha_i$ and $m_i$ in place of $\alpha$ and $m$, respectively, and the length of the paths will be $\ell$.

    Let $V' = V(G) \setminus U_0$. Let $n_2 = |V(G)\setminus U_0| = n - |U_0| \ge 0.9n$. 
    
    Let $(R, X)$ be a randomly chosen pair of disjoint subsets of $V(G)\setminus U_0$ with sizes $|R| = m$ and $|X| = m + 1.$ Since $R$ is distributed as a uniformly random subset of $V(G) \setminus U_0$ of size $m$, by  Lemma~\ref{lem:random-s-set-uniform}, with probability $1 - o(1)$, $R$ is $\alpha_0$-uniform with respect to $G,$ where we used that $d' \le d^{1-c} \le d/ 64, |U_0| \le d^{-c/1000} n \le n / 64, \alpha_U \le d^{-c} \le 1/64, (m / |V'|) d \ge 2p d \ge 10\log n$ and $20(\alpha_U + d'/d + \sqrt{(m/|V'|) \log n / d}) \le \alpha_0$. Analogously, with probability $1 - o(1)$, $X$ is $\alpha_0$-uniform with respect to $G$.

    Since $\Pr[\Bin(n_2, p) \ge m] = o(1)$, we may couple a $p$-random subset of $V(G)\setminus U_0$ and $R$ such that the $p$-random subset is a subset of $R$ with probability $1-o(1)$. Noting that a superset of a $(D,\ell_0)$-connecting set is also $(D, \ell_0)$-connecting, by Lemma~\ref{lem:randomset_avoiding_connecting} we have that with probability $1-o(1)$, $R$ is $(D, \ell_0)$-connecting in $G$ for $D = \frac{p^9 \explet^{13} d}{(\log n)^{52}} \ge d^{1-c/1000}$ and $\ell_0 = 5 \explet^{-1} (\log n)^4$. We used that $d'\leq d^{1-c}\leq \explet d/100$, $|U_0|\leq d^{-c/1000}n\leq \explet n/100$, $\alpha_U\leq d^{-c}\leq \explet/100$ and $d \ge p^{-10}\explet^{-100}$, and to justify $D \ge d^{1-c/1000},$ we used that $p = d^{-c/10^4}, d \ge (\explet \eps)^{-10^6/c}$ and $\explet < 1/(\log n)^2$.
    
    Hence, let us assume that $R$ is $(D, \ell_0)$-connecting in $G$ and the sets $R$ and $X$ are both $\alpha_0$-uniform with respect to $G$, which by a union bound happens with probability $1 - o(1)$.

    We denote $R = \{a_1, \dots, a_{m}\}$ and $X = \{ x_1, \dots, x_{m+1}\}$. For $i \in [m],$ we construct an $(x_i, a_i, x_{i+1})$-absorber $\Abs_i$. Suppose the internal vertices of all of these absorbers are pairwise disjoint. Then, the union of these absorbers is an absorber for $R$ with endpoints $x_1$ and $x_{m+1}$. Indeed, denote $S = R \cup X \cup \bigcup_{i \in [m]} V(\Abs_i)$ and consider an arbitrary set $R' \subseteq R$. Let $I = \{ i \mid a_i \in R'\}$ and for each $i \in I,$ let $P_i$ be the path from $x_i$ to $x_{i+1}$ with vertex set $V(\Abs_i) \setminus \{a_i\}$ and for each $i \in [m] \setminus I$, let $P_i$ be the path from $x_i$ to $x_{i+1}$ with vertex set $V(\Abs_i)$. Then the concatenation of these paths $P_1, P_2, \dots, P_m$ is an $x_1-x_{m+1}$ path with vertex set $S \setminus R'$ as needed.

    Finally, it remains to show that we can construct the absorbers $\Abs_1, \dots, \Abs_m$ such that the resulting set $S$ is $\alpha$-uniform. We shall do so in three rounds following the construction of an $(x, a, y)$-absorber as described above. In the first round, for each $i \in [m],$ we construct an $a_i - x_{i+1}$ path $Q_i^1$ of length $\ell$. In the second round, for each $i \in [m],$ we construct an $a_i - Q_i^1(\ell-1)$ path of length $\ell$ and in the third round we construct the remaining $\ell-1$ paths per absorber.

    First we establish several properties and inequalities needed to apply Lemma~\ref{lem:connections} in each of the rounds. For $i \in [3],$ in round $i$, we shall apply Lemma~\ref{lem:connections} with $\alpha = \alpha_i$, $m = m_i$ and an $\alpha_i$-uniform forbidden set $U_i$ of size at most $4 \ell m_i$. Provided this holds, since $d'/d \le d^{-c} \le \eps^2 \explet^2 / 10^5$ and $\alpha_i^{7/8} \le d^{-c/9 \cdot 7/8} \le (\explet\eps)^{100} \le \eps^2 \explet^2 / 10^5,$ for any $i \in [3],$ by Lemma~\ref{lem:remove-set-far-from-bip-mixing-time}, for any $\alpha_i^{7/8}$-uniform set $S \subseteq V(G)$ of size at most $|U| + \ell m_i \le 5 \ell m_i \le 10 \ell^2 n p \le \eps \explet n / 16$, the graph $G \setminus S$ is $(V(G) \setminus S, V(G) \setminus S)$-mixing in time $t_0 = 10^7 \log n \cdot \eps^{-2} \explet^{-2} \le \ell / 10$. Furthermore, note that for any $i \in [3],$ the following hold: $d' \le d^{1-c} \le \alpha_i d, m_i/n \le 2 \ell p \le 2 \ell (\explet \eps)^{100} \le \ell^{-8}, \alpha_i \le d^{-c/9} \le (\eps\explet)^{10^4} \le \ell^{-400}, \alpha_i \ge d^{-c} \ge d^{-1/4}$ and $\ell = (\eps \explet)^{-4} \ge \log n$.

    \textbf{Round 1.} For $i \in [m]$, let $(a^1_i, b^1_i) = (a_i, x_{i+1})$ and define $U_1 = U_0 \cup R \cup X$. Note that $U_1$ is $\alpha_1$-uniform since $\alpha_U + 2\alpha_0 \le \alpha_1$ and that $|U_1| = |U_0| + 2m + 1 \le d^{-c/1000} n + 2m + 1 \le \ell m_1$. We apply Lemma~\ref{lem:connections} to the graph $G$, with the pairs $(a_i^1, b_i^1)$, the forbidden set $U_1$ and the length of the path $\ell$. Thus, for each $i \in [m]$, we find an $a_i - x_{i+1}$ path $Q^1_i$ of length $\ell$ such that these paths are pairwise internally vertex disjoint, the paths avoid the set $U_1$, and, for all $t \in [\ell-1]$, the set $\{ Q^1_i(t) \mid i \in [m]\}$ is $\alpha_1^{1/2}$-uniform.

    \textbf{Round 2.} Let $U_2 = U_1 \cup \bigcup_{i \in [m]} V(Q^1_i),$ note that $|U_2| \leq |U_1|+\ell m \le 2\ell m_2$ and that $U_2$ is $\alpha_2$-uniform since $\alpha_1 + \ell \alpha_1^{1/2} \le \alpha_2$. For $i \in [m],$ let $(a^2_i, b^2_i) = (a_i, Q^1_i(\ell-1))$ and recall that the set $\{ Q^1_i(\ell-1) \mid i \in [m]\}$ is $\alpha_2$-uniform. Applying Lemma~\ref{lem:connections} with the set of pairs $\{(a^2_i, b^2_i) \mid i \in [m]\}$, forbidden set $U_2$ and length $\ell$, for each $i \in [m],$ we obtain an $a_i - Q^1_i(\ell-1)$ path $Q^2_i$ of length $\ell,$ these paths are pairwise internally vertex disjoint, their internal vertices are disjoint from $U_2$ and for any $t \in [\ell-1],$ the set $\{ Q^2_i(t) \mid i \in [m]\}$ is $\alpha_2^{1/2}$-uniform.

    \textbf{Round 3.} Finally, we find the remaining $\ell-1$ paths for each absorber. To this end, for each $i \in [m],$ let
    \begin{itemize}
        \item $(a^3_{(\ell-1)(i-1) + 1}, b^3_{(\ell-1)(i-1) + 1}) = (x_i, Q^2_i(1))$ and,
        \item for $j \in [2, \ell-1]$, let $(a^3_{(\ell-1)(i-1) + j}, b^3_{(\ell-1)(i-1) + j}) = (Q^1_i(j-1), Q^2_i(j))$.
    \end{itemize}

    Recall that $m_3 = (\ell-1) m$. Let $U_3 = U_2 \cup \bigcup_{i \in [m]} V(Q^2_i).$ Observe that $|U_3|\leq |U_2|+\ell m\leq 4\ell m_3$. Note that the multisets $U_3, \{ a^3_i \mid i \in [m_3] \}$ and $\{ b^3_i \mid i \in [m_3]\}$ are $\alpha_3$-uniform, since $4\ell \alpha_2^{1/2} \le \alpha_3$. Applying Lemma~\ref{lem:connections} with the set of pairs $\{ (a^3_i, b^3_i) \mid i \in [m_3] \}$, the forbidden set $U_3$, $\alpha = \alpha_3$ and length $\ell$, for $i \in [m_3],$ we obtain an $a^3_i - b^3_i$ path $P_i$ of length $\ell$, such that these paths are pairwise internally vertex disjoint, their internal vertices are disjoint from $U_3$ and for all $t \in [\ell-1],$ the set $\{ P_i(t) \mid i \in [m_3]\}$ is $\alpha_3^{1/2}$-uniform. 
    
    Note that for $i \in [m],$ the union $\Abs_i = Q^1_i \cup Q^2_i \cup \bigcup_{j \in [\ell-1]} P_{(\ell-1)(i-1) + j}$ is an $(x_i, a_i, x_{i+1})$-absorber. Therefore, by the discussion above, the union $H = \bigcup_{i \in [m]} \Abs_i$ is an absorber for $R$ with endpoints $x_1$ and $x_{m+1}$. Recall that $V(H) = R \cup X \cup \bigcup_{i \in [m], j \in [2]} \Vint(Q_i^j) \cup \bigcup_{i \in [m_3]} \Vint(P_i)$, which is, by construction, disjoint from $U_0$. Thus, $|V(H)| \le 4\ell^2 m \le d^{-c / (2 \cdot 10^4)} n$ and $V(H)$ is $\alpha$-uniform with respect to $G$, since $\alpha \ge 2\ell \alpha_3^{1/2}.$ This completes the proof.
\end{proof}



\section{Robust Hamiltonicity theorems} \label{sec:robust}
Before presenting the proofs of our robust Hamiltonicity theorems, we show how to construct a linear forest in a nearly-regular graph such that the endpoints of the paths in this forest form a uniform set. We can use the following simple lemma from \cite{chakraborti2025edge} to argue that after taking our matchings in the leftover set, the set of vertices that we still need to connect contains few vertices in each neighbourhood.

\begin{lemma} {\cite[Lemma 5.1]{chakraborti2025edge}} \plabel{lem:matchings with little leftover}
    Let $G$ be a graph and let $V_1,\dots,V_t$ be disjoint subsets of $V(G)$ such that, for each $j\in [t-1]$, the degree of every vertex in $G[V_j,V_{j+1}]$ is between $\delta$ and $\Delta$. Let $T_1,\dots,T_n$ be subsets of $V(G)$ such that, for each $i\in [n]$ and $j\in [t]$, we have $|T_i\cap V_j|\leq r$. Assume that $1-\delta/\Delta\leq t^{-1/2}\log n$. Then, there exist matchings $M_j$ for each $j\in [t-1]$ in $G[V_j,V_{j+1}]$ such that, if $Y$ consists of all those vertices $y\in V_1\cup V_t$ which belong to neither $M_1$ nor $M_{t-1}$ and all those vertices $y\in V_j$ with $2\leq j\leq t-1$ which do not belong to both $M_{j-1}$ and $M_j$, then $|Y|\le 10|\bigcup_{i=1}^t V_i|t^{-1/2}\log n$ and $|Y\cap T_i|\leq 10rt^{1/2}\log n$ for all $i\in [n]$.
\end{lemma}

For convenience, we will also use the following very easy lemma from \cite{chakraborti2025edge}.

\begin{lemma} {\cite[Lemma~5.2]{chakraborti2025edge}} \plabel{lem:simple random subgraph}
    Let $G$ be an $n$-vertex graph. Let $V\subset V(G)$. Let $d\ge (\log n)^{15}$, let $0\le d' \le d$ and let $1\le t\le d$ be an integer. 
    Let $U$ be a uniformly random subset of $V$ of size $\lfloor \frac{1}{t}|V| \rfloor$. 
    \begin{enumerate}
        \item[$\mathrm{(a)}$] If $v\in V(G)$ satisfies $d_G(v,V)\le d+ d'$, then $d_G(v,U)\le \frac{1}{t}d + d' + d^{2/3}$ with probability $1-n^{-\omega(1)}$. 
        \item[$\mathrm{(b)}$] If $v\in V(G)$ satisfies $d_G(v,V)\ge d- d'$, then $d_G(v,U)\ge \frac{1}{t}d - d' - d^{2/3}$ with probability $1-n^{-\omega(1)}$. 
    \end{enumerate}
\end{lemma}

We now deduce the following more useful lemma.

\begin{lemma} \plabel{lem:linear-forest}
   For any $0<c<1/3$, the following holds for all sufficiently large $n$. Let $G$ be an $n$-vertex graph and let $V\subset V(G)$. Let $d\geq (\log n)^{21/c}$. Assume that each $v\in V(G)$ satisfies $d-d^{1-c}\leq d_G(v,V)\leq d+d^{1-c}$. Then there exists a linear forest $\mathcal{F}$ in $G[V]$ such that the number of vertices in $V$ which have degree less than $2$ in $\mathcal{F}$ is at most $d^{-c/5}n$, and for each $v\in V(G)$, the number of vertices in $N_G(v)\cap V$ which have degree less than $2$ in $\mathcal{F}$ is at most $d^{1-c/5}$.
\end{lemma}

\begin{proof}
    Let $t=d^{c/2}$. Consider a random partition $V=V_1\cup \dots \cup V_t$ into sets of size $|V|/t$. By Lemma \ref{lem:simple random subgraph}, with high probability, for all $v\in V(G)$ and all $j\in [t]$, we have $d^{1-c/2}-2d^{1-c}\leq d_G(v,V_j)\leq d^{1-c/2}+2d^{1-c}$. In particular, the degree of every vertex in $G[V_j,V_{j+1}]$ is between $\delta=d^{1-c/2}-2d^{1-c}$ and $\Delta=d^{1-c/2}+2d^{1-c}$. Observe that $1-\delta/\Delta = 4d^{1-c}/\Delta\leq 4d^{-c/2}\leq t^{-1/2}\log n$. Let $V(G)=\{v_1,\dots,v_n\}$ and let $T_i=N_G(v_i)$. Observe that $|T_i\cap V_j|=d_G(v_i,V_j)\leq 2d^{1-c/2}$ for each $i\in [n]$ and $j\in [t]$. Hence, applying Lemma \ref{lem:matchings with little leftover} with $r=2d^{1-c/2}$, we find that there exist matchings $M_j$ for each $j\in [t-1]$ in $G[V_j,V_{j+1}]$ such that, if $Y$ consists of all those vertices $y\in V_1\cup V_t$ which belong to neither $M_1$ nor $M_{t-1}$ and all those vertices $y\in V_j$ with $2\leq j\leq t-1$ which do not belong to both $M_{j-1}$ and $M_j$, then $|Y|\leq 10|V|t^{-1/2}\log n$ and $|Y\cap T_i|\leq 10rt^{1/2}\log n$ for all $i\in [n]$. Let $\mathcal{F}$ be the union of these matchings. Clearly, $\mathcal{F}$ is a linear forest.
    The number of vertices in $V$ which have degree less than $2$ in $\mathcal{F}$ is at most $|Y|+|V_1|+|V_t|\leq 10|V|t^{-1/2}\log n +2|V|/t\leq d^{-c/5}n$.
    Moreover, the number of vertices in $N_G(v_i)\cap V$ which have degree less than $2$ in $\mathcal{F}$ is at most $|Y\cap T_i|+|V_1\cap T_i|+|V_t\cap T_i|\leq 10rt^{1/2}\log n +d_G(v_i,V_1)+d_G(v_i,V_t)\leq 20d^{1-c/4}\log n + 2(d^{1-c/2}+2d^{1-c})\leq d^{1-c/5}$. 
\end{proof}

We proceed with the robust Hamiltonicity theorems, starting with the bipartite case. The proof is schematically represented in Figure~\ref{fig:robust}.

\begin{figure}[H]
    \centering
    \includegraphics[width=0.85\textwidth]{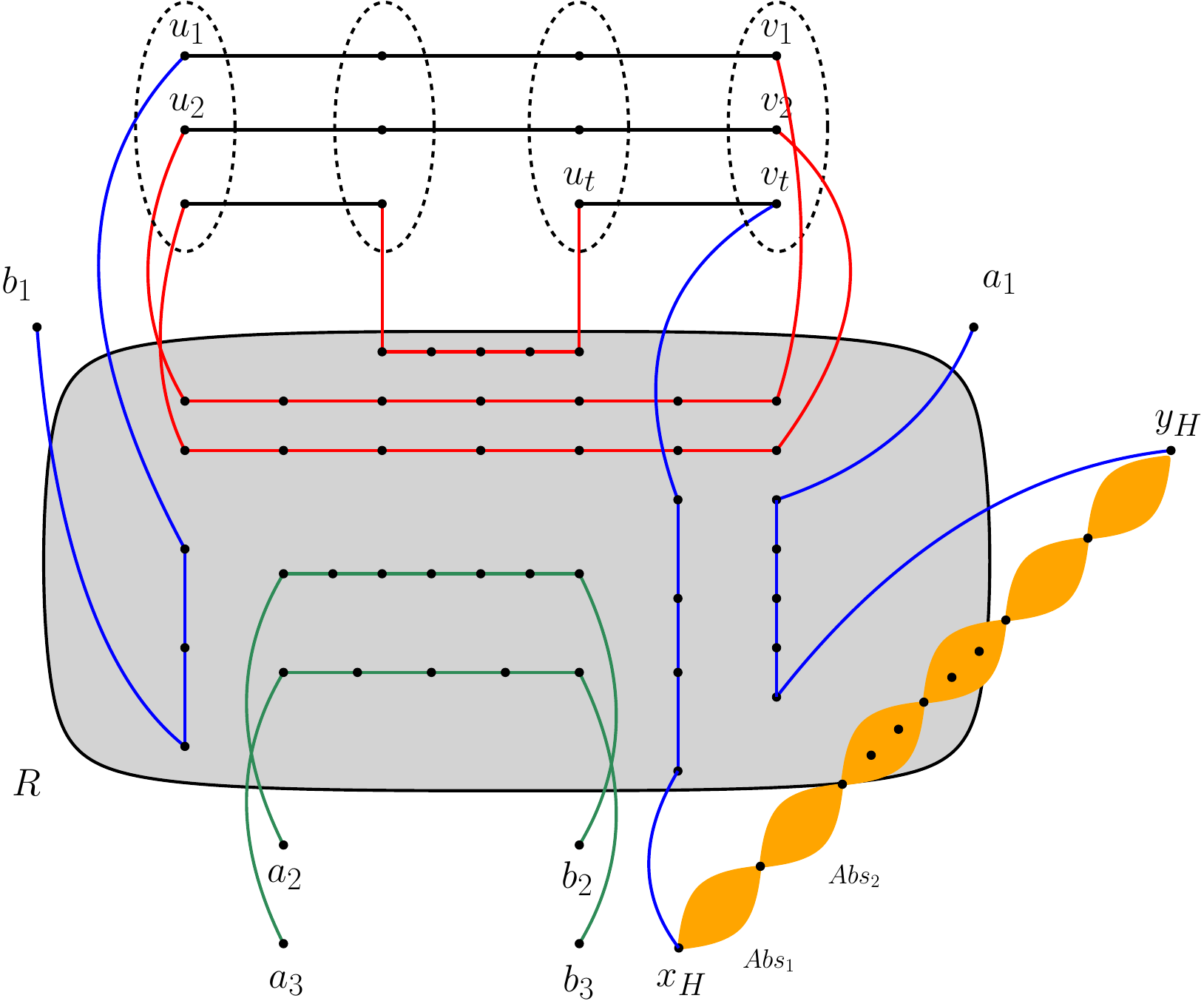}
    \caption{An example with $m=3$. The red, blue and green paths are constructed using the connecting properties of $R$. The orange lens-shaped objects represent $(x, a, b, y)$-absorbers.}
    \plabel{fig:robust}
\end{figure}

\begin{theorem} \plabel{thm:robust-bipartite}
   Let $G$ be a $(d \pm d')$-nearly-regular bipartite $\explet$-expander with parts $V_0$ and $V_1$ of size $n/2$, where $\explet < 1 / (\log n)^2.$ Suppose we are given a positive integer $m$ and pairs of vertices $(a_i, b_i)_{i \in [m]}$ with $a_i \in V_0, b_i \in V_1$ for all $i \in [m]$. Assume that $\{ a_i \mid i \in [m]\}$ is $(V_0, V_1, \alpha)$-bi-uniform both as a set and as a multiset, and analogously, that $\{ b_i \mid i \in [m]\}$ is $(V_1, V_0, \alpha)$-bi-uniform both as a set and as a multiset. Furthermore, assume that the sets $\{a_i \mid i \in [m]\}$ and $\{b_i \mid i \in [m]\}$ are of equal size.

   Suppose we are given $c$ with $0 < c < 1/4$ and assume that $n$ is sufficiently large with respect to $c$ and the following inequalities hold: $d' \le d^{1 - c}, \alpha \le d^{-c}, m \le d^{- c/900} n$ and $d \ge \explet^{-10^6/c}.$ Then, there exist internally vertex disjoint paths $P_1, \dots, P_m$ such that $P_i$ is an $a_i-b_i$ path for each $i \in [m]$ and $\bigcup_{i \in [m]} \Vint(P_i) = V(G) \setminus \bigcup_{i \in [m]} \{a_i, b_i\}$. 
\end{theorem}
\begin{proof}

    Let $U = \bigcup_{i \in [m]} \{ a_i, b_i\}$. Note that $U$ is an $\alpha$-uniform set with respect to $G$ and $|U| \le 2m \le d^{-c/1000} n$. Applying Lemma~\ref{lem:absorber-bipartite} with $U_0 = U$, we obtain a balanced set $R$ which is $(D, \ell_0)$-connecting in $G$ with $|R| \ge d^{-c/10^4} n$ and a balanced bipartite absorber $H$ for $R$ with endpoints $x_H\in V_0$ and $y_H\in V_1$ such that $V(H) \cap U = \emptyset, |V(H)| \le d^{-c / (2 \cdot 10^4)} n \le n/100$ and $V(H)$ is $\alpha_H$-uniform with respect to $G,$ where $D = d^{1 - c/1000}$, $\ell_0 =  5 \explet^{-1} (\log n)^4$ and $\alpha_H = d^{-c/20}$.

    Note that the set $U \cup V(H)$ is $(\alpha+\alpha_H)$-uniform with respect to $G$. Let $V = V(G) \setminus (U \cup V(H))$. Observe that $V$ is a balanced set since $U$ and $V(H)$ are balanced and note that for any $v \in V(G),$ we have 
    \begin{align*}
        d_G(v, V) &= d_G(v) - |(U \cup V(H)) \cap N_G(v)| = d \pm d' - (|U| + |V(H)|) d / n \pm d' \pm 2 (\alpha+\alpha_H) d\\
        &= d(1 - (|U| + |V(H)|) / n) \pm 10d^{1-c/40}.
    \end{align*}
    Let $d_V = d(1 - (|U| + |V(H)|) / n)$ and $d_V' = 10d^{1-c/40}.$ Note that $d_V \ge d / 2$ and thus $d_V' \le d_V^{1-c/45}$. Note that $d_V \ge d/2 \ge (\log n)^{10^6 / c},$ thus we may apply Lemma~\ref{lem:linear-forest} with $d = d_V$ and parameter $c/45$ in place of $c$ to obtain a linear forest $\calF$ in $G[V]$ such that the number of vertices in $V$ which have degree less than $2$ in $\mathcal{F}$ is at most $d^{-c/225}n$, and for each $v \in V(G)$, the number of vertices in $N_G(v) \cap V$ which have degree less than $2$ in $\calF$ is at most $d_V^{1 - c/225} \le d^{1-c/225}.$ 

    Next, we define pairs to be connected through the connecting set $R$. Let $(u_1, v_1), \dots, (u_t, v_t)$ be the set of endpoints of the paths in $\calF$, where we treat vertices in $V$ with degree zero in $\calF$ as paths with the same endpoint $u_i = v_i$. Observe that each $u_i$ has degree less than $2$ in $\mathcal{F}$, so by the previous paragraph, we have $t\leq d^{-c/225}n$. First, we pair up the endpoints of the paths in $\calF$ such that after finding connections between the pairs, we are left with a single path. Concretely, for $i \in [t-1],$ let $x_i = v_i, y_i = u_{i+1}$. Then, let $x_t = b_1, y_t = u_1, x_{t+1} = v_t, y_{t+1} = x_H$. Observe that for any set of pairwise internally vertex disjoint $x_i - y_i$ paths through $R$, for $i \in [t+1]$, the union of these $t+1$ paths together with the paths in $\calF$ yields a $b_1 - x_H$ path. Let $x_{t+2} = y_H, y_{t+2} = a_1$ and for $i \in [2, m],$ let $x_{t+i+1} = a_i, y_{t+i+1} = b_i$.
        
    

    
    Denote $r = t+m+1$ and let $W = \{ x_i \mid i \in [r] \} \cup \{ y_i \mid i \in [r]\}$ denote the multiset of endpoints we wish to connect through $R$. Given some $v\in R$, we wish to upper bound $|N_G(v)\cap W|$. Note that each vertex in $V$ that has degree less than $2$ in $\mathcal{F}$ appears at most twice in $W$, so these vertices contribute at most $2d^{1-c/225}$ to $|N_G(v)\cap W|$. Since $a_i, b_i,x_H, y_H$ each appear once in $W$ and the multiset $\{a_i \mid i\in [m]\}\cup \{b_i \mid i\in [m]\}$ is $\alpha$-uniform, for each $v \in R$, 
    \[ |N_G(v) \cap W| \le 2d^{1 - c/225} + \left( \frac{2m \bar{d}(G)}{n} + \alpha \bar{d}(G) \right) + 2 \le 20d^{1- c /900} \le D, \]
    where we used that $D=d^{1-c/1000}$, $m \le d^{-c/900} n$, $\bar{d}(G) \le 2d, \alpha \leq d^{-c}$ and $d = \omega(1)$.

    Since $R$ is $(D, \ell_0)$-connecting in $G$, there exists a collection $(P^*_i)_{i \in [r]}$ of pairwise internally disjoint paths of length at most $\ell_0$ such that for each $i \in [r]$, $P^*_i$ is an $x_i - y_i$ path with internal vertices in $R$. Let $R' = \bigcup_{i \in [r]} \Vint(P^*_i)$. We claim that $R'$ is a balanced subset of $R$. Indeed, as already mentioned, the path $Q^* \coloneqq \calF \cup \bigcup_{i \in [t+1]} P^*_i$ is a $b_1 - x_H$ path and note that $\Vint(Q^*) \supseteq V$ and $\Vint(Q^*) \setminus V \subseteq R$. Since $V$ is a balanced set, it follows that $\Vint(Q^*) \cap R$ is a balanced set. For $i \in [t+2, t+m+1],$ observe that $\{x_i, y_i\}$ is a balanced set, thus $\Vint(P^*_i)$ is a balanced subset of $R$ for all $i \in [t+2, t+m+1]$. Hence, $R' = (\Vint(Q^*) \cap R) \cup \bigcup_{i \in [t+2, t+m+1]} \Vint(P^*_i)$ is a balanced subset of $R$, as claimed.
    
    Furthermore, we have $|R'| \le r \ell_0 = (t + m + 1) \ell_0 \le (d^{-c/225}n + d^{-c/900}n + 1) \cdot 5 \explet^{-1} (\log n)^4 \le d^{-c/920} n < |R| / 2.$ Hence, by the property of the absorber, there is an $x_H - y_H$ path $Q$ with vertex set $V(H) \setminus R'$. Recalling the discussion above, we see that $P_1 = P^*_{t+2} \cup Q \cup Q^*$ is an $a_1 - b_1$ path whose internal vertices contain the set $V$. For $i \in [2,m]$, let $P_i = P^*_{t+i+1}$ and note that $P_i$ is an $a_i - b_i$ path. Since $V(H) = V(Q) \cup \bigcup_{i \in [r]} \Vint(P^*_i)$, it follows that $\bigcup_{i \in [m]} \Vint(P_i) = V(G) \setminus U$. Finally, it is easy to see that the paths $(P_i)_{i \in [m]}$ are pairwise internally disjoint, finishing the proof; see Figure~\ref{fig:robust} for an illustration.
\end{proof}

Next, we have the far-from-bipartite analogue of the previous theorem with a similar proof.
\begin{theorem}
    \plabel{thm:robust-far-from-bipartite}
    Let $G$ be a $(d \pm d')$-nearly-regular $\explet$-expander which is $\eps$-far from bipartite, for some $0<\explet<1/(\log n)^2$ and $\eps > 0$. Suppose we are given a positive integer $m$ and pairs of vertices $(a_i, b_i)_{i \in [m]}$ such that $a_i \neq b_i$ for all $i \in [m]$ and $\bigcup_{i \in [m]} \{ a_i, b_i \}$ is $\alpha$-uniform both as a set and as a multiset.
   
    Suppose we are given $c$ with $0 < c < 1/4$ and assume that $n$ is sufficiently large with respect to $c$ and the following inequalities hold: $d' \le d^{1-c}, \alpha \le d^{-c}, m \le d^{- c/900} n$ and $d \ge (\explet \eps)^{-10^6/c}.$ Then, there exist internally vertex disjoint paths $P_1, \dots, P_m$ such that $P_i$ is an $a_i-b_i$ path for each $i \in [m]$ and $\bigcup_{i \in [m]} \Vint(P_i) = V(G) \setminus \bigcup_{i \in [m]} \{a_i, b_i\}$. 
\end{theorem}
\begin{proof}
    Let $U = \bigcup_{i \in [m]} \{ a_i, b_i\}$. Note that $U$ is an $\alpha$-uniform set with respect to $G$ and $|U| \le 2m \le d^{-c/1000} n$. Applying Lemma~\ref{lem:absorber-far-from-bipartite} with $U_0 = U$, we obtain a set $R \subseteq V(G) \setminus U$ which is $(D, \ell_0)$-connecting in $G$ and an absorber $H$ for $R$ with endpoints $x_H$ and $y_H$ such that $V(H) \cap U = \emptyset, |V(H)| \le d^{-c / (2 \cdot 10^4)} n$ and $V(H)$ is $\alpha_H$-uniform with respect to $G,$ where $D = d^{1 - c/1000}$, $\ell_0 =  5 \explet^{-1} (\log n)^4$, $\alpha_H = d^{-c/20}$.

    Note that the set $U \cup V(H)$ is $(\alpha+\alpha_H)$-uniform with respect to $G$. Let $V = V(G) \setminus (U \cup V(H))$ and note that for any $v \in V(G),$ we have 
    \begin{align*}
        d_G(v, V) &= d_G(v) - |(U \cup V(H)) \cap N_G(v)| = d \pm d' - (|U| + |V(H)|) d / n \pm d' \pm 2 (\alpha+\alpha_H) d\\
        &= d(1 - (|U| + |V(H)|) / n) \pm 10d^{1-c/150}.
    \end{align*}
    
    Let $d_V = d(1 - (|U| + |V(H)|) / n)$ and $d_V' = 10d^{1-c/150}.$ Note that $d_V \ge d / 2$ and thus $d_V' \le d_V^{1-c/170}$. Since $d_V \ge d/2 \ge (\log n)^{10^6 / c},$ we may apply Lemma~\ref{lem:linear-forest} with $d = d_V$ and parameter $c/170$ in place of $c$ to obtain a linear forest $\calF$ in $G[V]$ such that for each $v \in V(G)$, the number of vertices in $N_G(v) \cap V$ which have degree less than $2$ in $\calF$ is at most $d_V^{1 - c/850} \le d^{1-c/900}.$

    Next, we define the pairs to be connected through the connecting set $R$. Let $(u_1, v_1), \dots, (u_t, v_t)$ be the set of endpoints of the paths in $\calF$, where we treat vertices in $V$ with degree zero in $\calF$ as paths with the same endpoint $u_i = v_i$. First, we pair up the endpoints of the paths in $\calF$ such that after finding connections between the pairs, we are left with a single path. Concretely, for $i \in [t-1],$ let $x_i = v_i$ and $y_i = u_{i+1}$. Observe that for any set of pairwise internally disjoint $x_i - y_i$ paths, $i \in [t-1],$ through $R$, the union of $\calF$ and these $t-1$ paths is a $u_1 - v_t$ path. Let $x_t = a_1, y_t = y_H, x_{t+1} = x_H, y_{t+1} = v_t$ and $x_{t+2} = u_1, y_{t+2} = b_1$. Note that for any set of pairwise internally disjoint $x_i - y_i$ paths, $i \in [t+2],$ through $R$, the union of these $t+2$ paths together with $\calF$ and an $x_H - y_H$ path internally disjoint from these paths, yields an $a_1 - b_1$ path. Finally, for $i \in [2, m],$ let $x_{t+i+1} = a_i, y_{t+i+1} = b_i$. 
    
    Denote $r = t+m+1$ and let $W = \{ x_i \mid i \in [r] \} \cup \{ y_i \mid i \in [r]\}$ denote the multiset of endpoints we wish to connect through $R$. 
    Given some $v\in R$, we wish to upper bound $|N_G(v)\cap W|$. Note that each vertex in $V$ that has degree less than $2$ in $\mathcal{F}$ appears at most twice in $W$, so these vertices contribute at most $2d^{1-c/900}$ to $|N_G(v)\cap W|$. Since $a_i, b_i,x_H, y_H$ each appear once in $W$ and the multiset $\{a_i \mid i\in [m]\}\cup \{b_i \mid i\in [m]\}$ is $\alpha$-uniform, for each $v \in R$, 
    \[ |N_G(v) \cap W| \le 2d^{1 - c/900} + \left( \frac{2m \bar{d}(G)}{n} + \alpha \bar{d}(G) \right) + 2 \le 20d^{1-c/900} \le D, \]
    where we used that $D=d^{1-c/1000}$, $m \le d^{-c/900} n$, $\bar{d}(G) \le 2d, \alpha \le d^{-c}$ and $d = \omega(1)$.

    Since $R$ is $(D, \ell_0)$-connecting in $G$, there exists a collection $(P^*_i)_{i \in [r]}$ of pairwise internally disjoint paths of length at most $\ell_0$ such that for each $i \in [r]$, $P^*_i$ is an $x_i - y_i$ path with internal vertices in $R$. Let $R' = \bigcup_{i \in [r]} \Vint(P^*_i)$. By the property of the absorber, there is an $x_H - y_H$ path $Q$ with vertex set $V(H) \setminus R'$. Recalling the discussion above, we see that $P_1 = Q \cup \mathcal{F} \cup \bigcup_{i \in [t+2]} P^*_i$ is an $a_1 - b_1$ path, whose internal vertices contain the set $V$. For $i \in [2,m]$, let $P_i = P^*_{t+i+1}$ and note that $P_i$ is an $a_i - b_i$ path. Since $V(H) = V(Q) \cup \bigcup_{i \in [r]} \Vint(P^*_i)$, it follows that $\bigcup_{i \in [m]} \Vint(P_i) = V(G) \setminus U$. Finally, it is easy to see that the paths $(P_i)_{i \in [m]}$ are pairwise internally vertex disjoint, finishing the proof.
\end{proof}

Note that Theorems~\ref{thm:main-bipartite}~and~\ref{thm:main-far-from} follow directly by applying Theorems~\ref{thm:robust-bipartite}~and~\ref{thm:robust-far-from-bipartite}, respectively, with $m=1$ and $
a_1b_1$ an arbitrary edge.

\section{Applications} \label{sec:applications}
\subsection{Hamiltonicity of moderately dense regular sublinear expanders}
In this subsection, we prove Theorem~\ref{thm:dense-regular}. We recall the statement for reader's convenience. 

\thmregularrestate*

We will need the following simple lemma. We remark that a slightly stronger bound can be obtained using known results towards the linear arboricity conjecture, but this easier bound is sufficient for our purposes.

\begin{lemma} \plabel{lem:large linear forest}
    Every graph with $m$ edges and maximum degree at most $\Delta$ contains a linear forest with at least $m/\Delta$ edges.
\end{lemma}

\begin{proof}
    We apply induction on $m$. In the case $m=0$, the statement vacuously holds. Now let $G$ be a graph with $m\geq 1$ edges and maximum degree at most $\Delta$. Consider a maximal path $P=v_0v_1 \dots v_k$ in $G$. By maximality, each edge in $G$ incident to some $v_i$ has to be incident to one of $v_0,\dots,v_{k-1}$, hence, by the maximum degree assumption, there are at most $k\Delta$ such edges. Let $G'$ be the graph obtained from $G$ by removing these edges. Then $G'$ has at least $m-k\Delta$ edges, so by the induction hypothesis, it contains a linear forest with at least $m/\Delta-k$ edges. Removing isolated vertices from this forest and adding $P$ to it, we obtain a linear forest in $G$ with at least $m/\Delta$ edges.
\end{proof}

Using Lemma~\ref{lem:large linear forest}, it is easy to deduce the following lemma which will be convenient for the proof of Theorem~\ref{thm:dense-regular}.

\begin{lemma} \plabel{lem:clean-up-forest}
    Let $G$ be a $d$-regular graph and suppose we are given a vertex partition $V(G) = A \cup B,$ where $|A| \le |B|$, and a set $X \subseteq V(G)$ such that $\Delta(G[A]), \Delta(G[B]) \le d/2$ and $|X| + |B| - |A| \le d/10$. Then there exists a linear forest $\calF$ in $G$ satisfying the following.
    \begin{enumerate}[label=LF\arabic*)]
        \item Every path in $\calF$ has one endpoint in $A$ and the other in $B$.
        \item $|V(\calF)| \le 4(|X| + |B|-|A|)$.
        \item \plabel{lf-prop-equal-parts} Let $U$ denote the set of internal vertices in the paths in $\calF$. Then, $X \subseteq U$ and $|A \setminus U| = |B \setminus U|$.
    \end{enumerate}    
\end{lemma}
\begin{proof}   
    Note that $d|B| = 2e_G(B) + e_G(A,B) \le 2e_G(B) + d|A|,$ which implies $e_G(B) \ge \frac{d}{2}(|B| - |A|)$.
    Applying Lemma~\ref{lem:large linear forest} to the graph $G[B]$ with $\Delta = d/2$, we obtain a linear forest in $G[B]$ with at least $e_G(B) / \Delta \ge |B| - |A|$ edges. So, let $\calF'$ be a linear forest consisting of non-empty paths in $G[B]$ with exactly $|B| - |A|$ edges in total. We extend each path in $\calF'$ using a distinct vertex in $A$, which can be done greedily since every vertex in $B$ has at least $d/2 > |B| - |A|$ neighbours in $A$. Let $\calF$ denote this augmented linear forest. 
    
    Next, iteratively, for every vertex $v \in X$ in an arbitrary order, we do the following. If $v \in V(\calF)$ and $v$ is an internal vertex of a path in $\calF$, do nothing. If $v \in V(\calF)$ and $v$ is an endpoint of a path in $\calF$, then extend that path using any two edges $uv$ and $uv' \in E(G[A, B])$ for some $u,v'$ such that $u, v' \not\in V(\calF)$. Finally, if $v \not\in V(\calF)$, let $v_1, v_2, v_3, v_4$ be a path of length three in $G[A, B]$ such that $v_2 = v$ and $v_i \not \in V(\calF)$, for all $i \in [4]$. Add the path $v_1, \dots, v_4$ to $\calF$. Let us verify that this is always possible. Indeed, at any step, $V(\calF)$ has size at most $3 (|B| - |A|) + 4|X| \le \frac{2d}{5}$ and every vertex has degree at least $d/2$ in $G[A, B]$. Therefore, at each step we can greedily extend the path if $v$ is an endpoint in $\calF$ or construct the desired path $v_1, v_2, v_3, v_4$ if $v \not \in V(\calF)$.

    Let $\calF$ denote this final linear forest and let $U$ be the set of internal vertices of the paths in $\calF$. From the construction, it is easy to see that every path has one endpoint in $A$ and one in $B$ and that $X \subseteq U$. We have already noted that $|V(\calF)| \le 4(|X| + |B| - |A|)$. Thus, it remains to prove that $|A \setminus U| = |B \setminus U|$.
    
    Note that in $\calF$, there are exactly $|B| - |A|$ edges in $G[B],$ while the other edges are in $G[A,B]$. This implies that $\sum_{b \in B} d_{\calF}(b) - \sum_{a \in A} d_{\calF}(a) = 2 (|B| - |A|)$. Hence, denoting by $k$ the number of paths in $\calF$ and using that every path in $\calF$ has one endpoint in $A$ and the other in $B$, we have
    \[ |U \cap B| - |U \cap A| = \frac{1}{2}\left( \sum_{b \in B} d_{\calF}(b) - k\right) - \frac{1}{2} \left( \sum_{a \in A} d_{\calF}(a) - k\right) = |B| - |A|, \]
    implying $|A\setminus U| = |B \setminus U|$, as claimed.
\end{proof}

\begin{proof}[Proof of Theorem~\ref{thm:dense-regular}]
    Clearly, we may assume that $n$ is sufficiently large, since otherwise, by taking $\eta$ small enough, the condition $d \ge n^{1- \eta}$ cannot be satisfied.
    
    We show that for sufficiently large $n$, the statement holds with $\eta = 10^{-15}$. Let $G$ be an $n$-vertex $d$-regular $\explet$-expander, where $d \ge n^{1 - \eta}$ and $\explet = n^{-\eta}.$ Denote $\delta = 10^{-7}$ and $\eps = n^{-\delta}$. We consider two cases.
    
    First suppose that $G$ is $\eps$-far from bipartite. Let $a_1b_1$ be an arbitrary edge in $G$. The set $\{a_1, b_1\}$ is $(2/d)$-uniform with respect to $G$. Let $c' = 1/5$ and note that by assumption, we have $(\explet \eps)^{-10^6 / c'} = (n^{-\eta} \cdot n^{-\delta})^{-10^6 / c'} \le n^{2/3} < d$. Thus, applying Theorem~\ref{thm:robust-far-from-bipartite} with $d' = 0, c =  c'= 1/5, \alpha = 2/d< d^{-c'}, m = 1 \le d^{-c'/900} n$ and the set of pairs $\{(a_1, b_1)\}$, we obtain a Hamilton path from $a_1$ to $b_1$ in $G$ which together with the edge $a_1b_1$ forms a Hamilton cycle, as needed.

    Now, suppose that $G$ is not $\eps$-far from bipartite and let $(A, B)$ be a maximum cut of $G$, i.e. a partition of $V(G)$ maximizing $e_G(A, B)$ and note that, by assumption, $e_G(A, B) \ge (1 - \eps) dn / 2$. We wish to apply Theorem~\ref{thm:robust-bipartite} to the bipartite subgraph $G[A, B]$. However, there are two things that we need to take care of beforehand. First, it is possible that $|A| \neq |B|$ and secondly, $G[A, B]$ need not be nearly regular. However, since $d \ge n^{1 - \eta},$ both of these properties nearly hold and we can easily ensure them using Lemma~\ref{lem:clean-up-forest} as follows.
    
    Without loss of generality, assume that $|B| \ge |A|.$ Since $G$ is $d$-regular, we have $d |B| = 2e_G(B) + e_G(B, A) \le 2e_G(B) + d|A|,$ which implies $e_G(B) \ge \frac{d}{2}(|B| - |A|)$. Since $G$ is $\eps$-close to bipartite, $e_G(B) \le \eps dn / 2,$ implying $|B| - |A| \le \eps n = n^{1-\delta}$. Since $(A, B)$ is a maxcut, we have $\Delta(G[A]), \Delta(G[B]) \le d/2$.
    
    Let $c = \delta/5$ and let $X$ denote the set of vertices in $G$ with degree at most $d - \frac{1}{2}d^{1-c}$ in $G[A, B]$, that is, with at least $\frac{1}{2} d^{1-c}$ neighbours inside their own part of the cut $(A,B)$. Since $G$ is $\eps$-close to bipartite, we have $\frac{1}{2} \cdot \frac{1}{2} d^{1-c} \cdot |X| \le \eps dn / 2,$ implying $|X| \le 2 \eps d^c n \le n^{1-\delta/2}$. Note that $|X| + |B| - |A| \le n^{1-\delta/2} + n^{1-\delta} \le 2n^{1-\delta/2} \le d / 10$. Let $\calF$ be the linear forest obtained by applying Lemma~\ref{lem:clean-up-forest}, with the partition $V(G) = A \cup B$ and the set $X$. (If Lemma \ref{lem:clean-up-forest} gives an empty linear forest, then let $\mathcal{F}$ consist of a single edge between $A$ and $B$.) We have $|V(\calF)| \le 4(|X| + |B| - |A|) \le 8n^{1-\delta/2}$ and denoting by $U$ the set of internal vertices of $\calF$, we have $X \subseteq U$ and $|A \setminus U| = |B \setminus U|.$ Let $k$ denote the number of paths in $\calF$, and for $i \in [k],$ let the endpoints of the $i$-th path be $(x_i, y_i)$ with $x_i \in A, y_i \in B$.
    
     Let $G' = G[A,B] \setminus U$. Denote $A' = A \setminus U$ and $B' = B \setminus U$ and note that $G'$ is bipartite with parts $A'$ and $B'$, where $|A'| = |B'|$, since $|A \setminus U| = |B \setminus U|$. Denote $n' = |V(G')| \ge n - |V(\calF)| \ge  n/2$. We claim that $G'$ is a $(d \pm d')$-nearly regular $\explet'$-expander, where $d' = d^{1-c}$ and $\explet' = \explet / 2$. For each $v \in V(G'),$ since $v\not \in X$, we have $d_{G'}(v) \ge d - \frac{1}{2} d^{1-c} - |U| \ge d - d^{1-c}$ and clearly $d_{G'}(v) \le d$, so $G'$ is $(d \pm d')$-nearly regular. Now, consider an arbitrary set $S \subseteq V(G')$ of size $|S| \le 2n' / 3 \le 2n / 3$. Using that $G$ is a $\explet$-expander and that each vertex in $V(G')$ has at most $\frac{1}{2}d^{1-c}$ neighbours in its own part in $G$, we have
     \begin{align*}
        e_{G'}(S, V(G') \setminus S) &\ge e_G(S, V(G) \setminus S) - \frac{1}{2} d^{1-c} |S| - e_G(S, U) \ge \explet d |S| - d^{1-c} |S| - |U| |S|\\
        &\ge  \explet d|S| - d^{1-c} |S| - 8n^{1 - \delta/2} |S| \ge \explet' d |S|.
    \end{align*} 
    Since $\bar{d}(G') \le d,$ it follows that $G'$ is a $\explet'$-expander. 
    
    Let $m = k$ and for $i \in [k],$ let $a_i = x_i$ and $b_i = y_{i+1},$ where we denote $y_{k+1} = y_1$. Note that $m \le |V(\calF)| \le 8n^{1-\delta/2}$. We apply Theorem~\ref{thm:robust-bipartite} to the graph $G'$ with the set of pairs $\{ (a_i, b_i) \mid i \in [m] \}$ and with $\explet' = n^{-\eta} / 2$ in place of $\explet$. Let us verify the requirements for this application. Denote $\alpha = d^{-c}$ and observe that $\alpha (d-d') \ge m$, so $\{ a_i \mid i \in [m]\}$ and $\{ b_i \mid i \in [m]\}$ are $(A', B', \alpha)$-bi-uniform and $(B', A', \alpha)$-bi-uniform with respect to $G'$, respectively and these sets are both of size $m$ since the vertices $a_1, \dots a_m, b_1, \dots b_m$ are all distinct. Furthermore, we have $d' = d^{1-c}, m \le 8n^{1-\delta/2} \le d^{-c/900} n_2$ and $(\explet')^{-10^6 / c} \le n^{-2\eta \cdot (-10^6/c)} = n^{10^7 \eta / \delta} = n^{1 / 10} < d$.

    Thus, using Theorem~\ref{thm:robust-bipartite}, for each $i \in [k]$, we obtain an $x_i - y_{i+1}$ path $P_i$, where these paths are pairwise vertex disjoint and $\bigcup_{i \in [k]} \Vint(P_i) = V(G') \setminus \bigcup_{i \in [k]} \{x_i, y_i\}$. It is easy to see that $\calF \cup \bigcup_{i \in [k]} P_i$ is a Hamilton cycle in $G$, finishing the proof.
\end{proof}

\subsection{Hamiltonicity of edge-percolated Cayley graphs}

As a second application, we recover the recent result of Bedert, Dragani\'{c}, M\"{u}yesser and Pavez-Sign\'{e}~\cite{bedertdraganic} about Hamiltonicity of moderately dense Cayley graphs and moreover we prove that Hamiltonicity is preserved after (modest) edge-percolation. Given a group $\Gamma$ and a symmetric set $S \subseteq \Gamma$ (i.e. such that $s^{-1} \in S$ for all $s \in S$) not containing the identity, the \emph{Cayley graph} $\Cay_\Gamma(S)$ is the graph with vertex set $\Gamma$ and edge set $\{ (g, gs) \mid g \in \Gamma, s \in S \}$.

Recall that $G_p$ denotes the random graph obtained from $G$ by keeping each edge independently with probability $p$ and that we wish to prove the following.

\thmlovaszpercolatedrestate*

The following result was proved by Bedert, Buci\'{c}, Kravitz, Montgomery and M\"{u}yesser~\cite{bedert2025graham} in their work on Graham's rearrangement conjecture and is the starting point for the proof of Theorem~\ref{thm:lovasz-percolated} as well as in~\cite{bedertdraganic}. The result can be viewed as a weak version of the arithmetic regularity lemma of Green~\cite{green}; we refer to \cite{bedert2025graham} for a detailed discussion of its advantages and drawbacks.

\begin{theorem}[{\cite[Corollary~4.1,~rephrased]{bedert2025graham}}] \plabel{thm:weak-ARL}
    Let $\sigma \in (0, 1]$ and $\eps \in (0,1/2)$. Let $\Gamma$ be a finite group, and let $S \subseteq \Gamma$ be a symmetric subset of $\Gamma$ with density $\sigma = |S| / |\Gamma|$. Then, there is a subgroup $\Gamma'$ of $\Gamma$ such that
    \begin{enumerate}[label=(\arabic*)]
        \item $|S \cap \Gamma'| \ge (1 - \eps) |S|$;
        \item $\Cay_{\Gamma'}(S \cap \Gamma')$ is a $\explet$-expander, for $\explet = \frac{\eps \sigma^2}{3000}$.
    \end{enumerate}        
\end{theorem}

We shall also use the following result about edge-connectivity of vertex-transitive graphs, proved by Mader~\cite{mader} and Watkins~\cite{watkins} independently.
\begin{lemma}[e.g.~{\cite[Lemma~3.3.3]{godsil-royle}}] \plabel{lem:edge-connectivity}
    Any connected vertex-transitive $d$-regular graph is $d$-edge-connected.
\end{lemma}

Recalling our main theorems as well as the proof of Theorem~\ref{thm:dense-regular}, it should come as no surprise that we need to take special care to deal with graphs which are close to being bipartite. We shall use the following lemma of Christofides, Hladk\'{y} and M\'{a}th\'{e}~\cite{christofides} from their proof that Cayley graphs, and more generally, vertex-transitive graphs with linear degree are Hamiltonian.

Before we state this lemma, let us introduce the notion of iron connectivity used in~\cite{christofides}. A graph $G$ is said to be \emph{$\ell$-iron} if $G$ stays connected after simultaneous removal of an arbitrary edge-set $E' \subseteq E(G)$ with $\Delta(E') \le \ell$ and an arbitrary vertex-set $U \subseteq V(G)$ with $|U| \le \ell$. 
\begin{lemma}[{\cite[Lemma~16]{christofides}}] \plabel{lem:transitive-close-to-bip}
    Let $c \in (0, 1/17)$ be arbitrary. Suppose that $G$ is a $cn$-iron vertex-transitive graph on $n$ vertices and that $G$ can be made bipartite by removing at most $c^4 n^2$ edges. Then there exists a bipartition $V(G) = A \cup B$ such that $|A| = |B|$ and for each $u \in A$ and each $v \in B$, we have $|N(u) \cap A| \le 6c^2 n$ and $|N(v) \cap B| \le 6c^2 n$. Furthermore, we have $g(A) = A$ or $g(A) = B$ for each $g \in \mathrm{Aut}(G)$.
\end{lemma}

In order to use Lemma~\ref{lem:transitive-close-to-bip}, we prove the following simple relation between iron-connectivity and expansion.
\begin{lemma} \plabel{lem:expander-to-iron}
    Any $d$-regular $\explet$-expander is $\ell$-iron for $\ell = \explet d / 4$.
\end{lemma}
\begin{proof}
    Consider an arbitrary edge set $E' \subseteq E(G)$ with $\Delta(E') \le \ell$ and an arbitrary vertex set $U \subseteq V(G)$ with $|U| \le \ell$, where $\ell = \explet d / 4$. Let $G'$ be obtained from $G$ by removing the vertex set $U$ and the edge set $E'$. We need to show that $G'$ is connected. Assume the contrary and let $S$ be a connected component of $G'$ of size at most $|V(G')| / 2 \le n/2$. Let $F = E(G[S, V(G) \setminus S])$. Since $S$ is a connected component of $G'$, we have $F \subseteq E' \cup E(G[S, U])$, implying $|F| \le \Delta(E') |S| + |U| |S| \le 2 \ell |S|$. On the other hand, since $G$ is a $\explet$-expander, we have $|F| \ge \explet d |S| > 2 \ell |S|,$ a contradiction.
\end{proof}

Finally, we shall use the following simple result about expansion after edge percolation. The first part of this lemma follows, for example, from the work of Karger~\cite{karger}.

\begin{lemma} \plabel{lem:edge-percolated-expander}
    Let $G$ be a $(d \pm d')$-nearly regular $n$-vertex $\explet$-expander, where $d' \le d / 100$, let $p \in [0,1]$ satisfy $p \explet d \ge 200 \log n$ and assume that $n$ is sufficiently large. Then, with probability at least $1 - n^{-4}$, $G_p$ is a $(pd \pm 10 (pd' + \sqrt{ pd \log n}))$-nearly regular $(\explet / 4)$-expander. Moreover, if $G$ is $\eps$-far from bipartite and $p \eps d \ge 100$, then with probability at least $1 - n^{-4}$, $G_p$ is $(\eps / 4)$-far from bipartite.
\end{lemma}
\begin{proof}
    Denote $d'' = 10 (pd' + \sqrt{ pd \log n})$. By a Chernoff bound, for any vertex $v \in V(G)$, the probability that $|d_{G_p}(v) - p d_G(v)| > d''$ is at most $2 e^{-100 \log n / 6} < n^{-10}.$ Hence, with probability, at least $1 - n^{-9},$ $G_p$ is $(pd \pm d'')$-nearly regular. Let us assume this holds, so in particular, $\bar{d}(G_p) = pd \pm 10 (pd' + \sqrt{pd \log n})$, which further implies $\bar{d}(G_p) \le 2pd$. Let $S \subseteq V(G)$ be an arbitrary set of size $|S| \le 2n / 3.$ Since $G$ is a $\explet$-expander, we have $e_G(S, V(G) \setminus S) \ge \explet \bar{d}(G) |S| \ge \frac{99}{100} \explet d |S|$. By a Chernoff bound, we have 
    \[ \Pr[e_{G_p}(S, V(G) \setminus S) < (\explet/2) \cdot p d |S|] \le \exp\left( - \explet p d |S| / 13 \right) \le \exp\left(-7 |S| \log n \right), \]
    where we used $p \explet d \ge 200 \log n$. Recalling that if $G$ is $(pd \pm d'')$-nearly regular, then $\bar{d}(G_p) \le 2pd$, it follows that $G_p$ is a $(pd \pm d'')$-nearly regular $(\explet / 4)$-expander with probability at least
    \[ 1 - n^{-9} - \sum_{s = 1}^{2n/3} \binom{n}{s} \exp\left(-7 s \log n \right) \ge 1 - n^{-9} - \sum_{s = 1}^{2n/3} n^s n^{-7s} \ge 1 - n^{-4}, \]
    proving the first part of the lemma.

    Now, assume that, additionally, $G$ is $\eps$-far from bipartite and we aim to show that $G_p$ is $(\eps/4)$-far from bipartite with probability at least $1 - n^{-4}.$ Indeed, consider an arbitrary partition $A \cup B = V(G)$. Since $G$ is $\eps$-far from bipartite, we have $e_G(A) + e_G(B) \ge \frac{99}{100} \eps dn / 2.$ By a Chernoff bound, we have $e_{G_p}(A) + e_{G_p}(B) \ge p \eps dn / 4$ with probability at least $1 - e^{-p \eps d n / 30} \ge 1 - 4^{-n}$. Union bounding over all possible partitions and noting that $e(G_p) \le p dn$ with probability at least $1 - n^{-9}$, it follows that $G_p$ is $(\eps/4)$-far from bipartite with probability at least $1 - 2^n \cdot 4^{-n} - n^{-9} \ge 1 - n^{-4}$, as needed.
\end{proof}

We are ready to present the proof of Theorem~\ref{thm:lovasz-percolated}.
\begin{proof}[Proof~of~Theorem~\ref{thm:lovasz-percolated}]
    We show that the statement holds with $C = 10^{20}$. Since the statement of Theorem \ref{thm:lovasz-percolated} is asymptotic in nature, we may assume that $n$ is sufficiently large. Let $\Gamma$ be an arbitrary finite group of order $n$ and let $S \subseteq \Gamma$ be a symmetric set of size $d = \sigma n$ not containing the identity such that $\Cay_\Gamma(S)$ is connected, or equivalently, $S$ is a generating set of $\Gamma$. By monotonicity, we may assume that $pd = \left( \frac{n}{d} \log n\right)^C = (\sigma^{-1} \log n)^C$.
    
    Applying Theorem~\ref{thm:weak-ARL} with $\eps = 1/3$, we obtain a subgroup $\Gamma' \subseteq \Gamma$ such that for $S' = S \cap \Gamma',$ we have $|S'| \ge \frac{2}{3} d$ and $\Cay_{\Gamma'}(S')$ is a $\explet$-expander for $\explet = \sigma^2 \log^{-3} n$, where we replaced the factor of $1/9000$ by the smaller $\log^{-3} n$ so that we always have $\explet < 1 / (\log n)^2$ needed to apply our main theorems. Denote $H = \Cay_{\Gamma'}(S')$, $d_H = |S'|$ and $n_H = |\Gamma'|$ and note that $H$ is a $d_H$-regular graph on $n_H$ vertices, where $n_H \ge d_H \ge \frac{2}{3} d$. Let $k = \frac{n}{n_H}$ and note that $k \le \frac{3}{2} \sigma^{-1}$. Let $V_1, \dots, V_k$ be an arbitrary ordering of the set of left cosets $\{ g \Gamma' \mid g \in \Gamma \}$ of $\Gamma'$ in $\Gamma$.
    
    Since $\Cay_\Gamma(S)$ is connected, there exists a symmetric set $T \subseteq S \setminus \Gamma'$ of size $|T| \le 2(k-1)$ such that $\Cay_\Gamma(S' \cup T)$ is connected. We denote $G = \Cay_\Gamma(S' \cup T)$ and we will show that with high probability $G_p$ is Hamiltonian which clearly suffices since $G \subseteq \Cay_\Gamma(S)$.  Note that, for each $i \in [k],$ the induced subgraph $G[V_i]$ is isomorphic to $H$ and that $H$ is connected since it is a $\explet$-expander for $\explet > 0$.

    Let $\delta = 10^{-7}$ and $\eps = (pd_H)^{-\delta}$. We consider two cases, starting with the simpler one.

    \textbf{Case 1: $H$ is $\eps$-far from bipartite.} Let $F$ be an auxiliary graph with vertex set $V(F) = \{V_1, \dots, V_k\},$ where for $1 \le i < j \le k$, we have $V_iV_j \in E(F)$ if and only if in $G[V_i,V_j]$ there is a matching of size at least $s \coloneqq \sigma^3 d / 100$.

    \begin{claim} \plabel{claim:F-connected-far-from-bip}
        $F$ is connected.
    \end{claim}
    
    \begin{proof}
        Suppose not and let $W \subseteq V(F)$ be a connected component of $F$ of size at most $k/2$. Consider arbitrary $V_i \in W, V_j \in V(F) \setminus W$. Note that all edges in the induced subgraph $G[V_i, V_j]$ are of the form $\{x, xt\}, x \in V_i, t \in T,$ so $\Delta(G[V_i, V_j]) \le |T|.$ Since $V_iV_j \not\in E(F),$ in $G$ there is no matching of size $s$ between $V_i$ and $V_j,$ so $e_G(V_i, V_j) \le 2 |T|\cdot s$. It follows that $e_G(\bigcup_{V_i \in W} V_i, \bigcup_{V_j \in V(F) \setminus W} V_j) \le k^2 \cdot 2|T| s \le 4 k^3 s \le 4\left( \frac{3}{2} \sigma^{-1} \right)^3 s < d/2$. However, by Lemma~\ref{lem:edge-connectivity}, $G$ is $\frac{2}{3}d$-edge-connected since $|S'| \ge \frac{2}{3}d,$ a contradiction.
    \end{proof}

    We shall assume that, for every $V_i V_j \in E(F),$ there is a matching of size at least $sp / 2$ in $G_p[V_i, V_j]$ and moreover, that for all $i \in [k],$ the subgraph $G_p[V_i]$ is a $(pd_H \pm (pd_H)^{2/3})$-nearly regular $(\eps/4)$-far from bipartite $(\explet/4)$-expander. Let us justify that these hold with high probability. By a Chernoff bound, the former property holds with probability at least $1 - k^2 \cdot e^{-sp / 12} = 1 - o(1)$, where we used that $sp = \sigma^3 p d / 100 \ge (\log n)^C / 100.$ Recalling that $G[V_i] \cong H,$ the latter property follows from Lemma~\ref{lem:edge-percolated-expander} and a union bound over all $k$, where we used that $(pd_H)^{2/3} > 10 \sqrt{pd_H \log n_H}, \;p \gamma d_H > 200 \log n$ and $p \eps d_H = (pd_H)^{1-\delta} \ge 100$.

    Let $F'$ be an arbitrary spanning tree in $F$ and let $F_2$ be the multigraph obtained by doubling each edge in $F'$. For each $e = V_i V_j \in E(F_2)$, let $\phi(e)$ be an edge in $G_p[V_i, V_j]$ such that $M = \{ \phi(e) \mid e \in E(F_2)\}$ is a matching. Note that we can greedily construct $M$ since, by assumption, for every $V_iV_j \in E(F_2),$ there is a matching of size $sp / 2 > \sigma^3 d p / 200$ in $G_p[V_i, V_j]$ and $|E(F_2)| \le 2k \le 3 \sigma^{-1} < \sigma^3 d p / 200,$ where in the last inequality we used that $pd \ge (\sigma^{-1} \log n)^C$.
    
    Since $F_2$ is connected and every vertex has even degree, there exists an Euler tour $U_0, e_1, U_1, e_2, U_2, \dots, e_{\ell}, U_\ell=U_0$ in $F_2$, where $\ell = |E(F_2)|$ and for each $i \in [\ell],$ $U_i \in V(F_2)$ and $e_i \in E(F_2)$.

    The edges in $M$ will serve as connections between the different cosets and it remains to cover all of the vertices within each coset with appropriate paths. So, fix $i \in [k]$ and let $p_1,\dots, p_{r_i} \in [1, \ell]$ with $p_1 < \dots < p_{r_i}$ be the positions where $V_i$ appears in the Euler tour, i.e. such that $U_{p_j} = V_i$ for all $j \in [r_i]$. For $j \in [r_i],$ let $a^i_j$ be the unique vertex in $V_i \cap \phi(e_{p_j})$ and let $b^i_j$ be the unique vertex in $V_i \cap \phi(e_{p_j+1})$, where we denote $e_{\ell+1} = e_1$.

    Denote $c = 1/5$ and $\alpha = (pd_H)^{-c}$. Recall that $G_p[V_i]$ is $(pd_H \pm (pd_H)^{2/3})$-nearly regular and $(\eps/4)$-far from bipartite. Since $r_i \le |E(F_2)| \le 2k \le 3\sigma^{-1},$ the set $\{a^i_j \mid j \in [r_i]\}\cup \{b^i_j \mid j \in [r_i]\}$ is $\alpha$-uniform with respect to $G_p[V_i],$ where we used that $2r_i \le 6\sigma^{-1} \le \alpha pd_H / 4$ and $\bar{d}(G_p[V_i]) \ge pd_H /2$.
    
    We apply Theorem~\ref{thm:robust-far-from-bipartite} to the graph $G_p[V_i]$ with $m=r_i$ and the set of pairs $\{ (a^i_j, b^i_j) \mid j \in [m] \}.$ Let us verify the inequalities for this application. Here $d' = (pd_H)^{2/3} \le (pd_H)^{1-c}, \alpha = (pd_H)^{-c}$ as discussed, and we have $m \le 3\sigma^{-1} \le 3 n^{1/C} < (pd_H)^{-c/900} n_H$ and $((\explet/4) (\eps/4))^{-10^6/c} = (16 / \explet)^{5 \cdot 10^6} \cdot (pd_H)^{\delta \cdot 5 \cdot 10^6} \le (\sigma^{-2} (\log n)^4)^{5 \cdot 10^6} \cdot (pd_H)^{1/2} \le pd_H$. Hence, by Theorem~\ref{thm:robust-far-from-bipartite}, we obtain internally vertex disjoint paths $P^i_1, \dots, P^i_{r_i},$ where for $j \in [r_i],$ $P^i_j$ is an $a^i_j - b^i_j$ path and $\bigcup_{j \in [r_i]} \Vint(P^i_j) = V_i \setminus \bigcup_{j \in [r_i]} \{ a^i_j, b^i_j\} = V_i \setminus V(M)$.

    We obtain such paths for all $i$ and it is easy to see that $M \cup \bigcup_{i \in [k]} \bigcup_{j \in [r_i]} P^i_j$ is a Hamilton cycle in $G_p$, as needed.

    \textbf{Case 2: $H$ is $\eps$-close to bipartite.} Let $c' = \eps^{1/4}$. Observe that by Lemma~\ref{lem:expander-to-iron}, $H$ is $c' n_H$-iron, where we used that $\explet d_H / 4 \ge \sigma^2 (\log n)^{-3} d / 6 = \sigma^3 (\log n)^{-3} n / 6$, while $c' n_H \le \eps^{1/4} n \le (p d/2)^{-\delta/4} n \le (\sigma^{10} \log^{-10} n) n$. Since $H$ is $\eps$-close to bipartite, it can be made bipartite by removing at most $\eps e(H) \le \eps n_H^2 = (c')^4 n_H^2$ edges. Since $n$ is large enough, we have $c' < 1/17,$ so we may apply Lemma~\ref{lem:transitive-close-to-bip} to the graph $H$ with parameter $c'$ in place of $c$, to obtain a bipartition $A \cup B = V(H)$ such that $|A| = |B|$, for all $a \in A, b \in B,$ we have $|N_H(a) \cap A|, |N_H(b) \cap B| \le 6(c')^2 n_H \le c' n_H$ and for each $g \in \mathrm{Aut}(H),$ it holds that $g(A) = A$ or $g(A) = B$.
    
    Let $H' = H[A, B]$ and note that $H'$ is $(d_H \pm c' n_H)$-nearly regular. Next, we claim that $H'$ is a $(\explet/2)$-expander. Indeed, consider an arbitrary set $X \subseteq V(H')$ of size at most $2n_H / 3$. Since $H$ is a $\explet$-expander, we have $e_{H'}(X, V(H') \setminus X) \ge e_H(X, V(H) \setminus X) - |X| \cdot c' n_H \ge \explet d_H |X| - c' n_H |X| \ge (\explet/2) d_H |X|,$ where we used that $c' n_H \le c' n \le \frac{1}{100} \sigma^{3} (\log n)^{-4} n$ and $\explet d_H \ge \explet \cdot \frac{2}{3}d > \sigma^3 (\log n)^{-4} n$.
    
    Recall that for each $i \in [k],$ there is an element $x_i \in \Gamma$ such that $V_i = x_i \Gamma'$. For $i \in [k],$ let us denote $A_i = x_i A$ and $B_i = x_i B$ and note that $A_i \cup B_i$ is a bipartition of $V_i$ and $G[A_i, B_i] \cong H'$. Let $F$ be the graph with vertex set $V(F) = \{A_1, B_1, \dots, A_k, B_k\}$ where $XY \in E(F)$ if and only if there is a matching of size at least $s \coloneqq \sigma^3 d / 100$ between $X$ and $Y$ in $G$.

    \begin{claim}
        $F$ is connected.
    \end{claim}
    
    \begin{proof}
        First, let us observe that for any $i \in [k],$ $A_iB_i \in E(F)$. This trivially holds since $G[A_i,B_i] \cong H'$ and $H'$ has minimum degree at least $d_H - c' n_H \ge d_H / 2 \ge 2s$. Now, suppose $F$ is not connected and consider a connected component $W \subseteq V(F)$ of $F$ of size at most $|V(F)| / 2$. By the above, we have $A_i \in W \iff B_i \in W$ for any $i \in [k]$. For any two sets $X, Y$ in $V(F),$ where $X \in \{A_i, B_i\} ,Y \in \{A_j, B_j\}$ for $i \neq j$, the edges between $X$ and $Y$ in $G$ are of the form $\{ x, tx \}$ with $t \in T$, implying that $\Delta(G[X,Y]) \le |T|$. Hence, for any $X \in W, Y \in V(F) \setminus W,$ we have $e_G(X, Y) \le 2 |T| s$. This implies that setting $Z = \bigcup_{X \in W} X \subseteq V(G)$, we have $e_G(Z, V(G) \setminus Z) \le |V(F)|^2 \cdot (2 |T| s) \le (2k)^2 \cdot 4k \cdot s \le 16 (\frac{3}{2} \sigma^{-1})^3 s < \frac{2d}{3}.$ However, by Lemma~\ref{lem:edge-connectivity}, $G$ is $\frac{2}{3}d$-edge-connected since $|S'| \ge \frac{2}{3}d,$ a contradiction.
    \end{proof}

    We proceed similarly as in the first case. The key difference is that we are working with the bipartite graph $H'$. Initially, $H'$ has parts of equal size so we would like to make the connections between different $V_i$'s that preserve this property. To this end, we prove the following.
    \begin{claim}
        $F$ is vertex-transitive.
    \end{claim}
    
    \begin{proof}
        For an arbitrary element $y \in \Gamma$, define $f_y \colon V(G) \rightarrow V(G)$ as $f_y(g) = yg$ and note that $f_y$ is an automorphism of $G$. We claim that for any set $X \in \{ A_1, B_1, \dots, A_k, B_k \}$ and any $y \in V(G),$ $f_y(X) \in \{ A_1, B_1, \dots, A_k, B_k \}$. Without loss of generality, assume that $X = A_i$ for some $i \in [k]$. Note that $f_y(V_i) = yx_i \Gamma' = V_j$ for some $j \in [k]$. Let $z = y x_i x_j^{-1}$ and note that $f_z(V_j) = y x_i x_j^{-1} x_j \Gamma' = y x_i \Gamma' = V_j$. Thus, $f_z$ maps $V_j$ to $V_j$ so its restriction to $V_j$ is an automorphism of $G[V_j]$. Note that $f_z = f_y \circ f_{x_i x_j^{-1}}$ and that $f_{x_i x_j^{-1}}(A_j) = x_i x_j^{-1} x_j A = x_i A = A_i$. Recalling that $G[V_j] \cong H$, we have $f_z(A_j) \in \{A_j, B_j\}$ by the last property inherited from Lemma~\ref{lem:transitive-close-to-bip}. It follows that $f_y(A_i) = f_z(A_j) \in \{ A_j, B_j \}$, as claimed.

        Using the above, for each $y \in \Gamma$, we may naturally define a bijection $\Phi_y \colon V(F) \rightarrow V(F)$, where for $X \in V(F),$ we define $\Phi_y(X) = Y$ if $f_y(X) = Y$ when viewed as a mapping from $V(G)$ to $V(G)$. We claim that for any $y \in \Gamma$, $\Phi_y$ is an automorphism of $F$. Indeed, for any distinct $X, Y \in V(F)$, as $f_y$ is an automorphism of $G$, we have $G[X, Y] \cong G[f_y(X), f_y(Y)],$ so $XY \in E(F) \iff \Phi_y(X) \Phi_y(Y) \in E(F)$, as needed.

        Finally, consider arbitrary $X, Y \in V(F)$. Clearly, there exists an element $y \in \Gamma$ such that $f_y(X) \cap Y \neq \emptyset.$ By the above, $\Phi_y$ is an automorphism of $F$ such that $\Phi_y(X) = Y$, proving that $F$ is vertex-transitive.
    \end{proof}

    We shall assume that, for every $XY \in E(F)$, there is a matching of size at least $sp / 2$ in $G_p[X, Y]$ and moreover, that for all $i \in [k]$, the subgraph $G_p[A_i, B_i]$ is a $(pd_H \pm (pd_H)^{1 - c})$-nearly regular $\explet'$-expander for $c = \delta / 10 = 10^{-8}$ and $\explet' = \gamma / 8$. Let us argue that these properties hold with high probability. By a Chernoff bound, the former property holds with probability at least $1 - (2k)^2 \cdot e^{-sp / 12} = 1 - o(1)$. Recall that for each $i \in [k],$ the graph $G[A_i, B_i]$ is isomorphic to $H[A, B]$ which is a $(d_H \pm c' n_H)$-nearly regular $(\explet/2)$-expander. Note that $c' n_H \le d_H / 100$, $p (\explet/2) d_H \ge 200 \log n_H$ and $10 (p c' n_H + \sqrt{pd_H \log n_H}) \le 10 (2pd_H c' \sigma^{-1} + (pd_H)^{2/3}) \le (pd_H)^{1-c},$ where in the last inequality we used that $c' = \eps^{1/4} < (pd_H)^{-2c} < \sigma (pd_H)^{-c} / 100$. Thus, by Lemma~\ref{lem:edge-percolated-expander}, we conclude that with probability at least $1 - k \cdot n_H^{-4} = 1 - o(1)$, for all $i \in [k],$ the graph $G_p[A_i, B_i]$ is a $(pd_H \pm (pd_H)^{1 - c})$-nearly regular $\explet'$-expander, as required.

    Let $F_2$ be the multigraph obtained by doubling each edge in $F$ and let $M$ be a matching containing an edge $\phi(e) \in G_p[X,Y]$ for every $e = XY \in E(F_2)$, where we take multiplicity into account. Note that we can greedily construct $M$, since $|E(F_2)| \le 2 |V(F)|^2 = 2\cdot (2k)^2 \le 32 \sigma^{-2} < sp / 2.$ Since $F_2$ is connected and each vertex is of even degree, it contains an Euler tour $U_0, e_1, U_1, e_2, U_2, \dots, U_\ell = U_0$, where $\ell = |E(F_2)|$, and for $i \in [\ell],$ we have $U_i \in V(F_2)$ and $e_i \in E(F_2)$.

    Let us denote by $2r$ the degree of every vertex in $F_2$. Consider a fixed $i \in [k]$. Let $p_1,\dots, p_r \in [1, \ell]$ with $p_1 < \dots < p_r$ be the positions where $A_i$ appears in the Euler tour, i.e. such that $U_{p_j} = A_i$ for all $j \in [r]$. For $j \in [r],$ let $v_j$ be the unique vertex in $A_i \cap \phi(e_{p_j})$ and let $w_j$ be the unique vertex in $A_i \cap \phi(e_{p_j+1})$, where we denote $e_{\ell+1} = e_1$. Choose distinct vertices $z_1,\dots,z_r \in B_i \setminus V(M)$. Note that this is possible since $3r < |B_i|$. For each $j\in [r]$, let $(a^i_{2j-1}, b^i_{2j-1}) = (v_j, z_j)$ and $(a^i_{2j}, b^i_{2j}) = (w_j, z_j)$.

    For each occurrence of $B_i$, we analogously define two pairs to be connected through $V_i$. Namely, let $p'_1,\dots, p'_r \in [1, \ell]$ with $p'_1 < \dots < p'_r$ be the positions where $B_i$ appears in the Euler tour. Then, for $j \in [r],$ let $v'_j$ be the unique vertex in $B_i \cap \phi(e_{p'_j})$ and let $w'_j$ be the unique vertex in $B_i \cap \phi(e_{p'_j+1})$, where we denote $e_{\ell+1} = e_1$. Choose distinct vertices $z'_1,\dots,z'_r \in A_i \setminus V(M)$. Note that this is possible since $3r < |A_i|$. For each $j\in [r]$, let $(a^i_{2r + 2j-1}, b^i_{2r + 2j-1}) = (z'_j, v'_j)$ and $(a^i_{2r+2j}, b^i_{2r+2j}) = (z'_j, w'_j)$.

    Let $m = 4r$ and let $Q_i = \{ (a^i_j, b^i_j) \mid j \in [m] \}$. Observe that $a^i_j \in A_i, b^i_j \in B_i$ for all $j \in [m]$. We aim to apply Theorem~\ref{thm:robust-bipartite} to the graph $G_p[A_i, B_i]$ with the set of pairs $Q_i$. Let us verify the conditions for this application. Note that the set $\{ a^i_j \mid j \in [m] \}$ equals $(V(M) \cap A_i) \cup \{z_1', \dots, z_r'\}$ so it has size $3r$, and analogously, the set $\{ b^i_j \mid j \in [m] \}$ has size $3r$.

    Recall that $G_p[A_i, B_i]$ is a $(pd_H \pm (pd_H)^{1-c})$-nearly-regular $\explet'$-expander. Let $\alpha = (pd_H)^{-c}$ and note that $(\explet')^{-10^{6} / c} \le (\sigma^{-2} \log^{4} n)^{10^{14}} \le pd_H$. Additionally, we have $m = 4r \le 8k \le 12 \sigma^{-1} \le 12 n^{1/C} \le (pd_H)^{-c/900} n_H,$ which further implies that $\{ a^i_j \mid j \in [m]\}$ is $(A_i, B_i, \alpha)$-bi-uniform with respect to $G_p[A_i,B_i]$ both as a set and as a multiset since $\alpha(pd_H - (pd_H)^{1-c}) \ge m$, and, analogously, $\{ b^i_j \mid j \in [m]\}$ is $(B_i, A_i, \alpha)$-bi-uniform with respect to $G_p[A_i,B_i]$ both as a set and as a multiset. Thus, applying Theorem~\ref{thm:robust-bipartite}, we obtain a collection of internally vertex disjoint paths $\{ P^i_j \mid j \in [m]\}$ such that $P^i_j$ has endpoints $a^i_j$ and $b^i_j$ and $\bigcup_{j \in [m]} \Vint(P^i_j) = V_i \setminus \bigcup_{ j \in [m]} \{a^i_j, b^i_j\}$.

    We obtain such paths for each $i \in [k]$. By construction, it is easy to see that $M \cup \bigcup_{i \in [k], j \in [m]} P^i_j$ is a Hamilton cycle in $G_p,$ as needed.    
\end{proof}

\section{Concluding remarks} \label{sec:concluding}
In close-to-bipartite graphs, random walks behave much less nicely than in bipartite or far-from-bipartite graphs, thus our methods work poorly in that case. For $\explet = 1 / (\log n)^2$, say, we can prove Hamiltonicity of close-to-bipartite $d$-regular $\explet$-expanders only when $d \ge n^{1-\eta},$ for some small $\eta$ as given by Theorem~\ref{thm:dense-regular}. On the other hand, for bipartite or $\explet$-far-from-bipartite graphs, Theorems~\ref{thm:main-bipartite}~and~\ref{thm:main-far-from} only require $d \ge (\log n)^K$, for some constant $K$. We believe that $d \ge (\log n)^K$ should be sufficient even for close-to-bipartite graphs.

\begin{conjecture}
    There exists a constant $C$ such that any $n$-vertex $d$-regular $\explet$-expander with $d \ge (\explet^{-1} \log n)^C$ is Hamiltonian.
\end{conjecture}

Even for constant expansion $\explet$, our main results require $d \ge (\log n)^K$, while it is possible that constant degree already suffices. In fact, we believe that the dependence between $d$ and $\explet$ should be polynomial, i.e. that one can simply remove the logarithmic factor in Theorems~\ref{thm:main-bipartite} and Theorem~\ref{thm:main-far-from}.

\begin{conjecture}
    There exists a constant $C$ such that any bipartite $d$-regular $\explet$-expander with $d \ge (C\explet^{-1})^C$ is Hamiltonian.
\end{conjecture}

\begin{conjecture}
    There exists a constant $C$ such that the following holds. Let $G$ be a $d$-regular $\explet$-expander that is $\explet$-far from bipartite. If $d \ge (C\explet^{-1})^C$, then $G$ is Hamiltonian.
\end{conjecture}

Finally, recall that our approach is based on random walks for which we require good edge expansion. It is possible that edge expansion can be replaced by vertex expansion, but this would most likely require different methods.

\medskip

\textbf{Acknowledgements.} We thank Benny Sudakov for suggesting we use our main results to give a new proof of Theorem~\ref{thm:lovasz} and Alp M\"{u}yesser for helpful comments on an earlier draft. The second author thanks Abhishek Methuku for helpful discussions related to the topic of this paper.

\bibliographystyle{abbrv}
\bibliography{references}

\appendix
\section{Connecting pairs of vertices by vertex-disjoint paths through a random vertex set} 
\label{sec: sublinear expander}

In this appendix, we prove Lemma \ref{lem:randomsetconnecting}, which we now restate for the reader's convenience.

\lemconnectingingrestate*

As mentioned before, our proof follows very closely the proof of Lemma 3.4 from \cite{chakraborti2025edge}. We now define a notion of robust vertex expansion, very closely related to the notion used in \cite{chakraborti2025edge}.

\begin{defn} \label{defn:robust vertex expander}
    An $n$-vertex graph $G$ is a $(\explet,s)$-expander if, for every $U\subseteq V(G)$ and $F\subseteq E(G)$ with $1\leq |U|\leq \frac{2}{3}n$ and $|F|\leq s|U|$, we have
    $$|N_{G-F}(U)|\geq \explet |U|.$$
\end{defn}

Note that an $(\eps,s)$-expander in the sense of \cite{chakraborti2025edge} (see Definition 3.1 there) is an $(\frac{\eps}{(\log n)^2},s)$-expander in the sense of Definition \ref{defn:robust vertex expander}. We now relate our notions of edge and robust vertex expanders.

\begin{lemma} \label{lem:notions of expander}
    Let $G$ be an $n$-vertex $(d\pm d')$-nearly-regular $\explet$-expander. If $d' \le d/4$, then $G$ is also a $(\explet',s)$-expander with $\explet' = \frac{2}{5} \explet $ and $s = \explet d / 4.$
\end{lemma}
\begin{proof}
    Let $G$ be an $n$-vertex $(d\pm d')$-nearly-regular $\explet$-expander. Let $U \subseteq V(G)$ and $F \subseteq E(G)$ be such that $1 \le |U| \le \frac{2}{3} n$ and $|F| \le s |U|$ and denote $W = N_{G-F}(U).$ We need to show that $|W| \ge \frac{2 \explet |U|}{5}.$ Since $G$ is a $\explet$-expander, we have
    \[ e_G(U, V(G) \setminus U) \ge \explet (d - d') |U|. \]
    On the other hand,
    \[ e_G(U, V(G) \setminus U) \le |F| + (d + d') |W|. \]
    Combining, we have
    \[ |W| \ge \frac{1}{d+d'} \left( \explet |U|(d-d') - |F| \right) \ge \frac{4}{5d} \left( \frac{3}{4} \explet |U| d - \explet |U| d / 4 \right) = \frac{2}{5}\explet|U|. \]
\end{proof}

The main result in this section is as follows.

\begin{lemma} \label{lem:randomsetisnice}
    Let $G$ be an $n$-vertex $(\explet,s)$-expander with minimum degree $\delta$ and maximum degree $\Delta$, where $\explet<\frac{1}{(\log n)^2}$ and $s\ge p^{-4}\explet^{-5}(\log n)^{21}$ for some $\frac{100\log n}{\delta}\leq p\leq 1$. Let $V$ be a $p$-random subset of $V(G)$. Then, with probability $1-o(1)$, $V$ is $(D,\ell)$-connecting for $D=\frac{p^9 \explet^{12} s \delta}{\Delta (\log n)^{50}}$ and $\ell=\explet^{-1}(\log n)^4$.
\end{lemma}

Combining Lemmas \ref{lem:notions of expander} and \ref{lem:randomsetisnice}, Lemma \ref{lem:randomsetconnecting} follows readily.

\subsection{Proof of Lemma~\ref{lem:randomsetisnice}}

Given $U, V \subset V (G)$, the ball of radius $i$ around $U$ within $V$, denoted by $B_G^i(U, V )$, is the set of vertices in $V$ that can be reached by a path of length at most $i$ starting from a vertex in $U$ which has all of its internal vertices in $V$. The starting vertex in $U$ is not required to be in $V$ itself. We do, however, only consider reachable vertices within $V$, so that $B_G^i(U, V) \subset V$.

It is convenient to use the following definition.

\begin{defn}[$(\mu,\ell)$-reachable set]\label{defn:reachable}
    Let $G$ be an $n$-vertex graph. We say that a set $V\subset V(G)$ is $(\mu,\ell)$-reachable if, for every $U\subset V(G)$ and every $F\subset E(G)$ with $|F|\leq \mu |U|$,
    $$|B_{G-F}^{\ell} (U,V)|>\frac{|V|}{2}.$$
\end{defn}

Most of the work in this section will go into proving the following lemma.

\begin{lemma} \label{lem:reachable}

Let $0 < p < 1$. Suppose that $G$ is an $n$-vertex $(\explet, s)$-expander with $\explet<\frac{1}{(\log n)^2}$ and $s\ge p^{-4}\explet^{-5}(\log n)^{20}$. Let $V$ be a $p$-random subset of $V(G)$.
Then, with probability $1-o(1)$, $V$ is $(\mu,\ell)$-reachable for $\mu=\frac{p^8\explet^{10}s}{(\log n)^{40}}$ and $\ell=8\explet^{-1}(\log n)^2$.
\end{lemma}

In the rest of this subsection, we will complete the proof of Lemma~\ref{lem:randomsetisnice} assuming Lemma~\ref{lem:reachable}. Subsection \ref{sec:reachable} is then devoted to proving Lemma~\ref{lem:reachable}.

We will need the following, which follows from\footnote{Note that in \cite{bucic2022towards}, logarithms have base $2$.} \cite[Proposition 8]{bucic2022towards}.

\begin{proposition} \label{prop:connect one of t}
    Let $1\leq \ell,t\leq n$. Let $G$ be an $n$-vertex graph and let $V\subset V(G)$ with $|V|\geq 4t-2$ be such that, for every $U\subset V(G)$ with size $|U|=t$, we have $|B^{\ell}_G(U,V)|>\frac{|V|}{2}$. Let $z_1,\dots,z_{2t-1},w_1,\dots,w_{2t-1}$ be distinct vertices of $G$. Then, for some $j\in [2t-1]$, there is a $z_j-w_j$ path in $G$ with internal vertices in $V$ and with length at most $8\ell \log n$.
\end{proposition}

We will also use the following form of the Aharoni-Haxell hypergraph matching theorem (see Corollary~1.2 in \cite{aharoni-haxell}).

\begin{theorem}\label{thm:hyperhall}
Let $r\in \mathbb{N}$, and let $H_1,\ldots,H_r$ be a collection of hypergraphs with at most $\ell$ vertices in each edge. Suppose that, for each $I\subset [r]$, there is a matching in $\bigcup_{i\in I}H_i$ containing more than $\ell(|I|-1)$ edges. Then, there is an injective function $f:[r]\to \bigcup_{i\in [r]}E(H_i)$ such that $f(i)\in E(H_i)$ for each $i\in [r]$ and $\{f(i):i\in [r]\}$ is a matching of $r$ edges.
\end{theorem}

\begin{lemma} \label{lem:connectafterexpand}
    Let $n$ be sufficiently large, let $G$ be a graph with maximum degree $\Delta$ and let $V\subset V(G)$ be a $(\mu,\ell)$-reachable subset.
    Let $x_1,\dots, x_{r}, y_1, \dots, y_{r}$ be a sequence of (not necessarily distinct) vertices outside of $V$. Let $W\subset V(G)\setminus (V\cup \{x_1,\dots , x_{r}, y_1, \dots , y_{r}\})$ be such that $|W|\leq |V|$ and assume that there exists some $\delta_0$ such that each $x_i$ and $y_i$ sends at least $\delta_0$ edges to $W$ and each vertex in $W$ has at most $\frac{\mu \delta_0}{2\Delta \ell^2 (\log n)^{4}}$ neighbours in the multiset $\{x_1,\dots,x_{r},y_1,\dots,y_{r}\}$.
    
    Then, there is a collection of internally vertex-disjoint $x_i-y_i$ paths (one for each $i\in [r]$) of length at most $8\ell\log n$ with internal vertices in $W\cup V$.
\end{lemma}

\begin{proof}
    We first prove the following claim using the degree conditions.
    \begin{claim*}
        There exist pairwise disjoint sets $X_1,\dots,X_{r},Y_1,\dots,Y_{r}\subset W$ of size at least $\mu^{-1}2\Delta \ell^2 (\log n)^{4}$ such that $X_i\subset N_G(x_i)$ and $Y_i\subset N_G(y_i)$ for all $i\in [r]$.
    \end{claim*}

    \begin{proof}[Proof of Claim]
        Let $q=\mu^{-1} 2 \Delta \ell^2(\log n)^{4}$. Note that, as $V$ is $(\mu,\ell)$-reachable, we have $\mu\leq \Delta$, so $q\geq 2\ell^2(\log n)^{4}$. Let us assume for simplicity that $q$ is an integer. Define a bipartite graph $H$ with parts $A:=\{x_i^j: i \in [r],j\in [q]\}\cup \{y_i^j: i \in [r],j\in [q]\}$ and $W$, where the vertices $x_i^j$ and $y_i^j$ are all distinct, and in which there is an edge between $x_i^j$ and some $w\in W$ if and only if $x_i w\in E(G)$, and similarly there is an edge between $y_i^j$ and some $w\in W$ if and only if $y_i w\in E(G)$. By the assumption on the degrees in $G$, each $u\in A$ has degree at least $\delta_0$ in $H$, and each $w\in W$ has degree at most $q\cdot \frac{\mu \delta_0}{2\Delta \ell^2(\log n)^{4}}=\delta_0$ in $H$. Hence, by Hall's theorem, there exists an injective map $f:A\rightarrow W$ such that, for each $u\in A$, there is an edge in $H$ between $u$ and $f(u)$. For each $i\in [r]$, let $X_i=f(\{x_i^j: j\in [q]\})$ and $Y_i=f(\{y_i^j:j\in [q]\})$. It is straightforward to verify that these sets satisfy the conditions in the claim.
    \end{proof}
    
    For each $i\in [r]$, let $H_i$ be the hypergraph on vertex set $V$ in which a hyperedge corresponds to the set of internal vertices of a path of length at most $8\ell\log n$ between $X_i$ and $Y_i$ with all internal vertices in $V$. Note that in order to prove the lemma, it suffices to find a matching of $r$ edges in which the $i$-th edge belongs to $H_i$ for each $i\in [r]$.

    Using Theorem~\ref{thm:hyperhall}, it suffices to verify that, for each $I\subset [r]$, there is a matching in $\cup_{i\in I} H_i$ containing more than $(|I|-1)8\ell\log n$ edges. For this, let $I\subset [r]$ and let $M_I$ be a maximal matching in $\cup_{i\in I} H_i$. We need to show that $|M_I|>(|I|-1)8\ell\log n$. Assume for a contradiction that $|M_I|\leq (|I|-1)8\ell\log n$. Let $S$ be the set of vertices in $G$ used by the edges in $M_I$. Note that $|S|\leq |M_I|8\ell\log n \leq |I|\ell^2(\log n)^{4}$. Let $F$ be the set of edges in $G$ which are incident to at least one vertex in $S$. Then $|F|\leq |S|\Delta\leq |I|\Delta \ell^2(\log n)^{4}$. Let $t=\mu^{-1}|I|\Delta \ell^2 (\log n)^{4}$. Now, if $U\subset V(G)$ and $|U|=t$, then $|F|\leq \mu |U|$, so as $V$ is $(\mu,\ell)$-reachable, we have
    $$|B_{G-F}^{\ell} (U,V)|>\frac{|V|}{2}.$$
    Note also that, by the claim above, we have $|V|\geq |W|\geq 2r\cdot \mu^{-1}2\Delta \ell^2(\log n)^{4}\geq 4t$. Hence, we can apply Proposition~\ref{prop:connect one of t} with $G-F$ in place of $G$ to conclude that if $z_1,\dots,z_{2t-1},w_1,\dots,w_{2t-1}$ are distinct vertices in $G$, then for some $j\in [2t-1]$ there is a $z_j$-$w_j$ path in $G-F$ with internal vertices in $V$ and with length at most $8\ell\log n$. However, since $|X_i|,|Y_i|\geq \mu^{-1}2\Delta \ell^2(\log n)^{4} $, we have $|\cup_{i\in I} X_i|\geq |I|\mu^{-1}2\Delta \ell^2(\log n)^{4} =2t$ and $|\cup_{i\in I} Y_i|\geq 2t$, so there is a path in $G-F$ with internal vertices in $V$ and with length at most $8\ell\log n$ which connects some element of $X_i$ to some element of $Y_i$, for some $i\in I$. Since this path does not use the edges in $F$, it does not have any internal vertex which is in $S$. Hence, the internal vertices of this path are disjoint from the vertices used by $M_I$, contradicting the maximality of $M_I$.
\end{proof}

We are now ready to prove Lemma~\ref{lem:randomsetisnice} (assuming Lemma~\ref{lem:reachable}, which will be proved in the rest of this section).

\begin{proof}[Proof of Lemma~\ref{lem:randomsetisnice}]
    Let $W$ be a random subset of $V$ obtained by including each vertex of $V$ independently at random with probability $1/3$. Let $V'=V\setminus W$. Note that $V'$ is a $(2p/3)$-random subset of $V(G)$, so by Lemma~\ref{lem:reachable}, $V'$ is $(\mu,\ell_0)$-reachable with probability $1-o(1)$, where $\mu=\frac{(2p/3)^8\explet^{10}s}{(\log n)^{40}}$ and $\ell_0=8\explet^{-1}(\log n)^2$. Let $v$ be an arbitrary vertex in $G$. Since $W$ is a $(p/3)$-random subset of $V(G)$, the expected number of neighbours of $v$ in $W$ is at least $\delta p/3$. Hence, by the lower bound on $p$ and by the Chernoff bound, the probability that $v$ has fewer than $\delta p/6$ neighbours in $W$ is $o(1/n)$. Hence, with probability $1-o(1)$, every vertex in $G$ has at least $\delta p/6$ neighbours in $W$. Also, the probability that $|V'|\geq |W|$ is $1-o(1)$.
    
    We now show that, if $V'$ is $(\mu,\ell_0)$-reachable, each $v\in V(G)$ has at least $\delta p/6$ neighbours in $W$ and $|V'|\geq |W|$, then $V$ is $(D,\ell)$-connecting for $D = \frac{p^9 \gamma^{12} s \delta}{\Delta (\log n)^{50}}$ and $\ell=\explet^{-1}(\log n)^4$.
    
    Let $x_1,\dots,x_r,y_1,\dots,y_r$ be a sequence of vertices outside of $V$ and suppose that every $v\in V$ has at most $D$ neighbours in the multiset $\{x_1,\dots,x_r,y_1,\dots,y_r\}$.
    Let $\delta_0=\delta p/6$. Now note that each $x_i$ and $y_i$ sends at least $\delta_0$ edges to $W$ and each vertex in $W$ has at most $D\leq \frac{\mu \delta_0}{2\Delta \ell_0^2 (\log n)^{4}}$ neighbours in the multiset $\{x_1,\dots,x_r,y_1,\dots,y_r\}$. Hence, by Lemma~\ref{lem:connectafterexpand} (applied with $V'$ in place of $V$ and with $\ell_0$ in place of $\ell$), there is a set of internally vertex-disjoint $x_i-y_i$ paths (one for each $i\in [r]$) of length at most $8\ell_0\log n\leq \ell$ with internal vertices in $W\cup V'=V$. Thus, $V$ is indeed $(D,\ell)$-connecting, completing the proof. 
\end{proof}

It remains to prove Lemma~\ref{lem:reachable}.

\subsection{The proof of Lemma \ref{lem:reachable}} \label{sec:reachable}

For a graph $G$, it will be convenient to define the `robust neighbourhood' of a set $U \subset V(G)$ for any parameter $d$ as
$N_{G,d}(U) \coloneqq \{v \in V(G) \setminus U : |N_G(v) \cap U| \ge d \}$,
i.e., the set of vertices in $G$, outside of the set $U$, which have degree at least $d$ in $U$.


We now recall three lemmas from \cite{chakraborti2025edge}.
The first lemma follows from\footnote{Note that in \cite{chakraborti2025edge}, logarithms have base $2$.} {\cite[Lemma 7.7]{chakraborti2025edge}}.

\begin{lemma} \label{lem:structuredichotomy}
    There is an $n_0$ such that the following holds whenever $n\geq n_0$, $0< \explet< \frac{1}{(\log n)^2}$, $r\geq 4(\log n)^2$, $t\geq 4(\log n)^2$ and $s\geq 20rt$. Let $G$ be an $n$-vertex $(\explet, s)$-expander, let $U \subset V (G)$ have size $|U| \leq 2n/3$ and let $F$ be a set of at most $s|U|/4$ edges. Then, in $G - F$ we can find either
    \begin{enumerate}[label=\alph*)]
        \item $\frac{|U|}{10r}$ vertex-disjoint stars, each with $t$ leaves, centre in $U$ and all leaves in $V(G)\setminus U$, or \label{prop:large stars}
        \item a bipartite subgraph $H$ with vertex classes $U$ and $X\subset V(G)\setminus U$ such that \label{prop:robust nhood}
        \begin{itemize}
            \item $|X|\geq \explet |U|/8$ and
            \item every vertex in $X$ has degree at least $r$ in $H$ and every vertex in $U$ has degree at most $2t$ in $H$.
        \end{itemize}
    \end{enumerate}
    
\end{lemma}

The next lemma follows from {\cite[Lemma 7.8]{chakraborti2025edge}}.

\begin{lemma}
\label{lem:findingwell-expandingset}
    Let $n\geq 2$, $0<\explet<\frac{1}{(\log n)^2}$ and $s\geq \mu\geq 1$. Let $G$ be an $n$-vertex $(\explet,s)$-expander and let $U\subset V(G)$ have size $|U|\leq 2n/3$.
    Then, there is a set $U'\subset U$ with $|N_G(U')|\geq |U'|\mu$ and $|U'|\geq \frac{\explet |U|}{12\mu}$.
\end{lemma}

The next lemma follows from {\cite[Lemma 7.9]{chakraborti2025edge}}.

\begin{lemma}\label{lem:partitionedgesintoexpanders} Let $n$ and $s$ be sufficiently large (i.e.\ greater than some absolute constant), $k\in \mathbb{N}$ and $0<\explet < \frac{1}{(\log n)^2}$.
Suppose that $G$ is an $n$-vertex $(\explet,s)$-expander and $\frac{\explet s}{k} \geq 10^6 \log n$. Then, there are edge-disjoint graphs $G_1,\ldots,G_k$ such that $E(G)=\bigcup_{i\in [k]}E(G_i)$ and, for each $i\in [k]$, $G_i$ is an $\left(\frac{\explet}8,\frac{\explet s}{10^5k}\right)$-expander  with vertex set $V(G)$.
\end{lemma}

In order to prove Lemma~\ref{lem:reachable}, we need to show that, for an expander $G$, if we take a large random subset $V$ in $V(G)$, then with high probability it is true that, for every $U\subset V(G)$ and not too large $F\subset E(G)$, more than half of the vertices in $V$ can be reached from $U$ by short paths inside $V$ which avoid all the edges in $F$. We prove this in three steps: first we deal only with `well-expanding' sets $U$ and very small $F$, then we extend this to arbitrary $U$ but still only very small $F$, and finally, we deal with all $U$ and much larger $F$, completing the proof of Lemma~\ref{lem:reachable}.

\begin{lemma}
\label{lem:wellexpandingsetscanreach}
    Let $n$ be sufficiently large, let $0 < p < 1$, and suppose that $G$ is an $n$-vertex $(\explet, s)$-expander with $\explet<\frac{1}{(\log n)^2}$ and $s \geq 20p^{-3}\explet^{-2}(\log n)^{9}$. Let $\ell=\explet^{-1}(\log n)^2$. Let $U \subset V(G)$ satisfy $|N_G(U)| \geq |U|p^{-4}\explet^{-4}(\log n)^{16}$ and let $F \subset E(G)$ satisfy $|F| \leq |U|$. Let $V$ be a $p$-random subset of $V(G)$.
    Then, with probability $1-e^{-\Omega(|U|(\log n)^2)}$,
    $$|B_{G-F}^{\ell} (U,V)|>\frac{|V|}{2}.$$
\end{lemma}

\begin{proof}
Let $q\in (0,1)$ be such that $1-(1-q)^{\ell-1}(1-\frac{4p}{5})=p$, i.e., that $(1-q)^{\ell-1}=\frac{1-p}{1-\frac{4p}{5}}$, so that
\begin{equation}\label{eqn:p15}
q\ge \frac{p}{6\ell}= \frac{p\explet}{6(\log n)^2}.
\end{equation}
Independently, for each  $i\in [\ell]$, let $V_i$ be a $q$-random subset of $V(G)$ if $i\leq \ell-1$ and a $(4p/5)$-random subset of $V(G)$ if $i=\ell$. Set $V=V_1\cup\ldots \cup V_\ell$, and note that $V$ is a $p$-random subset of $V(G)$. Thus, we wish to show that, with probability at least  $1-e^{-\Omega\left({|U|}{(\log n)^{2}}\right)}$ we have $|B^{\ell}_{G-F}(U , V)|>\frac{|V|}{2}$.

For each $0\le i\leq \ell$, let $B_{i}$ be the set of vertices of $G$ which can be reached via a path in $G-F$ which starts in $U$ and has length at most $i$ and whose internal vertices (if there are any) are in $V_1\cup \dots \cup V_{i-1}$. In particular, we have $B_0=U$ and $B_1=U\cup N_{G-F}(U)$. 
Observe also that $B_0\subset B_1\subset \ldots \subset B_{\ell}$. We emphasise that the vertices of $B_{i}$ do not themselves have to be inside $V_1\cup \ldots \cup V_{i-1}$, only the internal vertices of some path from $U$ to the vertex in $B_{i}$ are required to be inside $V_1\cup \ldots \cup V_{i-1}$. An important property of $B_{i}$ is that it is completely determined by the sets $U,V_1,\ldots, V_{i-1}$, and therefore is independent of $V_{i}$. Note also that any vertex in $N_{G-F}(B_i)$ with a neighbour in $B_i$ that gets sampled into $V_i$ belongs to $B_{i+1}$. These two observations will be the key behind why the sets $B_{i+1}$ will grow in size until they occupy most of the set $V(G)$. The lemma will then follow from
\begin{equation}\label{eqn:overkill4}
B_\ell\cap V_\ell\subset B^{\ell}_{G-F}(U, V).
\end{equation}

We now show that indeed, for each $1\le i\le \ell-1$, unless $B_{i}$ is already very large, $B_{i+1}$ is likely to be somewhat larger than $B_i$.

\begin{claim*} For each $1\le i\le \ell-1$, with probability $1-e^{-\Omega\left({|U|}{(\log n)^{2}}\right)}$, either $|B_i| \ge \frac23 n$, or  $$|B_{i+1}\setminus B_i| \ge \frac{\explet |B_i|}{2^7 }.$$
\end{claim*}
\begin{proof}[Proof of Claim] For each $v\in N_{G-F}(B_i)$, $v$ is in $B_{i+1}$ if at least one of its neighbours in $G-F$ in $B_i$ gets sampled into $V_{i}$. That is,
\begin{equation}\label{eqn:overkill1}
\{v\in N_{G-F}(B_i):(N_{G-F}(v)\cap B_i)\cap V_i\neq\emptyset\} \subset B_{i+1}\setminus B_i. 
\end{equation}
We will show that, for any set $W\subset V(G)$ with $|W|\leq \frac23 n$ and $B_1\subset W$ 
\begin{equation}\label{eqn:overkill2}
\Pr\left(|\{v\in N_{G-F}(W):(N_{G-F}(v)\cap W)\cap V_i\neq\emptyset\}|\geq \frac{\explet |W|}{2^7}\right)\ge 1-e^{-\Omega\left({p^4 \explet^4|B_1|}/{(\log n)^{14} }\right)}.
\end{equation}
Given~\eqref{eqn:overkill2}, we will have that for all $1\le i \le \ell -1$,
\begin{align*}
\Pr\left(|B_i| \ge \frac23 n  \: \text{ or  } \: |B_{i+1}\setminus B_i| \ge \frac{\explet|B_i|}{2^7}\right) &\overset{\textcolor{white}{\eqref{eqn:overkill1}}}{\geq} \Pr\left(|B_{i+1}\setminus B_i| \ge \frac{\explet|B_i|}{2^7}\: \Big| \: |B_i|\leq \frac{2}{3}n\right)
\\
&\overset{\eqref{eqn:overkill1}}{\geq} 
\Pr\left(|\{v\in N_{G-F}(B_i):(N_{G-F}(v)\cap B_i)\cap V_i\neq\emptyset\}|\ge \frac{\explet|B_i|}{2^7}\: \Big| \: |B_i|\leq \frac{2}{3}n\right)\\
&\overset{\eqref{eqn:overkill2}}{\ge} 
1-e^{-\Omega\left({p^4 \explet^4|B_1|}/{(\log n)^{14} }\right)}\ge 1-e^{-\Omega\left(|U| (\log n)^{2} \right)},
\end{align*}
where in the last inequality we used that $|B_1|\geq |N_{G-F}(U)|\ge |U|p^{-4} \explet^{-4}(\log n)^{16}-|F|\geq \frac12|U|p^{-4}\explet^{-4}(\log n)^{16}$.

Let then $W\subset V(G)$ satisfy $|W|\leq \frac23n$ and $B_1\subset W$. As $|W|\leq \frac23 n$ and $|F|\le |U|\leq |B_1|\leq |W|\le {s|W|}/4$, we can apply Lemma~\ref{lem:structuredichotomy} with $W$ in place of $U$, $r=p^{-1}\explet^{-1}(\log n)^2$ and $t=p^{-2}\explet^{-1}(\log n)^7$ (note that the lemma applies since $s \ge 20 p^{-3} \explet^{-2} (\log n)^{9} = 20 rt$). Hence, one of the two cases \ref{prop:large stars} or \ref{prop:robust nhood} from Lemma~\ref{lem:structuredichotomy} holds; we will show that \eqref{eqn:overkill2} holds in either case.

a) Suppose that $G-F$ contains $\frac{|W|}{10p^{-1}\explet^{-1}(\log n)^{2}}$ vertex-disjoint stars, each with $p^{-2}\explet^{-1}(\log n)^7$ leaves, centre in $W$ and all leaves in $N_{G-F}(W)$. Let $C\subset W$ be the set of centres of such a collection of stars, and note that
\begin{equation}\label{eqn:overkill3}
|\{v\in N_{G-F}(W):(N_{G-F}(v)\cap W)\cap V_i\neq \emptyset\}|\geq |C\cap V_i|\cdot p^{-2}\explet^{-1}(\log n)^7.
\end{equation}
By the Chernoff bound and \eqref{eqn:p15}, and as $p^2\explet^2/(\log n)^{10}\leq 1$ and $|W|\geq |B_1|$, with probability  at least 
$1-e^{-q|C|/12}= 1-e^{-\Omega\left({|W|p^2 \explet^2}/{(\log n)^{4}}\right)} \geq 1-e^{-\Omega\left({p^4 \explet^4|B_1|}/{(\log n)^{14}}\right)}$, we have $|C\cap V_i|\geq \frac{q|C|}2\geq \frac{|W|p^2 \explet^2}{120(\log n)^{4}}$. This, combined with \eqref{eqn:overkill3}, implies \eqref{eqn:overkill2}.

b) Suppose instead that there is a bipartite subgraph $H\subset G-F$ with vertex classes $W$ and $X\subset V(G)\setminus W$ such that
    \begin{itemize}
        \item $|X| \ge \frac{\explet|W|}{8}$ and
        \item every vertex in $X$ has degree at least $r=p^{-1}\explet^{-1}(\log n)^{2} $ in $H$ and every vertex in $W$ has degree at most $\Delta:=2t=2p^{-2}\explet^{-1}(\log n)^{7}$ in~$H$.
    \end{itemize}  
    For each $v\in X$, the probability that $v$ has no neighbours in $H$ in $V_i$ is at most
    \[
(1-q)^r= (1-q)^{p^{-1}\explet^{-1}(\log n)^2}\le e^{-qp^{-1}\explet^{-1}(\log n)^2}\overset{\eqref{eqn:p15}}{\leq} e^{-1/6}\leq \frac{7}{8}.
    \]
Let $Y$ be the random variable counting the number of vertices in $X$ having a neighbour in $V_{i}$ in $H$, so that $\E [Y] \ge \frac{|X|}{8}$. Observe also that $Y$ is $\Delta$-Lipschitz since for each $v\in W$ the event $\{v\in V_i\}$ affects $Y$ by at most $d_H(v)\leq \Delta$. 
Hence, by McDiarmid's inequality (see, e.g., Lemma 1.2 in \cite{mcdiarmid1989method}), we have
$$\Pr\left(Y < \frac{|X|}{16}\right)\le \Pr\left(Y < \E [Y] - \frac{|X|}{16} \right)\le2\exp\left(-\frac{2^{-9}|X|^2}{\Delta^2|W|}\right)= e^{-\Omega\left({|W|p^4 \explet^4}/{(\log n)^{14}}\right)}. $$
Each vertex in $X$ with a neighbour in $V_i$ in $H$ belongs to $\{v\in N_{G-F}(W):(N_{G-F}(v)\cap W)\cap V_i\neq\emptyset\}$. Hence, with probability at least $1-e^{-\Omega\left({|W|p^4 \explet^4}/{(\log n)^{14}}\right)}$, we have $|\{v\in N_{G-F}(W):(N_{G-F}(v)\cap W)\cap V_i\neq\emptyset\}|\geq Y\geq \frac{|X|}{16}\geq \frac{\explet|W|}{2^7}$. This means that \eqref{eqn:overkill2} holds in case b) as well, completing the proof of the claim.
\end{proof}

As $B_\ell$ and $V_\ell$ are independent, by the Chernoff bound, we have that
\[
\Pr\left(|B_{\ell}\cap V_\ell|\le\frac{31p}{60}n \: \big| \: |B_{\ell}|\geq \frac{2}3n\right)\leq \Pr\left(\Bin\left(\frac{2}3n,\frac{4p}5\right)\leq \frac{31p}{60}n\right)\le e^{-\Theta(np)},
\]
and, similarly, we have $\Pr\left(|V|\geq \frac{61p}{60}n\right)\le e^{-\Theta(np)}$.

Thus, by the claim, altogether we have that

\begin{enumerate}[label=\roman*)]
\item \label{itm1}
for each $i\in [\ell-1]$, $|B_i| \ge \frac23 n$ or  $|B_{i+1}\setminus B_i| \ge \frac{\explet|B_i|}{2^7}$, 

 \item $|B_\ell|<\frac23 n$ or $|B_\ell\cap V_\ell|>\frac{31p}{60}n$, and

\item \label{itm3} $|V|\leq \frac{61p}{60}n$
\end{enumerate}

with probability at least
$$1-\ell\cdot e^{-\Omega\left({|U|}{(\log n)^{2}}\right)}-e^{-\Theta(np)}\ge 1-e^{-\Omega\left({|U|}{(\log n)^{2} }\right)},$$ where we used that $|U|\le \frac{np}{(\log n)^{2}}$ as $|N_G(U)|\geq |U|p^{-4}\explet^{-4}(\log n)^{16}$.

However, if \ref{itm1}--\ref{itm3} all hold, then, for each $i\in [\ell]$, we have
$$|B_i| \ge \min\left\{\frac23 n,\left(1+\frac{\explet}{2^7 }\right)^{i}|U|\right\}\geq \min\left\{\frac23 n,\exp\left(\frac{\explet i}{2^8}\right)\right\},$$
so that, setting $i=\ell=\explet^{-1}(\log n)^2$, we obtain that $|B_{\ell}|\geq \frac23 n$, and hence, by ii) and iii), that $|B_\ell\cap V_\ell|>\frac{|V|}2$.

Thus, by \eqref{eqn:overkill4}, we have that $|B^{\ell}_{G-F}(U,V)|>\frac{|V|}2$ with probability at least $1-e^{-\Omega\left({|U|}{(\log n)^{2}}\right)}$.
\end{proof}

We now use Lemma~\ref{lem:findingwell-expandingset} and Lemma~\ref{lem:wellexpandingsetscanreach} to prove the following result.   
\begin{lemma} \label{lem:reachhalfwithfewedgesblocked}
    Let $0 < p < 1$. Suppose that $G$ is an $n$-vertex $(\explet, s)$-expander with $\explet<\frac{1}{(\log n)^2}$ and $s \geq 20 p^{-4}\explet^{-4}(\log n)^{16}$. Let $\ell=\explet^{-1}(\log n)^2$. Let $V$ be a $p$-random subset of $V(G)$.
    Then, with probability $1-o(1/n)$, for every $U\subset V(G)$ and every $F\subset E(G)$ with $|F|\leq \frac{p^4\explet^5|U|}{(\log n)^{17}}$,
    \begin{equation}
    \label{eqn:B1UV}
    |B_{G-F}^{\ell} (U,V)|>\frac{|V|}{2}.    
    \end{equation}
\end{lemma}
\begin{proof} 
Say that a set $U'\subset V(G)$ is \emph{well-expanding} in $G$ if 
$|N_{G}(U')| \ge |U'|p^{-4}\explet^{-4}(\log n)^{16}.$ Since $s \geq 20p^{-3}\explet^{-2}(\log n)^{9}$, given a non-empty well-expanding set $U'\subset V(G)$ and a set of edges $F$ of size
at most $|U'|$, Lemma~\ref{lem:wellexpandingsetscanreach} applied to $U'$ implies that 
\begin{equation}\label{eq:1}
    |B_{G-F}^{\ell}(U', V)|> \frac{|V|}2
\end{equation}
fails with probability at most $e^{-\Omega(|U'|(\log n)^2)}.$

Now, a union bound over all pairs $(U',F)$ such that $U'$ is a well-expanding set in $G$ and $F$ is a set of at most $|U'|$ edges tells us that \emph{some} such pair $(U',F)$ fails \eqref{eq:1} with probability at most 
\begin{align*}
    \sum_{(U',F)}  e^{-\Omega(|U'|(\log n)^2)} &\le 
    \sum_{u=1}^n\sum_{f=1}^{u} \binom{n}{u}\binom{n^2}{f} \cdot e^{-\Omega(u(\log n)^2)} \\
&    \le \sum_{u=1}^n u\cdot n^{3u}\cdot e^{-\Omega(u(\log n)^2)} \le \sum_{u=1}^n e^{-\Omega(u(\log n)^2)}=o(1/n).
\end{align*}

Thus, with probability $1-o(1/n)$, we can assume that \eqref{eq:1} holds for every well-expanding set $U'$ and set $F\subset E(G)$ with $|F|\leq |U'|$. We will now show that this implies that \eqref{eqn:B1UV} holds for all $U\subset V(G)$ and $F\subset E(G)$ with $|F|\leq \frac{p^4\explet^5|U|}{(\log n)^{17}}$, completing the proof.

Let $U\subset V(G)$ with $|U|\leq \frac23n$ and let $F\subset E(G)$ satisfy the (slightly weaker) condition $|F|\leq \frac{2p^4\explet^5|U|}{(\log n)^{17}}$. Then, applying Lemma~\ref{lem:findingwell-expandingset} (with $\mu = p^{-4}\explet^{-4}(\log n)^{16} \le s$), there is a set $U'\subset U$ which is well-expanding for which $|U'|\geq \frac{p^4 \explet^5 |U|}{12(\log n)^{16}}$. Noting that $|F|\leq |U'|$ (as we may assume that $n$ is large), we therefore have that
\[
|B_{G-F}^{\ell}(U,V)|\geq |B_{G-F}^{\ell}(U',V)|> \frac{|V|}{2}.
\]
Finally, consider $U\subset V(G)$ with $|U|> \frac23n$ and let $F\subset E(G)$ satisfy $|F|\leq \frac{p^4\explet^5|U|}{(\log n)^{17}}$. Let $\bar{U}\subset U$ be an arbitrary subset with $\frac{n}2\leq |\bar{U}|\leq \frac23 n$, so that we have $|F|\leq \frac{2p^4 \explet^5 |\bar{U}|}{(\log n)^{17}}$, and hence, from what we have just shown,
\[
|B_{G-F}^{\ell}(U,V)|\geq |B_{G-F}^{\ell}(\bar{U},V)|> \frac{|V|}{2},
\]
as required.
\end{proof}

We are now ready to prove Lemma~\ref{lem:reachable} in the following (equivalent) form which incorporates Definition~\ref{defn:reachable}.

\begin{lemma} \label{lem:reachhalfwithmanyedgesblocked}

Let $0 < p < 1$. Suppose that $G$ is an $n$-vertex $(\explet, s)$-expander with $\explet<\frac{1}{(\log n)^2}$ and $s\ge p^{-4}\explet^{-5}(\log n)^{20}$. Let $\ell=8\explet^{-1}(\log n)^2$. Let $V$ be a $p$-random subset of $V(G)$.
Then, with probability $1-o(1)$, for every $U\subset V(G)$ and every $F\subset E(G)$ with $|F|\leq \frac{p^8\explet^{10}|U|s}{(\log n)^{40}}$,
    $$|B_{G-F}^{\ell} (U,V)|>\frac{|V|}{2}.$$
\end{lemma}

\begin{proof}
    Let $k= \lfloor p^4 \explet^5 s/(\log n)^{20}  \rfloor$. Then by the assumption on $s$, we have $k \ge 1$. Also, as our lemma concerns the asymptotic probability as $n\rightarrow \infty$, we may assume that $n$ is sufficiently large and therefore $\explet s/k \ge 10^6 \log n$. Hence, using Lemma \ref{lem:partitionedgesintoexpanders}, we can obtain edge-disjoint graphs $G_1,\dots,G_k$ such that $E(G)=\cup_{i\in [k]} E(G_i)$ and, for each $i\in [k]$, $G_i$ is an $(\frac{\explet}{8},s')$-expander with vertex set $V(G)$, where $s'=\frac{\explet s}{10^5 k}$. By the choice of $k$, we have $s'\geq 20 p^{-4}(\explet/8)^{-4}(\log n)^{16}$, so by Lemma \ref{lem:reachhalfwithfewedgesblocked} and a union bound over the $k$ graphs $G_i$, with probability $1-o(1)$ we have that, for each $i\in [k]$ and every $U\subset V(G_i)$ and $F\subset E(G_i)$ with $|F|\leq \frac{p^4 (\explet/8)^5 |U|}{(\log n)^{17}}$,
    $$|B_{G_i-F}^{\ell} (U,V)|>\frac{|V|}{2}.$$

    Now let $U\subset V(G)$ and $F\subset E(G)$ with $|F|\leq \frac{p^8\explet^{10}|U|s}{(\log n)^{40}}$. As the graphs $G_i$, $i\in [k]$, are edge-disjoint, there exists some $i\in [k]$ such that $|F\cap E(G_i)|\leq \frac{p^8\explet^{10}|U|s}{k(\log n)^{40}} \le \frac{p^4(\explet/8)^5|U|}{(\log n)^{17}}$, and therefore
    $$|B_{G-F}^{\ell} (U,V)|\geq |B_{G_i-(F\cap E(G_i))}^{\ell} (U,V)|>\frac{|V|}{2}.$$
    This proves the lemma.
\end{proof}

\end{document}